\documentclass{amsart}

\newtheorem{theorem}{Theorem}[section]
\newtheorem{proposition}[theorem]{Proposition}
\newtheorem{lemma}[theorem]{Lemma}

\theoremstyle{definition}

\theoremstyle{remark}
\newtheorem{remark}[theorem]{Remark}

\numberwithin{equation}{section}

\begin{document}

\title{Unsteady non-Newtonian fluid flows with boundary conditions of friction type: the case of shear thinning fluids.}


\author{Mahdi Boukrouche}
\address{Lyon University, F-42023 Saint-Etienne, Institut Camille Jordan CNRS UMR 5208, 23 Rue Paul Michelon, 42023 Saint-Etienne Cedex 2, France}
\email{mahdi.boukrouche@univ-st-etienne.fr}

\author{Hanene Debbiche}
\address{Lyon University, F-42023 Saint-Etienne, Institut Camille Jordan CNRS UMR 5208, 23 Rue Paul Michelon, 42023 Saint-Etienne Cedex 2, France}
\curraddr{University of Bordj bou Arr\'eridj, Algeria}
\email{hanane.debbiche@gmail.com}

\author{Laetitia Paoli}
\address{Lyon University, F-42023 Saint-Etienne, Institut Camille Jordan CNRS UMR 5208, 23 Rue Paul Michelon, 42023 Saint-Etienne Cedex 2, France}
\email{laetitia.paoli@univ-st-etienne.fr}





\begin{abstract}
Following the previous part of our study on unsteady non-New\-to\-nian fluid flows with boundary conditions of friction type we consider in this paper the case of pseudo-plastic (shear thinning) fluids. The problem is  described by a $p$-Laplacian non-stationary Stokes system with $p<2$ and we assume that the fluid is subjected to mixed boundary conditions, namely non-homogeneous Dirichlet boundary conditions on a part of the boundary and a slip fluid-solid interface law of friction type on another part of the boundary. Hence the fluid velocity should belong to a subspace of $L^p \bigl(0,T; (W^{1,p} (\Omega)^3) \bigr)$, where $\Omega$ is the flow domain and $T>0$, and satisfy a non-linear parabolic variational inequality. In order to solve this problem we introduce first a vanishing viscosity technique which allows us to consider an auxiliary problem formulated in $L^{p'} \bigl(0,T; (W^{1,p'} (\Omega)^3) \bigr)$ with $p' >2$ the conjugate number of $p$ and to use the existence results already established in \cite{BDP2}. Then we apply both compactness arguments and a fixed point method to prove the existence of a solution to our original fluid flow problem.
\end{abstract}

\maketitle



%
%


\section{Description of the problem} \label{description}
\bigskip

Following the previous part of our study on unsteady general incompressible fluid flows with boundary conditions of friction type (\cite{BDP2}), we focus in this paper on the case of pseudo-plastic (or shear thinning) fluids like molten polymers.
More precisely we consider non-Newtonian fluids satisfying the following power law
 \begin{eqnarray} \label{sigma}
 \sigma = 2 \mu \bigl( \theta, \upsilon, \bigl| D(\upsilon) \bigr| \bigr) \bigl|D(\upsilon) \bigr|^{p-2} D(\upsilon) - \pi {\rm Id}_{{\mathbb{R}^3}}
 \end{eqnarray}  
 with $p \in (1,2)$,  where $\mu$ is a given mapping, $\upsilon$ is the fluid velocity, $\pi$ is the pressure, $\theta$ is the temperature,  $\sigma$ is   the stress tensor  and $D(\upsilon)=\bigl(d_{ij}(\upsilon) \bigr)_{1\leq i,j\leq3}$ is the strain rate tensor given by 
\begin{eqnarray*} 
d_{ij}(\upsilon) =\frac{1}{2} \left(\frac{\partial \upsilon_{i}}{\partial x_j} +\frac{\partial \upsilon_{j}}{\partial x_i} \right) \quad 1\leq i,j\leq 3.
\end{eqnarray*}
The problem is then described by the conservation of mass and momentum i.e.
  \begin{eqnarray} \label{stokes}
\qquad \left\{ \begin{array}{ll}
\displaystyle \frac{\partial \upsilon}{\partial t}-2div \bigl( \mu \bigl(\theta , \upsilon, \bigl|D(\upsilon) \bigr| \bigr) \bigl|D(\upsilon) \bigr|^{p-2}D(\upsilon) \bigr) +\nabla \pi=f \quad \mbox{in} \ 
(0,T) \times \Omega  \\
\displaystyle  div(\upsilon)=0\qquad \mbox{in} \   (0,T) \times \Omega 
 \end{array} \right.
\end{eqnarray}
where $f$ represents the vector of external forces,  $(0,T)$ is the time interval ($T>0$) and $\Omega \subset \mathbb{R}^3$ is the fluid flow domain. 
The evolution of the temperature $\theta$ is described by the heat equation which is fully decoupled from the flow problem when the convection term is neglected. Thus $\theta$ appears as  a data in (\ref{stokes}).

\smallskip

Since the 90's several experimental studies (see for instance \cite{H.B, 84, slip, H. HERV, Amir, Soh}) have shown that shear thinning fluids may exhibit complex behaviour at the boundary like non-linear threshold slip phenomena of friction type,  which can be described by 
\begin{eqnarray*}
 \upsilon_n = 0   , \quad \upsilon_{\tau} -  s \in -  \partial  \psi_{\overline{B}_{\mathbb{R}^3}(0,k) } (\sigma_{\tau})
\end{eqnarray*}
where $k$ is a given positive threshold, $\partial \psi_{\overline{B}_{\mathbb{R}^3}(0,k) }$ is the subdifferential 
 of the indicator function of the closed ball $\overline{B}_{\mathbb{R}^3}(0,k) $,  $s$ is the sliding velocity of the wall, 
$\upsilon_n$, $\upsilon_{\tau}$ and $\sigma_{\tau}$ are the normal component of the velocity,  the tangential component of the velocity and the shear  stress respectively.  This kind of boundary condition, reminiscent of Tresca's friction for solids (\cite{h'}), had been introduced by H.Fujita in \cite{F0}, leading to an abundant literature in the case of steady or unsteady Newtonian fluid flows (see for instance \cite{F1, F2, 3, saito1, F5, F6,   Roux2, saito2, Roux1, Roux3, K.T, imane3, imane4}). The case of non-Newtonian fluids satisfying a power law with $p \not=2$ yields a very different mathematical  framework since both the velocity and the pressure are expected to belong to Banach spaces depending on $p$ and a first result is given in \cite{BDP1} for steady flows.

\smallskip


Motivated by lubrication and extrusion / injection applications,  we consider 
\begin{eqnarray*}
 \Omega=\Bigl\{(x' , x_{3}) \in \mathbb{R}^{2}\times \mathbb{R} :  \, x' \in \omega, \,  0< x_{3} < h(x')\Bigr\}
 \end{eqnarray*} 
 where $\omega$ is a non-empty  bounded  domain of $\mathbb{R}^{2}$ with a Lipschitz continuous boundary 
 and $h$ is a Lipschitz continuous function  which is bounded from above and from below by some positive real numbers. 
 
  \smallskip
 
 We denote by $u \cdot v$ the Euclidean inner product of two vectors $u$ and $v$ in $\mathbb{R}^{3}$. We decompose the boundary of $\Omega$ as $\partial \Omega = \Gamma_0 \cup \Gamma_L \cup \Gamma_1$ with 
  \begin{eqnarray*}
 \Gamma_{0} =\Bigl\{(x' , x_{3}) \in \overline{\Omega} :  \, x_{3} =0 \Bigr\}, \qquad 
  \Gamma_{1}=\Bigl\{(x' , x_{3}) \in \overline{\Omega} :  \, x_{3} = h(x') \Bigr\}
 \end{eqnarray*} 
and   $\Gamma_{L}$  the lateral part of $\partial \Omega$. 
  Let $n=(n_{1},n_{2},n_{3}) $  be the unit outward normal vector to   $\partial \Omega$ and 
 $g: \partial \Omega \to \mathbb{R}^{3}$  such that
\begin{eqnarray*}\label{mol}
\begin{array}{ll}
\displaystyle{\int_{\partial\Omega} g\cdot n\,dY= 0,\quad g=0\ \mbox{on}\  \Gamma_{1}, \quad g\neq0 \  \mbox{on}\  \Gamma_{L},\quad g\cdot n=0\  \mbox{on}\  \Gamma_0}.
\end{array}
\end{eqnarray*}
 We define the normal and tangential velocities on $\partial \Omega$ as $\upsilon_{n}=  \upsilon \cdot n$ and  $\upsilon_{\tau} = \upsilon- \upsilon_{n} n $ 
 and the  normal and tangential components  of the stress vector  are given by  
\begin{eqnarray*}
\displaystyle \sigma_n = \sum_{i,j=1}^3 \sigma_{ij} n_{j}n_{i}, \quad
\displaystyle \sigma_{\tau} =\left( \sum_{j=1}^3 \sigma_{ij} n_{j} -\sigma_{n}n_{i} \right)_{1 \le i \le 3}.
\end{eqnarray*} 

 We assume that the fluid is subjected to non-homogeneous   Dirichlet boundary conditions on $\Gamma_{1}\cup\Gamma_{L} $ and  friction boundary conditions on $\Gamma_0$, i.e.   
\begin{eqnarray}\label{7}
\upsilon=0 \ \mbox{on}\  (0,T) \times \Gamma_{1} ,\quad
\upsilon=g\xi \ \mbox{on}\  (0,T) \times \Gamma_{L} 
\end{eqnarray}
where $\xi$ is a function depending only on the time variable such that
\begin{eqnarray}\label{xi}
\xi(0)=1
 \end{eqnarray}
and 
 \begin{eqnarray}\label{8}
\upsilon_n=0\qquad \mbox{on}\  (0,T) \times \Gamma_0  \quad \hbox{\rm (slip condition)}
\end{eqnarray}
while $\upsilon_{\tau}$ is unknown on $\Gamma_0$ and satisfies Tresca's friction law, i.e.
\begin{eqnarray}\label{9}
\left .\begin{array}{l}
|\sigma_{\tau}|=k\Rightarrow\exists\lambda\geq0\quad \upsilon_{\tau}=s-\lambda \sigma_{\tau} \\
|\sigma_{\tau}|<k\Rightarrow \upsilon_{\tau}=s \\ 
\end{array}\right\}\qquad \mbox{on} \  (0,T) \times \Gamma_0 
\end{eqnarray} 
where $s$ is the sliding velocity of the lower part of the boundary and $k$ is the positive friction threshold. 

\smallskip

We complete the description of the problem with the initial condition
\begin{equation}\label{10}
\upsilon(0)=\upsilon_{0}\ \mbox{in}\  \Omega
\end{equation}
and we assume that  $\upsilon_{0}\in \bigl(W^{1,p}(\Omega) \bigr)^3$ such that
\begin{eqnarray}\label{coco}
\begin{array}{ll}
div(\upsilon_{0})=0\ \mbox{in}\ \Omega,\quad \upsilon_{0}=0\ \mbox{on}\ \Gamma_{1},\quad  \upsilon_{0}=g \ \mbox{on}\ \Gamma_{L}, \quad \upsilon_{0}\cdot n=0\ \mbox{on}\  \Gamma_0.
\end{array}
\end{eqnarray}

\bigskip

At a first glance this problem seems very similar to the case $p \ge 2$ already treated in \cite{BDP2}. Unfortunately, when $p\in (1,2)$, we have to deal with additional mathematical difficulties. Indeed we expect  
the fluid velocity to take its values in $\bigl( W^{1,p} (\Omega) \bigr)^3$. Since we consider an evolution problem we need a Gelfand triplet with some pivot space between $\bigl( W^{1,p} (\Omega) \bigr)^3$ and its dual. 
Moreover the proof strategy in \cite{BDP2} relies on a finite difference approximation of the evolution term and monotonicity methods to solve the corresponding  elliptic inequalities in a functional framework of the form ${\mathcal K}^p \subset {\mathcal H} = {\mathcal H}  ' \subset ({\mathcal K}^p)'$ where ${\mathcal K}^p $ and ${\mathcal H}$ are two Hilbert spaces such that the unknown velocity ${\overline{\upsilon}}= \upsilon - \upsilon_0 \xi $ belongs to $ {\mathcal K}^p$ and $ {\mathcal K}^p$ is a subset of $ L^p \bigl( 0,T; ( W^{1,p} (\Omega) )^3 \bigr)$. Of course when $p \in (1,2)$ we can not obtain such embeddings. In order to overcome this  difficulty we will apply a vanishing viscosity technique, namely we will consider a sequence of  approximate problems  where the stress tensor is now given by 
\begin{eqnarray} \label{sigma-epsilon}
 \sigma^{\varepsilon} = \sigma + 2 \varepsilon \bigl|D(\overline{\upsilon}) \bigr|^{p'-2} D(\overline{\upsilon}) 
 \end{eqnarray}  
where $0< \varepsilon <<1$ and $\displaystyle p' = \frac{p}{p-1} >2$ is the conjugate number of $p$.


\bigskip

The paper is organized as follows. In Section \ref{framework} we describe   the mathematical framework and we derive the variational formulation of the problem. In Section \ref{approximate-problems} we introduce an auxiliary flow problem where the two first arguments of the mapping $\mu$ are given data and 
we consider 
the modified constitutive law given by (\ref{sigma-epsilon}).
For any $\varepsilon >0$ the existence results obtained in \cite{BDP2} can be applied leading to a sequence of approximate solutions 
 and we establish a priori estimates.
Then we pass to the limit as $\varepsilon$ tends to zero by using monotonicity properties and we obtain the existence and uniqueness of a solution to the auxiliary flow problem. Finally in Section \ref{fixed_point} we apply a fixed point technique to prove the existence of a solution to our original problem.


\bigskip

 
 \section{Mathematical framework} \label{framework}

 \bigskip
 
 We adopt the same notations as in \cite{BDP2}. 
 Throughout this paper we will denote by  $ \textbf{X} $  the functional space $ X^3 $. 
  For all $p>1$ we introduce  the following  subspaces   of $\mathbf{W}^{1,p}(\Omega)$: 
\begin{eqnarray*}
V^p_{0}=\bigl\{\varphi\in \mathbf{W}^{1,p}(\Omega);\ \  \varphi=0\ \mbox{on}\  \Gamma_{1}\cup\Gamma_{L}\ \mbox{and}\ \varphi\cdot n=0\ \mbox{on}\ \Gamma_0 \bigr\}
\end{eqnarray*}
and 
\begin{eqnarray*}
V^p_{0.div}=\bigl\{\varphi\in V^p_{0} ; \ \ div(\varphi)=0\  \mbox{in}\ \Omega \bigr\}, \quad
V^p_{\Gamma_{1}}=\bigl\{\varphi\in \mathbf{W}^{1,p}(\Omega);\ \ \varphi=0\ \mbox{on}\ \Gamma_{1}\bigr\}
\end{eqnarray*}
endowed with the norm
 \begin{eqnarray*} 
  \Vert\upsilon\Vert_{1.p}=\Big(\int_{\Omega}|\nabla\upsilon|^{p}\,dx\Big)^{\frac{1}{p}}.
  \end{eqnarray*}
  We observe that the mapping $z \mapsto z^p$ is convex on $\mathbb{R}^+_*$, thus
  \begin{eqnarray} \label{kornbis}
  \left( \int_{\Omega} \bigl| D(u) \bigr|^p \, dx \right)^{1/p} = \Vert D(u) \Vert_{(L^p(\Omega))^{3 \times 3}}  \le \Vert u \Vert_{1.p} \quad \forall u \in \mathbf{W}^{1,p}(\Omega).
  \end{eqnarray}
and, with Korn's inequality (\cite{Korn's}), there exists $C_{Korn, p}>0$ such that
\begin{eqnarray} \label{korn}
  \left( \int_{\Omega} \bigl| D(u) \bigr|^p \, dx \right)^{1/p} = \Vert D(u) \Vert_{(L^p(\Omega))^{3 \times 3}} \ge C_{Korn, p} \Vert u \Vert_{1.p} \quad \forall u \in V^p_{\Gamma_{1}}.
  \end{eqnarray}
 
 Let ${\mathcal Y} = \bigl\{ \psi \in {\bf L}^2(\Omega); \ div(\psi) \in L^2(\Omega) \bigr\}$ endowed with its canonical norm
 \begin{eqnarray*}
 \Vert \psi \Vert_{\mathcal Y} = \Bigl( \Vert \psi \Vert^2_{{\bf L}^2(\Omega)} + \Vert div(\psi) \Vert^2_{{ L}^2(\Omega)} \Bigr)^{1/2} \qquad \forall \psi \in {\mathcal Y} 
 \end{eqnarray*}
 and let $H$ its subspace given by 
 $H= \bigl\{ \psi \in {\bf L}^2(\Omega); \ div(\psi) = 0 \ {\rm in } \ \Omega, \ \psi \cdot n = 0 \ {\rm on } \ \partial \Omega \bigr\}$.

\smallskip

Starting from the constitutive power law  (\ref{sigma})  we introduce the mapping 
$\mathcal{F}:  \mathbb{R}\times\mathbb{R}^{3}\times\mathbb{R}^{3\times3}\rightarrow \mathbb{R}^{3\times3}$  
given by 
\begin{eqnarray*}\label{F}
\left\{ \begin{array}{ll}
\displaystyle \mathcal{F}(\lambda_{0}, \lambda_{1},\lambda_{2})= 2\mu(\lambda_{0} ,\lambda_{1},|\lambda_{2}
|)|\lambda_{2}|^{p-2}\lambda_{2} \quad \hbox{\rm if }  \lambda_2 \not = 0_{\mathbb{R}^{3\times3}},  
\\
\displaystyle 
 \mathcal{F}(\lambda_{0}, \lambda_{1},\lambda_{2})= 0_{\mathbb{R}^{3\times3}} \quad {\rm otherwise}
\end{array} \right.
\end{eqnarray*} 
and we assume that 
 \begin{eqnarray}\label{rop}
  (o,e,d)\mapsto\mu (o,e,d)\quad \mbox{is continuous on } 
  \mathbb{R}\times\mathbb{R}^{3}\times \mathbb{R}_{+},
\end{eqnarray}
 \begin{eqnarray}\label{m5}
 d\mapsto \mu(.,.,d) \quad \mbox{is monotone  increasing   on } 
 \mathbb{R_{+}},
\end{eqnarray} 
\begin{eqnarray}\label{mlo}
\begin{array}{ll}
\displaystyle \hbox{\rm there exists  $(\mu_0, \mu_1) \in {\mathbb R}^2$ such that } \ 
\\
\displaystyle  0<\mu_{0}\leq\mu(o,e,d)\leq\mu_{1}
\quad \hbox{\rm for all  $(o,e,d)\in \mathbb{R}\times\mathbb{R}^{3}\times\mathbb{R}_{+}$}.
\end{array}
\end{eqnarray}
 With  (\ref{mlo}) we have immediately 
\begin{eqnarray}\label{F1}
  \bigl| \mathcal{F}(\lambda_{0}, \lambda_{1},\lambda_{2}) \bigr| \leq 2\mu_{1}|\lambda_{2}|^{p-1}  
  \quad \forall (\lambda_{0}, \lambda_{1},\lambda_{2})\in 
  \mathbb{R}\times\mathbb{R}^{3}\times\mathbb{R}^{3\times3}.
\end{eqnarray}
Let $q>1$ such that $q-p+1>0$. Then for any $\theta \in L^{\tilde q} \bigl( 0,T; L^{\tilde p} (\Omega) \bigr)$ with $\tilde q \ge 1$ and  $\tilde p \ge 1$, and for any $u \in L^q \bigl( 0,T; \mathbf{W}^{1,p}(\Omega) \bigr)$ we have
\begin{eqnarray*}
\mathcal{F} \bigl( \theta, u + \upsilon_0 \xi, D( u + \upsilon_0 \xi) \bigr) \in L^{\frac{q}{p-1}}\bigl( 0,T; \bigl(L^{p'}(\Omega)\bigr)^{3 \times 3} \bigr)
\end{eqnarray*}
where $\displaystyle p' = \frac{p}{p-1}$ is the conjugate number of $p$ and  we may define the integral term $\displaystyle \int_0^T \int_{\Omega} \mathcal{F} \bigl( \theta, u + \upsilon_0 \xi, D( u + \upsilon_0 \xi) \bigr) : D( \overline{\varphi}) \, dx dt  $ for all $\overline{\varphi} \in L^{\frac{q}{q-p+1}} \bigl( 0,T;  {\bf W}^{1,p} (\Omega)  \bigr)$. 

\smallskip

Since $L^q \bigl( 0,T;  V^p_{0} \bigr) \subset L^{\frac{q}{q-p+1}} \bigl( 0,T;  V^p_{0} \bigr)$ if and only  if $q \ge p$, we consider the operator $\mathcal{A}: L^{q} \bigl(0,T;V^p_{0} \bigr)\rightarrow  \bigl( L^{q} \bigl(0,T;V^p_{0}\bigr) \bigr)'$  defined by
\begin{eqnarray*}
\bigl[\mathcal{A} (u) ,\overline{\varphi} \bigr]
= \int_0^T \int_{\Omega} \mathcal{F} \bigl( \theta, u + \upsilon_0 \xi, D( u + \upsilon_0 \xi) \bigr) : D( \overline{\varphi}) \, dx dt 
\qquad \forall \overline{\varphi}\in L^{q} \bigl(0,T;V^p_{0} \bigr) 
\end{eqnarray*}
with $q \ge p$, where   $[.,.]$ denotes  the duality product
between  $L^{q } \bigl(0,T;V^p_{0} \bigr)$ and  its dual $ \bigl( L^{q} \bigl(0,T;V^p_{0})  \bigr)'$.  

\smallskip

With (\ref{F1})  we have 
\begin{eqnarray} \label{eq17ante}
\Vert \mathcal{A} (u)  \Vert_{ \bigl( L^{q} (0,T;V^p_{0})  \bigr)'} \le 2 \mu_1  \Vert u + \upsilon_0 \xi \Vert_{  L^{q} (0,T;V^p_{\Gamma_1}) }^{p-1}  \quad \forall u \in L^{q } \bigl(0,T;V^p_{0} \bigr)
\end{eqnarray}
and thus ${\mathcal A}$ is a bounded operator.
Furthermore
\begin{eqnarray}\label{3.2.1}
\begin{array}{ll}
\displaystyle \bigl[\mathcal{A}( u),u - \hat u   \bigr]  & \geq   \displaystyle 2(C_{Korn, p})^p\mu_{0}\Big|\Vert u\Vert_{L^p(0,T;V_{0}^p)} - \Vert\upsilon_{0}\xi \Vert_{L^p(0,T;V_{\Gamma_{1}}^p) }\Big|^p  \\
& \displaystyle  - 2\mu_{1} \Bigl( \Vert u\Vert_{L^q(0,T;V_{0}^p)} + \Vert\upsilon_{0}\xi \Vert_{L^q(0,T;V_{\Gamma_{1}}^p) } \Bigr)^{p-1}  
\Vert \upsilon_{0}\xi + \hat u \Vert_{L^{\frac{q}{q-p+1}} (0,T;V_{\Gamma_{1}}^p)}
\end{array}
\end{eqnarray}
for any $(u, \hat  u)  \in  \bigl( L^{q} \bigl(0,T;V_{0}^p \bigr) \bigr)^2$ and we can not infer  that ${\mathcal A}$ is coercive unless $p=q$.  In such a case (i.e. $p=q$) we may expect $\displaystyle \frac{\partial \overline{\upsilon}}{\partial t} \in L^{p'} \bigl(0,T; (V_{0.div}^p)' \bigr)$. Moreover the embedding of $V_{0.div}^p $ into $H$ is continuous and dense if and only if  $\displaystyle p \ge \frac{6}{5}$ and the functional space 
\begin{eqnarray*}
\Bigl\{ {\overline{\varphi}} \in L^p \bigl(0,T;V_{0.div}^p \bigr) ; \  \frac{\partial \overline{\varphi}}{\partial t} \in L^{p'} \bigl(0,T; (V_{0.div}^p)'  \bigr) \Bigr\}
\end{eqnarray*}
is continuously embedded into $C^0 \bigl([0,T]; H \bigr)$.
Hence we will consider from now on $\displaystyle q=p \ge \frac{6}{5}$.  

\bigskip

In order to deal with homogeneous boundary conditions on $(0,T) \times (\Gamma_{1}\cup\Gamma_{L}) $, we set $\overline{\upsilon}=\upsilon-\upsilon_{0}\xi$.
The variational formulation  of the problem (\ref{sigma})-(\ref{stokes}) with the boundary conditions (\ref{7})-(\ref{9}) and the initial condition (\ref{10})   is given by

\bigskip

\noindent {\bf{Problem (P)}} 
Let     $f\in  L^ {p'} \bigl(0,T;\textbf{L}^2(\Omega) \bigr)$, $k \in L^{p'} \bigl(0,T;L_{+}^{p'}(\Gamma_0) \bigr)$, $\mu$ satisfying (\ref{rop})-(\ref{mlo}), 
 $\theta \in  L^{\tilde q}\bigl(0,T;L^{\tilde p}(\Omega)\bigr)$ with $\tilde q \ge 1$ and $\tilde p \ge 1$, $s\in L^{ p} \bigl(0,T;\textbf{L}^p(\Gamma_0)\bigr)$, $\xi\in W^{1,p'}(0,T)$ satisfying (\ref{xi})
and  $\upsilon_{0}\in \textbf{W}^{1,p}(\Omega)$ satisfying  (\ref{coco}). 
 Find  $\overline{\upsilon}\in C \bigl([0,T];\mathbf{L}^2(\Omega) \bigr)\cap L^{ p} \bigl(0,T;V^p_{0.div} \bigr)$  with  $\displaystyle \frac{\partial \overline{\upsilon}}{\partial t} \in L^{p'} \bigl(0,T; (V^p_{0.div})' \bigr)$  and  $\pi \in H^{-1}\bigl(0,T;L_{0}^{p'}(\Omega) \bigr)$ satisfying the following parabolic variational inequality
\begin{eqnarray*}\label{3.14.0}
\begin{array}{ll}
\displaystyle{ \Big<\frac{\partial}{\partial t}(\overline{\upsilon} , \tilde\vartheta )_{\bf{L}^2(\Omega)},\zeta\Big>_{\mathcal{D}'(0,T),\mathcal{D}(0,T)}
+ \int_{0}^{T} \int_{\Omega} \mathcal{F} \bigl( \theta, \overline{\upsilon} + \upsilon_0 \xi, D( \overline{\upsilon} + \upsilon_0 \xi) \bigr) : D( \tilde\vartheta) \zeta \, dx dt}
\\
\displaystyle{ 
-  \Big< \int_{\Omega} \pi \ div(\tilde \vartheta ), \zeta\Big>_{\mathcal{D}'(0,T),\mathcal{D}(0,T)} 
+ J( \overline{\upsilon}+\tilde\vartheta \zeta)-J(\overline{\upsilon})
} 
\medskip\\
\displaystyle{
\geq\,\, \int_{0}^{T} \left( f + \frac{\partial \xi}{\partial t} \upsilon_0, \tilde\vartheta \right)_{{\bf L}^2(\Omega)}\zeta\,dt} 
\qquad \quad  \forall \tilde\vartheta\in V^p_{0},  \  \forall \zeta\in \mathcal{D}(0,T)  
 \end{array}
 \end{eqnarray*}
and the initial condition
 \begin{eqnarray*}\label{3.15}
\overline{\upsilon}(0)=\upsilon_{0}-\upsilon_{0}\xi(0)=0 \qquad  \mbox{in}\  \Omega
\end{eqnarray*}
where
\begin{eqnarray*}
J :\left \{
\begin{array}{ll}
L^p ( 0,T; V^p_{0} )&\rightarrow \mathbb{R}\\
\qquad \qquad \overline{\varphi}& \mapsto \displaystyle{\int_0^T \int _{\Gamma_0}k|\overline{\varphi} -\tilde s|\,dx' dt }, \qquad \tilde s=s-(\upsilon_{0})_{\tau}\xi
\end{array} \right.
\end{eqnarray*}
and  $\big<.,.\big>_{\mathcal{D}'(0,T),\mathcal{D}(0,T)}$ and $(.,.)_{{\bf L}^2(\Omega) }$  denote respectively  the duality product between $\mathcal{D}(0,T)$ and $\mathcal{D}'(0,T)$
and     the inner product in $ {\bf L}^2(\Omega) $.

\bigskip


\section{Auxiliary flow problem } \label{approximate-problems}

\bigskip

We consider in this section an auxiliary  problem where the mapping $\mu$ depends on the modulus of the strain tensor while its two other arguments are given data. More precisely let 
${\bf u} \in L^p \bigl( 0,T; {\bf L}^p (\Omega) \bigr)$ be given. We consider the following auxiliary flow problem

\bigskip

\noindent {\bf{Problem (P${}_{\bf u}$)}} 
Let     $f\in  L^ {p'} \bigl(0,T;\textbf{L}^2(\Omega) \bigr)$, $k \in L^{p'} \bigl(0,T;L_{+}^{p'}(\Gamma_0) \bigr)$, $\mu$ satisfying (\ref{rop})-(\ref{mlo}), 
 $\theta \in  L^{\tilde q}\bigl(0,T;L^{\tilde p}(\Omega)\bigr)$ with $\tilde q \ge 1$ and $\tilde p \ge 1$, $s\in L^{ p} \bigl(0,T;\textbf{L}^p(\Gamma_0)\bigr)$, $\xi\in W^{1,p'}(0,T)$ satisfying (\ref{xi})
and  $\upsilon_{0}\in \textbf{W}^{1,p}(\Omega)$ satisfying  (\ref{coco}). 
 Find  $\overline{\upsilon}\in C \bigl([0,T];\mathbf{L}^2(\Omega) \bigr)\cap L^{ p} \bigl(0,T;V^p_{0.div} \bigr)$  with  $\displaystyle \frac{\partial \overline{\upsilon}}{\partial t} \in L^{p'} \bigl(0,T; (V^p_{0.div})' \bigr)$  and  $\pi \in H^{-1}\bigl(0,T;L_{0}^{p'}(\Omega) \bigr)$ satisfying the following parabolic variational inequality
\begin{eqnarray*}\label{3.14.00}
\begin{array}{ll}
\displaystyle{ \Big<\frac{\partial}{\partial t}(\overline{\upsilon} , \tilde\vartheta )_{\bf{L}^2(\Omega)},\zeta\Big>_{\mathcal{D}'(0,T),\mathcal{D}(0,T)}
+ \int_{0}^{T} \int_{\Omega} \mathcal{F} \bigl( \theta, {\bf u} + \upsilon_0 \xi, D( \overline{\upsilon} + \upsilon_0 \xi) \bigr) : D( \tilde\vartheta) \zeta \, dx dt}
\\
\displaystyle{
-  \Big< \int_{\Omega} \pi \ div(\tilde \vartheta ), \zeta\Big>_{\mathcal{D}'(0,T),\mathcal{D}(0,T)} 
+ J( \overline{\upsilon}+\tilde\vartheta \zeta)-J(\overline{\upsilon})
} 
\medskip\\
\displaystyle{
\geq\,\, \int_{0}^{T} \left( f + \frac{\partial \xi}{\partial t} \upsilon_0, \tilde\vartheta \right)_{{\bf L}^2(\Omega)}\zeta\,dt} 
\qquad \quad  \forall \tilde\vartheta\in V^p_{0},  \  \forall \zeta\in \mathcal{D}(0,T)  
 \end{array}
 \end{eqnarray*}
and the initial condition
 \begin{eqnarray*}\label{3.15}
\overline{\upsilon}(0)=\upsilon_{0}-\upsilon_{0}\xi(0)=0 \qquad  \mbox{in}\  \Omega .
\end{eqnarray*}

 By Lemma 1 in \cite{BDP1} we know that the mapping $\lambda_2 \mapsto {\mathcal F} (\cdot, \cdot, \lambda_2)$ is monotone in ${\mathbb R}^{3 \times 3}$ for any $p>1$. So the mathematical framework seems the same as in Section 3 in \cite{BDP2} where we consider a similar problem in  the case $p \ge 2$. Unfortunately  the space 
$L^p \bigl( 0,T;  V^p_{0.div} \bigr)$  is not embedded into its dual when $\displaystyle p \in \left[6/5, 2 \right)$ and 
some key properties  of the semi-group of contractions $\bigl( S(h) \bigr)_{h \ge 0}$ used in \cite{BDP2} (see Proposition 3.1 in \cite{BDP2}), are not any more satisfied. So 
we can not reproduce the same proof strategy as in \cite{BDP2}.

\bigskip

In order to overcome this difficulty we will apply a vanishing viscosity technique and we introduce a pertubed constitutive law given by 
 \begin{eqnarray*} \label{sigma_epsilon}
 \begin{array}{ll}
\displaystyle  \sigma^{\varepsilon} = 2 \mu \bigl( \theta,  {\bf u} + \upsilon_0 \xi , \bigl| D( \overline{\upsilon} + \upsilon_0 \xi) \bigr| \bigr) \bigl|D( \overline{\upsilon} + \upsilon_0 \xi) \bigr|^{p-2} D( \overline{\upsilon} + \upsilon_0 \xi)
  \\
  \displaystyle
  + 2 \varepsilon \bigl|D( \overline{\upsilon} ) \bigr|^{p'-2} D( \overline{\upsilon} )
 - \pi {\rm Id}_{{\mathbb{R}^3}}
 \end{array}
 \end{eqnarray*}
%
 where we recall that $\displaystyle p'= \frac{p}{p-1} >2$ is the conjugate number of $p$. Hence we consider  the following  approximate  variational inequality
 \begin{eqnarray*}\label{3.14.000}
\begin{array}{ll}
\displaystyle{ \Big<\frac{\partial}{\partial t}(\overline{\upsilon}_{\varepsilon} , \tilde\vartheta )_{\bf{L}^2(\Omega)},\zeta\Big>_{\mathcal{D}'(0,T),\mathcal{D}(0,T)}
+ \int_{0}^{T} \int_{\Omega} \mathcal{F} \bigl( \theta, {\bf u} + \upsilon_0 \xi, D( \overline{\upsilon}_{\varepsilon} + \upsilon_0 \xi) \bigr) : D( \tilde\vartheta) \zeta \, dx dt } 
\medskip \\
\displaystyle{
+ 2 \varepsilon \int_0^T \int_{\Omega} \bigl| D(\overline{\upsilon}_{\varepsilon} )  \bigr|^{p'-2} D(\overline{\upsilon}_{\varepsilon} )  : D( \tilde\vartheta) \zeta \, dx dt
-  \Big< \int_{\Omega} \pi_{\varepsilon} \ div(\tilde \vartheta ) \, dx , \zeta\Big>_{\mathcal{D}'(0,T),\mathcal{D}(0,T)} } 
\medskip\\
\displaystyle{
+ J( \overline{\upsilon}_{\varepsilon} +\tilde\vartheta \zeta)-J(\overline{\upsilon}_{\varepsilon} )
\geq\,\, \int_{0}^{T} \left( f + \frac{\partial \xi}{\partial t} \upsilon_0, \tilde\vartheta \right)_{{\bf L}^2(\Omega)}\zeta\,dt} 
\qquad \quad  \forall \tilde\vartheta\in V^{p'}_{0},  \  \forall \zeta\in \mathcal{D}(0,T).  
 \end{array}
 \end{eqnarray*}


\bigskip

Let  $ u \in L^{p'} ( 0,T;  V^{p'}_{0.div} )$. 
For the sake of notational simplicity let us define $\mathcal{A}_{{\bf u}}^{\varepsilon} (u)$  by 
\begin{eqnarray*}
\begin{array}{ll}
\displaystyle \bigl[ \bigl[\mathcal{A}_{{\bf u}}^{\varepsilon} (u) ,\overline{\varphi} \bigr]\bigr]
= \int_0^T \int_{\Omega} \mathcal{F} \bigl( \theta, {\bf u} + \upsilon_0 \xi, D( u + \upsilon_0 \xi) \bigr) : D( \overline{\varphi}) \, dx dt  \\
\displaystyle 
+  2 \varepsilon \int_0^T \int_{\Omega}  \bigl|D(u ) \bigr|^{p'-2} D(u  ) : D( \overline{\varphi}) \, dx dt 
\qquad \forall \overline{\varphi}\in L^{p'}(0,T;V^{p'}_{0.div}) 
\end{array}
\end{eqnarray*}
where   $[[.,.]]$ denotes  the duality product
between  $L^{p' }(0,T;V^{p'}_{0.div})$ and  its dual $ \bigl( L^{p'}(0,T;V^{p'}_{0.div}) \bigr)'$.  By observing that 
$L^{p' }(0,T;V^{p'}_{0.div}) \subset L^{p }(0,T;V^{p}_{0.div})$ we obtain immediately with  (\ref{eq17ante}) and (\ref{3.2.1})  that $\mathcal{A}_{{\bf u}}^{\varepsilon}$ is bounded and coercive on $L^{p' }(0,T;V^{p'}_{0.div}) $. Moreover (\ref{rop}) implies that $\mathcal{A}_{{\bf u}}^{\varepsilon}$ is hemicontinuous and using (\ref{rop})-(\ref{m5}) and Lemma 1 in \cite{BDP1} we obtain that $\mathcal{A}_{{\bf u}}^{\varepsilon}$ is monotone.
Furthermore the mapping $\displaystyle {\overline \varphi} \mapsto \int_0^T \int_{\Gamma_0} k | {\overline \varphi} - \tilde s | \, dx' dt$ is convex and Lipschitz continuous on $L^{p'} ( 0,T; V_0^{p'} )$. 
It follows that we can apply the existence result obtained in \cite{BDP2} and we have

\begin{proposition} \label{existence-epsilon}
Let     $f\in  L^ {p'} \bigl(0,T;\textbf{L}^2(\Omega) \bigr)$, $k \in L^{p'} \bigl(0,T;L_{+}^{p'}(\Gamma_0) \bigr)$, $\mu$ satisfying (\ref{rop})-(\ref{mlo}), 
 $\theta \in  L^{\tilde q}\bigl(0,T;L^{\tilde p}(\Omega)\bigr)$ with $\tilde q \ge 1$ and $\tilde p \ge 1$, $s\in L^{ p} \bigl(0,T;\textbf{L}^{p}(\Gamma_0)\bigr)$, $\xi\in W^{1,p'}(0,T)$ satisfying (\ref{xi})
and  $\upsilon_{0}\in \textbf{W}^{1,p}(\Omega)$ satisfying  (\ref{coco}). 
 For any $\varepsilon >0$ there exists $\overline{\upsilon}_{\varepsilon}\in C \bigl([0,T];\mathbf{L}^2(\Omega) \bigr)\cap L^{ p'}\bigl(0,T; V^{p'}_{0.div} \bigr)$  with  $\displaystyle \frac{\partial \overline{\upsilon}_{\varepsilon}}{\partial t} \in L^{p} \bigl(0,T; (V^{p'}_{0.div})' \bigr)$  and  $\pi_{\varepsilon} \in H^{-1}\bigl(0,T;L_{0}^{p}(\Omega) \bigr)$ satisfying the following parabolic variational inequality
\begin{eqnarray}\label{3.14}
\begin{array}{ll}
\displaystyle{ \Big<\frac{\partial}{\partial t}(\overline{\upsilon}_{\varepsilon} , \tilde\vartheta )_{\bf{L}^2(\Omega)},\zeta\Big>_{\mathcal{D}'(0,T),\mathcal{D}(0,T)}
+ \int_{0}^{T} \int_{\Omega} \mathcal{F} \bigl( \theta, {\bf u} + \upsilon_0 \xi, D( \overline{\upsilon}_{\varepsilon} + \upsilon_0 \xi) \bigr) : D( \tilde\vartheta) \zeta \, dx dt } 
\medskip \\
\displaystyle{
+ 2 \varepsilon \int_0^T \int_{\Omega} \bigl| D(\overline{\upsilon}_{\varepsilon} )  \bigr|^{p'-2} D(\overline{\upsilon}_{\varepsilon} )  : D( \tilde\vartheta) \zeta \, dx dt
-  \Big< \int_{\Omega} \pi_{\varepsilon} \ div(\tilde \vartheta ) \, dx , \zeta\Big>_{\mathcal{D}'(0,T),\mathcal{D}(0,T)} } 
\medskip\\
\displaystyle{
+ J( \overline{\upsilon}_{\varepsilon} +\tilde\vartheta \zeta)-J(\overline{\upsilon}_{\varepsilon} )
\geq\,\, \int_{0}^{T} \left( f + \frac{\partial \xi}{\partial t} \upsilon_0, \tilde\vartheta \right)_{{\bf L}^2(\Omega)}\zeta\,dt} 
\qquad \quad  \forall \tilde\vartheta\in V^{p'}_{0},  \  \forall \zeta\in \mathcal{D}(0,T)  
 \end{array}
 \end{eqnarray}
and the initial condition
 \begin{eqnarray}\label{3.15}
\overline{\upsilon}_{\varepsilon}(0)=\upsilon_{0}-\upsilon_{0}\xi(0)=0 \qquad  \mbox{in} \  \Omega.
\end{eqnarray}
\end{proposition}

\bigskip

Let us observe that for any ${\overline{\varphi}}= \tilde\vartheta \zeta$ with $\tilde\vartheta \in V^{p'}_{0.div}$ and $\zeta\in \mathcal{D}(0,T) $ we have 
\begin{eqnarray}\label{3.14bis}
\begin{array}{ll}
\displaystyle 
\underbrace{ \Big<\frac{\partial}{\partial t}(\overline{\upsilon}_{\varepsilon} , \tilde\vartheta )_{\bf{L}^2(\Omega)},\zeta\Big>_{\mathcal{D}'(0,T),\mathcal{D}(0,T)}}_{\displaystyle = \int_0^T \Big< \frac{\partial \overline{\upsilon}_{\varepsilon}}{\partial t}, {\overline{\varphi}}  \Big>_{ (V^{p'}_{0.div})', V^{p'}_{0.div} }    \, dt }
+ \bigl[ \bigl[\mathcal{A}_{{\bf u}}^{\varepsilon} (\overline{\upsilon}_{\varepsilon} ) , {\overline{\varphi}}  \bigr]\bigr]
\medskip\\
\displaystyle{
+ J( \overline{\upsilon}_{\varepsilon} + {\overline{\varphi}} )-J(\overline{\upsilon}_{\varepsilon} )
\geq\,\, \int_{0}^{T} \left( f + \frac{\partial \xi}{\partial t} \upsilon_0, {\overline{\varphi}}  \right)_{{\bf L}^2(\Omega)} \,dt} 
 \end{array}
 \end{eqnarray}
By density of $\mathcal{D}(0,T) \otimes V^{p'}_{0.div}$ into $L^{p'} \bigl(0,T; V^{p'}_{0.div} \bigr)$ the same inequality is true for any ${\overline{\varphi}} \in L^{p'} \bigl(0,T; V^{p'}_{0.div} \bigr)$.

\bigskip

\subsection{A priori estimates}

\bigskip


Let us establish now some a priori estimates for the sequence $(\overline{\upsilon}_{\varepsilon} , \pi_{\varepsilon})_{\varepsilon>0}$.


\begin{proposition} \label{apriori-estimates1}
Let ${\bf u} \in L^ {p} \bigl(0,T; {\bf L}^p (\Omega) \bigr)$. Let     $f\in  L^ {p'} \bigl(0,T; {\bf L}^2(\Omega) \bigr)$, $k \in L^{p'} \bigl(0,T;L_{+}^{p'}(\Gamma_0) \bigr)$, $\mu$ satisfying (\ref{rop})-(\ref{mlo}), 
 $\theta \in  L^{\tilde q}\bigl(0,T;L^{\tilde p}(\Omega)\bigr)$ with $\tilde q \ge 1$ and $\tilde p \ge 1$, $s\in L^{ p} \bigl(0,T; {\bf L}^{p}(\Gamma_0)\bigr)$, $\xi\in W^{1,p'}(0,T)$ satisfying (\ref{xi})
and  $\upsilon_{0}\in \textbf{W}^{1,p}(\Omega)$ satisfying  (\ref{coco}). Then there exists a constant $C$, independant of $\varepsilon$ such that, for all $\varepsilon \in (0, 1]$, we have
\begin{eqnarray}\label{203}
\Vert \overline{\upsilon}_{\varepsilon}\Vert_{L^{p}(0,T; V_{0.div}^p )}\leq C
\end{eqnarray}
 \begin{eqnarray}\label{es_epsilon1}
 \varepsilon^{1/p'} \Vert \overline{\upsilon}_{\varepsilon} \Vert_{L^{p'}(0,T;V_{0.div}^{p'})}\leq C
 \end{eqnarray}
 and
\begin{eqnarray}\label{es_epsilon2}
 \Vert \overline{\upsilon}_{\varepsilon} \Vert_{L^{\infty}(0,T; {\bf L}^2(\Omega))} 
 \leq C.
\end{eqnarray}
\end{proposition}

\begin{proof} 
Let  $t\in(0,T]$ and $\overline{\varphi} = - \overline{\upsilon}_{\varepsilon} {\bf 1}_{[0,t]} \in L^{p'} \bigl(0,T; V^{p'}_{0.div} \bigr)$ 
where ${\bf 1}_{[0,t]}$ is the indicatrix function of the time interval $[0,t]$.
With (\ref{3.14bis}) we obtain
 \begin{eqnarray*}\label{1.219 }
\begin{array}{ll}
 \displaystyle{ \int_{0}^{t} \Big< \frac{ \partial \overline{\upsilon}_{\varepsilon}}{\partial t} ,\overline{\upsilon}_{\varepsilon} \Big>_{(V_{0.div}^{p'})',V_{0.div}^{p'}} \,d\tilde t
 +\int_{0}^{t} \int_{\Omega} {\mathcal F}  \bigl(\theta, {\bf u} + \upsilon_0 \xi , D( \overline{\upsilon}_{\varepsilon} + \upsilon_0 \xi) \bigr) : D( \overline{\upsilon}_{\varepsilon}) \,dx d \tilde t} \\
  \displaystyle + 2 \varepsilon  \int_{0}^{t} \int_{\Omega} \bigl|  D( \overline{\upsilon}_{\varepsilon} ) \bigr|^{p'-2}  D( \overline{\upsilon}_{\varepsilon} ) : D( \overline{\upsilon}_{\varepsilon}) \,dx d\tilde t
 \\
   \displaystyle{\  \leq \int_{0}^{t} \left( f + \frac{ \partial \xi}{\partial t} \upsilon_0 , \overline{\upsilon}_{\varepsilon}  \right)_{{\bf L}^2(\Omega) } \,d \tilde t
+  \int_{0}^{t} \int_{\Gamma_0} k | \tilde s| \,dx' d \tilde t}.
\end{array}
\end{eqnarray*} 
From (\ref{mlo}) and (\ref{kornbis})-(\ref{korn}) we have 
\begin{eqnarray*}\label{1.220 }
\begin{array}{ll}
\displaystyle{\int_{0}^{t} \int_{\Omega} {\mathcal F}  \bigl(\theta, {\bf u}  + \upsilon_0 \xi , D( \overline{\upsilon}_{\varepsilon} + \upsilon_0 \xi) \bigr) : D( \overline{\upsilon}_{\varepsilon}) \,dx d \tilde t }
\\
\displaystyle \geq \displaystyle{2(C_{Korn, p})^{p} \mu_{0}\int_{0}^{t} \Vert \overline{\upsilon}_{\varepsilon}+\upsilon_{0}\xi\Vert _{1.p}^{p}\,d\tilde t}
- \displaystyle{2\mu_{1} \int_{0}^{t} \Vert  \upsilon_{0} \xi \Vert_{1.p} \Vert\overline{\upsilon}_{\varepsilon}+\upsilon_{0}\xi\Vert_{1.p}^{p-1}\,d\tilde t} 
\end{array}
\end{eqnarray*}
and with Young's inequality
\begin{eqnarray*}\label{1.220 }
\begin{array}{ll}
\displaystyle{\int_{0}^{t} \int_{\Omega} {\mathcal F}  \bigl(\theta, {\bf u}  + \upsilon_0 \xi , D( \overline{\upsilon}_{\varepsilon} + \upsilon_0 \xi) \bigr) : D( \overline{\upsilon}_{\varepsilon}) \,dx d \tilde t }
\\
 \geq \displaystyle{2(C_{Korn, p})^{p}  \mu_0
 \Bigl|  \Vert \overline{\upsilon}_{\varepsilon} \Vert_{L^{p} (0,t; V_{0.div}^{p})} 
 - \Vert \upsilon_{0}\xi\Vert_{L^{p} (0,t; V^{p}_{\Gamma_1})} \Bigr|^{p} }
\medskip\\
- \displaystyle{2\mu_1
 \Bigl( \Vert\overline{\upsilon}_{\varepsilon} \Vert_{L^{p} (0,t; V_{0.div}^{p})} 
 + \Vert \upsilon_{0}\xi\Vert_{L^{p} (0,t; V_{\Gamma_1}^{p}) } \Bigr)^{p -1} 
 \Vert  \upsilon_{0} \xi \Vert_{L^{p} (0,t; V_{\Gamma_1}^{p})} } .
\end{array}
\end{eqnarray*}
Similarly
\begin{eqnarray*}\label{1.220}
\begin{array}{ll}
\displaystyle 2 \varepsilon  \int_{0}^{t} \int_{\Omega} \bigl|  D( \overline{\upsilon}_{\varepsilon} ) \bigr|^{p'-2}  
D( \overline{\upsilon}_{\varepsilon}) : D( \overline{\upsilon}_{\varepsilon}) \,dx d \tilde t 
= 2 \varepsilon  \int_{0}^{t} \int_{\Omega} \bigl|  D( \overline{\upsilon}_{\varepsilon} ) \bigr|^{p'}  \,dx d \tilde t 
\\
\displaystyle 
 \geq \displaystyle 2 \varepsilon (C_{Korn, p'})^{p'}  
   \Vert \overline{\upsilon}_{\varepsilon} \Vert_{L^{p'} (0,t; V_{0.div}^{p'})}^{p'} .
\end{array}
\end{eqnarray*}

 For the sake of notational simplicity let us define $\displaystyle \overline{f} = f + \upsilon_0 \frac{ \partial \xi}{\partial t} \in  L^{p'} \bigl(0,T; {\bf L}^2(\Omega) \bigr) $. Then we obtain  
\begin{eqnarray} \label{norme_infini_b1}
\begin{array}{ll}
 &\displaystyle{\frac{1}{2}\Vert\overline{\upsilon}_{\varepsilon}(t)\Vert_{\mathbf{L}^2(\Omega)}^2} 
 \displaystyle{+ 2(C_{Korn, p})^{p}  \mu_0
 \Bigl|  \Vert \overline{\upsilon}_{\varepsilon} \Vert_{L^{p} (0,t; V_{0.div}^{p})} 
 - \Vert \upsilon_{0}\xi\Vert_{L^{p} (0,t; V^{p}_{\Gamma_1})} \Bigr|^{p} }
  \\
& 
+ \displaystyle{ 2 \varepsilon (C_{Korn, p'})^{p'}  \Vert \overline{\upsilon}_{\varepsilon} \Vert_{L^{p'} (0,t; V_{0.div}^{p'})}^{p'} }  \\
  &
+ \displaystyle{ \leq \displaystyle{\widetilde C \Vert \overline{f}\Vert _{L^{p'} (0,t; {\bf L}^2(\Omega))}
  \Vert \overline{\upsilon}_{\varepsilon}\Vert_{L^p(0,t; V_{0.div}^p)} 
   + \int_{0}^{t} \int_{\Gamma_0} k |\tilde s| \, dx' \,d\tilde t } } \\
 &
+ \displaystyle{2\mu_1
 \Bigl( \Vert\overline{\upsilon}_{\varepsilon} \Vert_{L^{p} (0,t; V_{0.div}^{p})} 
 + \Vert \upsilon_{0}\xi\Vert_{L^{p} (0,t; V_{\Gamma_1}^{p}) } \Bigr)^{p -1} 
 \Vert  \upsilon_{0} \xi \Vert_{L^{p} (0,t; V_{\Gamma_1}^{p})} }
\end{array}
\end{eqnarray}
where $\widetilde C$ denotes the norm of the continuous injection of $V^p_{0}$ into ${\bf L}^2 (\Omega)$.

\smallskip

Let us consider first $t=T$ and assume that $\varepsilon \in (0,1]$. We get
\begin{eqnarray*}\label{norme_infini_b}
\begin{array}{ll}
 & \displaystyle{ 2(C_{Korn, p})^{p}  \mu_0
 \Bigl|  \Vert \overline{\upsilon}_{\varepsilon} \Vert_{L^{p} (0,T; V_{0.div}^{p})} 
 - \Vert \upsilon_{0}\xi\Vert_{L^{p} (0,T; V^{p}_{\Gamma_1})} \Bigr|^{p} }
  \\
 & \leq \displaystyle{\widetilde C \Vert \overline{f}\Vert _{L^{p'} (0,T; {\bf L}^2(\Omega))}
  \Vert \overline{\upsilon}_{\varepsilon}\Vert_{L^p(0,T; V_{0.div})}  + J(0) } \\
 &
+ \displaystyle{2\mu_1
 \Bigl( \Vert\overline{\upsilon}_{\varepsilon} \Vert_{L^{p} (0,T; V_{0.div}^{p})} 
 + \Vert \upsilon_{0}\xi\Vert_{L^{p} (0,T; V_{\Gamma_1}^{p}) } \Bigr)^{p -1} 
 \Vert  \upsilon_{0} \xi \Vert_{L^{p} (0,T; V_{\Gamma_1}^{p})} } .
\end{array}
\end{eqnarray*}
If $\Vert \overline{\upsilon}_{\varepsilon} \Vert_{L^{p} (0,T; V_{0.div}^{p})} \not=0$ it follows that
\begin{eqnarray*}\label{norme_infini_b}
\begin{array}{ll}
 &
 \displaystyle{ 2(C_{Korn, p})^{p}  \mu_0
 \left|  1 - 
  \frac{ \Vert \upsilon_{0}\xi\Vert_{L^{p} (0,T; V^{p}_{\Gamma_1})} }{\Vert \overline{\upsilon}_{\varepsilon} \Vert_{L^{p} (0,T; V_{0.div}^{p})} } \right|^{p} }
  \\
&  \leq \displaystyle{\widetilde C \Vert \overline{f}\Vert _{L^{p'} (0,T; {\bf L}^2(\Omega))}
  \Vert \overline{\upsilon}_{\varepsilon}\Vert_{L^p(0,T; V_{0.div}^p)}^{1-p}  
  + \frac{J(0)}{\Vert \overline{\upsilon}_{\varepsilon} \Vert_{L^{p} (0,T; V_{0.div}^{p})}^p}  } \\
 &
+ \displaystyle{2\mu_1
 \left( 1  
 + \frac{\Vert \upsilon_{0}\xi\Vert_{L^{p} (0,T; V_{\Gamma_1}^{p}) }}{\Vert \overline{\upsilon}_{\varepsilon} \Vert_{L^{p} (0,T; V_{0.div}^{p})} } \right)^{p -1} 
 \frac{\Vert  \upsilon_{0} \xi \Vert_{L^{p} (0,T; V_{\Gamma_1}^{p})} }{\Vert \overline{\upsilon}_{\varepsilon} \Vert_{L^{p} (0,T; V_{0.div}^{p})} }} .
\end{array}
\end{eqnarray*}
By observing that the mapping 
\begin{eqnarray*}
\begin{array}{ll}
\displaystyle 
z \mapsto 
& 
\displaystyle {2(C_{Korn, p})^{p}  \mu_0
 \left|  1 - 
  \frac{\Vert \upsilon_{0}\xi\Vert_{L^{p} (0,T; V_{\Gamma_1}^{p}) }}{z }
  \right|^{p} }
  -  \displaystyle{\widetilde C \frac{\Vert \overline{f}\Vert _{L^{p'} (0,T; {\bf L}^2(\Omega))}}{z^{p-1}}  
  - \frac{J(0)}{z^p}  } 
  \\
 &
- \displaystyle{2\mu_1
 \left( 1  
 + \frac{\Vert \upsilon_{0}\xi\Vert_{L^{p} (0,T; V_{\Gamma_1}^{p}) }}{z } \right)^{p -1} 
 \frac{\Vert  \upsilon_{0} \xi \Vert_{L^{p} (0,T; V_{\Gamma_1}^{p})} }{z }}
 \end{array}
 \end{eqnarray*}
 admits $2(C_{Korn, p})^{p}  \mu_0>0$ as limit when $z$ tends to $+ \infty$, we infer that there exists a real number $C>0$, independent of $\varepsilon$, such that 
 \begin{eqnarray*}
 \Vert \overline{\upsilon}_{\varepsilon} \Vert_{L^{p} (0,T; V_{0.div}^{p})} \le C \qquad \forall \varepsilon \in (0,1]
 \end{eqnarray*}
 which yields (\ref{203}). Going back to (\ref{norme_infini_b1}) we obtain immediately (\ref{es_epsilon1}) and (\ref{es_epsilon2}).
 \end{proof}
 
 \begin{remark}
 Let us emphasize that the constant $C>0$ is independent of ${\bf u}$.
 \end{remark}


\begin{proposition} \label{apriori-estimates2}
Under the same assumptions as in Proposition \ref{apriori-estimates1} there exists a constant $C' >0$, independant of $\varepsilon$ such that, for all $\varepsilon \in (0, 1]$, we have
\begin{eqnarray} \label{es_derivee_epsilon}
\left\Vert \frac{ \partial \overline{\upsilon}_{\varepsilon}}{\partial t} \right\Vert_{L^p(0,T;(V_{0.div}^{p'} )')} \leq C'
\end{eqnarray}
and
 \begin{eqnarray}\label{es_pression_epsilon}
\Vert \pi_{\varepsilon}\Vert _{H^{-1} (0,T;L^{p}(\Omega))}\leq   C' .
\end{eqnarray}
\end{proposition}

\begin{proof}
Let us choose   $\overline{\varphi}=\pm\tilde  \vartheta \zeta$  with  $\tilde  \vartheta \in V^{p'}_{0.div} $
 and  $\zeta \in \mathcal{D}(0,T)$  in (\ref{3.14bis}) we obtain
\begin{eqnarray*}\label{inega_1}
\begin{array}{ll}
\displaystyle{ \int_0^T \left\langle \frac{\partial \overline{\upsilon}_{\varepsilon} }{\partial t} , \pm \tilde \vartheta \zeta \right\rangle_{(V_{0.div}^{p'})', V_{0.div}^{p'}}  \, dt 
+
\bigl[\bigl[ \mathcal{A}_{{\bf u}}^{\varepsilon} (\overline{\upsilon}_{\varepsilon}) ,\pm \tilde \vartheta \zeta \bigr]\bigr] }
 \\
 \displaystyle{ +J(\overline{\upsilon}_{\varepsilon} \pm \tilde \vartheta \zeta)
-J(\overline{\upsilon}_{\varepsilon} )
\geq  \int_0^T ( \overline{f},\pm\tilde \vartheta \zeta)_{{\bf L}^2(\Omega)} \, dt}.
\end{array}
\end{eqnarray*}
But
\begin{eqnarray*}
\displaystyle \bigl| J(\overline{\upsilon}_{\varepsilon} \pm \tilde \vartheta \zeta)
-J(\overline{\upsilon}_{\varepsilon} ) \bigr| 
\le \int_0^T \int_{\Gamma_0} k | \tilde \vartheta \zeta| \, dx'  dt 
\end{eqnarray*}
and recalling that $k \in L^{p'} \bigl( 0,T; L^{p'}_+(\Gamma_0) \bigr) \subset L^{p} \bigl( 0,T; L^{p}_+(\Gamma_0) \bigr)$ we get
\begin{eqnarray*}
\displaystyle \bigl| J(\overline{\upsilon}_{\varepsilon} \pm \tilde \vartheta \zeta)
-J(\overline{\upsilon}_{\varepsilon} ) \bigr| 
\le \| \gamma_{p'} \|_{{\mathcal L} ({\bf W}^{1,p'} (\Omega) , {\bf L}^{p'} (\partial \Omega))} \| k\|_{L^p(0,T; L^p(\Gamma_0))} \| \tilde \vartheta \zeta \|_{L^{p'} (0,T; V_{0.div}^{p'})}
\end{eqnarray*}
where $\gamma_{p'}$ denotes the trace operator from ${\bf W}^{1,p'} (\Omega)$ into ${\bf L}^{p'} (\partial \Omega)$. Since $\overline{f} \in L^{p'} \bigl( 0,T; {\bf L}^2(\Omega) \bigr) \subset L^{p} \bigl( 0,T; {\bf L}^2(\Omega) \bigr)$ we obtain 
\begin{eqnarray}\label{inega_1}
\begin{array}{ll}
& \displaystyle{ \left| \int_0^T  \left\langle \frac{\partial \overline{\upsilon}_{\varepsilon} }{\partial t} ,  \tilde \vartheta \zeta \right\rangle_{(V_{0.div}^{p'})', V_{0.div}^{p'}} \, dt \right| } 
\\
& 
\le 
\| \gamma_{p'} \|_{{\mathcal L} ({\bf W}^{1,p'} (\Omega) , {\bf L}^{p'} (\partial \Omega))} \| k\|_{L^p(0,T; L^p(\Gamma_0))} \| \tilde \vartheta \zeta \|_{L^{p'} (0,T; V_{0.div}^{p'})} 
\\
& 
\displaystyle 
+ \widetilde C' \| \overline{f} \|_{L^{p} (0,T; {\bf L}^2(\Omega))} \| \tilde \vartheta \zeta \|_{L^{p'} (0,T; V_{0.div}^{p'})}
+ 
\Bigl|\bigl[\bigl[ \mathcal{A}_{{\bf u}}^{\varepsilon} (\overline{\upsilon}_{\varepsilon}) , \tilde \vartheta \zeta \bigr]\bigr] \Bigr| 
\end{array}
\end{eqnarray}
where $\widetilde C'$ denotes the norm of the continuous injection of $V_{0}^{p'}$ into ${\bf L}^2(\Omega)$.

On the other hand
\begin{eqnarray*}
\begin{array}{ll}
& \displaystyle \Bigl|\bigl[\bigl[ \mathcal{A}_{{\bf u}}^{\varepsilon} (\overline{\upsilon}_{\varepsilon}) , \tilde \vartheta \zeta \bigr]\bigr] \Bigr|
\le 
2 \mu_1 \Vert D( \overline{\upsilon}_{\varepsilon} + \upsilon_0 \xi ) \Vert_{L^p(0,T; ( L^p (\Omega))^{3 \times 3} )}^{p-1} \Vert D( \tilde \vartheta \zeta) \Vert_{L^p(0,T; ( L^p (\Omega))^{3 \times 3} )} 
\\
& \displaystyle 
+ 2 \varepsilon \Vert D( \overline{\upsilon}_{\varepsilon} ) \Vert_{L^{p'}(0,T; ( L^{p'} (\Omega) )^{3\times 3} )}^{p'-1} \Vert D( \tilde \vartheta \zeta) \Vert_{L^{p'}(0,T; ( L^{p'} (\Omega) )^{3\times 3} )}
\\
& \displaystyle 
\le 
2 \mu_1 \bigl( {\rm meas}( \Omega) T \bigr)^{\frac{2-p}{p}} 
\Vert D( \overline{\upsilon}_{\varepsilon} + \upsilon_0 \xi ) \Vert_{L^p(0,T; ( L^p (\Omega))^{3 \times 3} )}^{p-1} 
\Vert D( \tilde \vartheta \zeta) \Vert_{L^{p'}(0,T; ( L^{p'} (\Omega))^{3 \times 3} )} 
\\
& \displaystyle 
+ 2 \varepsilon \Vert D( \overline{\upsilon}_{\varepsilon} ) \Vert_{L^{p'}(0,T; ( L^{p'} (\Omega) )^{3\times 3} )}^{p'-1} \Vert D( \tilde \vartheta \zeta) \Vert_{L^{p'}(0,T; ( L^{p'} (\Omega) )^{3\times 3} )} . 
\end{array}
\end{eqnarray*}
Thus
\begin{eqnarray*}
\begin{array}{ll}
& \displaystyle \Bigl|\bigl[\bigl[ \mathcal{A}_{{\bf u}}^{\varepsilon} (\overline{\upsilon}_{\varepsilon}) , \tilde \vartheta \zeta \bigr]\bigr] \Bigr|
\\
&
\le 
2 \mu_1 \bigl( {\rm meas}( \Omega) T \bigr)^{\frac{2-p}{p}} 
\Bigl( \Vert  \overline{\upsilon}_{\varepsilon} \Vert_{L^p(0,T; V_{0.div}^p)}  
+ \Vert \upsilon_0 \xi  \Vert_{L^p(0,T; V^p_{\Gamma_1})} \Bigr)^{p-1} 
\Vert \tilde \vartheta \zeta \Vert_{L^{p'}(0,T; V_{0.div}^{p'})} 
\\
& \displaystyle 
+ 2 \varepsilon \Vert  \overline{\upsilon}_{\varepsilon} \Vert_{L^{p'}(0,T; V_{0.div}^{p'})}^{p'-1} 
\Vert  \tilde \vartheta \zeta \Vert_{L^{p'}(0,T; V_{0.div}^{p'})}
\\
& \displaystyle 
\le 
2 \mu_1 \bigl( {\rm meas}( \Omega) T \bigr)^{\frac{2-p}{p}} 
\Bigl( \Vert  \overline{\upsilon}_{\varepsilon} \Vert_{L^p(0,T; V_{0.div}^p)}  
+ \Vert \upsilon_0 \xi  \Vert_{L^p(0,T; V^p_{\Gamma_1})} \Bigr)^{p-1} 
\Vert \tilde \vartheta \zeta \Vert_{L^{p'}(0,T; V_{0.div}^{p'})} 
\\
& \displaystyle 
+  \varepsilon^{1/p'}  \bigl( \varepsilon^{1/p'} \Vert  \overline{\upsilon}_{\varepsilon} \Vert_{L^{p'}(0,T; V_{0.div}^{p'})} \bigr)^{p'-1} 
\Vert  \tilde \vartheta \zeta \Vert_{L^{p'}(0,T; V_{0.div}^{p'})}.
%
%
\end{array}
\end{eqnarray*}
By using (\ref{203}) and (\ref{es_epsilon1}) we obtain 
\begin{eqnarray*}\label{inega_2}
\begin{array}{ll}
& 
\displaystyle \Bigl|\bigl[\bigl[ \mathcal{A}_{{\bf u}}^{\varepsilon} (\overline{\upsilon}_{\varepsilon}) , \tilde \vartheta \zeta \bigr]\bigr] \Bigr|
\le 
\Bigl( 
 2 \mu_1 \bigl( C + \Vert \upsilon_0 \xi  \Vert_{L^p(0,T; V^p_{\Gamma_1})} \bigr)^{p-1} 
\bigl({\rm meas}(\Omega) T \bigr)^{\frac{2-p}{p}}
\\
& \displaystyle 
+  \varepsilon^{1/p' }  C^{p'-1} 
\Bigr)
\Vert  \tilde \vartheta \zeta \Vert_{L^{p'}(0,T; V_{0.div}^{p'})}.
\end{array}
\end{eqnarray*}
Going back to (\ref{inega_1}) we obtain (\ref{es_derivee_epsilon}).

\bigskip


 Let us prove now  (\ref{es_pression_epsilon}).  We choose  $\overline{\varphi} = \pm\tilde \vartheta \zeta $ with $ \tilde \vartheta \in {\bf W}^{1, p'}_0(\Omega) $ and $\zeta \in \mathcal{D}(0,T)$ in (\ref{3.14}). We obtain 
\begin{eqnarray}\label{toto_epsilon}
 \begin{array}{ll}
\displaystyle \left\langle \int_{\Omega} \pi_{\varepsilon} div(\tilde \vartheta) \, dx, \zeta \right\rangle_{{\mathcal D}'(0,T), {\mathcal D} (0,T)} 
= -\int_{0}^{T} (\overline{\upsilon}_{\varepsilon},\tilde \vartheta )_{_{{\bf L}^2(\Omega)}} \zeta' \,dt
\\
\displaystyle + \int_{0}^{T} \int_{\Omega} \mathcal{F} \bigl( \theta, {\bf u} + \upsilon_0 \xi, D( \overline{\upsilon}_{\varepsilon} + \upsilon_0 \xi) \bigr) : D( \tilde\vartheta) \zeta \, dx dt 
 \\
\displaystyle
+ 2 \varepsilon \int_0^T \int_{\Omega} \bigl| D(\overline{\upsilon}_{\varepsilon} )  \bigr|^{p'-2} D(\overline{\upsilon}_{\varepsilon} )  : D( \tilde\vartheta) \zeta \, dx dt
- \int_{0}^{T}(\overline{f},\tilde \vartheta)_{{\bf L}^2(\Omega)}\zeta\,dt  .
\end{array}
 \end{eqnarray} 
We estimate the right-hand side with the same kind of computations as previously, i.e.
\begin{eqnarray*}\label{Smai_2epsilon}
 \begin{array}{ll}
\displaystyle \left| \left\langle \int_{\Omega} \pi_{\varepsilon} div(\tilde \vartheta) \, dx, \zeta \right\rangle_{{\mathcal D}'(0,T), {\mathcal D} (0,T)} \right|
\le \sqrt{T} \Vert \overline{\upsilon}_{\varepsilon} \Vert_{L^{\infty} (0,T; {\bf L}^2(\Omega))} 
\Vert \tilde \vartheta \zeta' \Vert_{L^{2} (0,T; {\bf L}^2(\Omega))}
\\
\displaystyle
+  2 \mu_1 \bigl( C + \Vert \upsilon_0 \xi  \Vert_{L^p(0,T; V^p_{\Gamma_1})} \bigr)^{p-1} 
\bigl({\rm meas}(\Omega))T \bigr)^{\frac{2-p}{p}}
\Vert  \tilde \vartheta \zeta \Vert_{L^{p'}(0,T; V_{0}^{p'})}
+  \varepsilon^{1/p'} 
C^{p'-1} 
\Vert  \tilde \vartheta \zeta \Vert_{L^{p'}(0,T; V_{0}^{p'})}
\\
\displaystyle 
 +  \| \overline{f} \|_{L^{p'} (0,T; {\bf L}^2(\Omega))} \| \tilde \vartheta \zeta \|_{L^{p} (0,T; {\bf L}^2(\Omega))}
\end{array}
\end{eqnarray*}
and with (\ref{es_epsilon2}) we get
\begin{eqnarray*}\label{Smai_2epsilon}
\begin{array}{ll}
& \displaystyle 
\left| \left\langle \int_{\Omega} \pi_{\varepsilon} div(\tilde \vartheta) \, dx, \zeta \right\rangle_{{\mathcal D}'(0,T), {\mathcal D} (0,T)} \right|
\\
& 
\displaystyle \le \Bigl( 
\sqrt{T} \widetilde C' C
+  2 \mu_1 C_{\infty} T^{1/p'}  \bigl( C + \Vert \upsilon_0 \xi  \Vert_{L^p(0,T; V^p_{\Gamma_1})} \bigr)^{p-1} 
\bigl({\rm meas}(\Omega) T \bigr)^{\frac{2-p}{p}}
\\
& \displaystyle 
+  \varepsilon^{1/p'} C_{\infty} T^{1/p'}  
C^{p'-1} 
 + \widetilde C'  C_{\infty} T^{1/p}  \Vert \overline{f} \Vert_{L^{p'} (0,T; {\bf L}^2(\Omega))} 
  \Bigr)
  \Vert \tilde \vartheta \zeta \Vert_{H^1_0 (0,T; V_{0}^{p'})}
\end{array}
\end{eqnarray*}
where $C_{\infty}$ is the norm of the continuous injection of $H^1(0,T; \mathbb{R})$ into $L^{\infty}(0,T; \mathbb{R})$.

Moreover, for any $p'>1$,  there exists a linear and continuous operator $P_{p'}: L^{p'}_0(\Omega) \to {\bf W}^{1,p'}_0(\Omega)$ such that 
\begin{eqnarray*}
div \bigl( P_{p'} (\varpi) \bigr) = \varpi \qquad \forall \varpi \in L^{p'}_0(\Omega) 
\end{eqnarray*}
(see Corollary 3.1 in \cite{1'}).
 It follows that for any $\varpi \in L^{p'}_0(\Omega) $ and $\zeta \in {\mathcal D}(0,T)$ we have
\begin{eqnarray*}\label{Smai_2epsilon}
 \begin{array}{ll}
& \displaystyle \left| \left\langle \int_{\Omega} \pi_{\varepsilon} \varpi \, dx, \zeta \right\rangle_{{\mathcal D}'(0,T), {\mathcal D} (0,T)} \right|
\\
& 
\displaystyle \le \Vert P_{p'} \Vert_{{\mathcal L}(L^{p'}_0(\Omega), {\bf W}^{1,p'}_0(\Omega))}
\Bigl( 
\sqrt{T} \widetilde C' C
+  2 \mu_1 C_{\infty} T^{1/p'} \bigl( C + \Vert \upsilon_0 \xi  \Vert_{L^p(0,T; V^p_{\Gamma_1})} \bigr)^{p-1} 
\bigl({\rm meas}(\Omega) T \bigr)^{\frac{2-p}{p}}
\\
& \displaystyle
+  \varepsilon^{1/p'} C_{\infty} T^{1/p'} 
 C^{p'-1} 
 + \widetilde C'  C_{\infty} T^{1/p}  \| \overline{f} \|_{L^{p'} (0,T; {\bf L}^2(\Omega))} 
 \Bigr)
  \| \varpi \zeta \|_{H^1_0 (0,T; L^{p'} (\Omega))}.
\end{array}
\end{eqnarray*}
Hence there exists a real number $C'>0$, independent of $\varepsilon$, such that for all $\varepsilon \in (0,1]$  we have 
\begin{eqnarray*}
\displaystyle \left| \left\langle \int_{\Omega} \pi_{\varepsilon} \varpi \, dx, \zeta \right\rangle_{{\mathcal D}'(0,T), {\mathcal D} (0,T)} \right|
\le C' \| \varpi \zeta \|_{H^1_0 (0,T; L^{p'} (\Omega))} \quad \forall \varpi \in L^{p'}_0(\Omega), \  \forall \zeta \in {\mathcal D}(0,T).
\end{eqnarray*}

Furthermore, for any $\varpi^* \in L^{p'}(\Omega)$, we may define $\varpi \in L^{p'}_0(\Omega)$ by 
\begin{eqnarray*}
\varpi=\varpi^* -\frac{1}{{\rm meas}( \Omega) }\int_{\Omega} \varpi^* \,dx.
\end{eqnarray*}
We have $\Vert \varpi \Vert_{L^{p'}(\Omega)} \le 2 \Vert \varpi^* \Vert_{L^{p'}(\Omega)}$ and since $\pi_{\varepsilon} \in H^{-1} \bigl(0,T; L^p_0(\Omega) \bigr)$ we have
\begin{eqnarray*}
&& \displaystyle \left\langle \int_{\Omega} \pi_{\varepsilon} \left(\varpi^*-\frac{1}{{\rm meas}( \Omega)} \int_{\Omega}\varpi^* \,dx  \right) \, dx , \zeta \right\rangle_{{\mathcal D}'(0,T), {\mathcal D}(0,T)} \\
&& 
\displaystyle = \left\langle \int_{\Omega} \pi_{\varepsilon} \varpi^*   \, dx , \zeta \right\rangle_{{\mathcal D}'(0,T), {\mathcal D}(0,T)}
- \frac{1}{{\rm meas}( \Omega)} \left( \int_{\Omega}\varpi^* \,dx ) \right)  \left\langle \int_{\Omega} \pi_{\varepsilon}  \, dx , \zeta \right\rangle_{{\mathcal D}'(0,T), {\mathcal D}(0,T)}
\\
&& \displaystyle = \left\langle \int_{\Omega} \pi_{\varepsilon} \varpi^*   \, dx , \zeta \right\rangle_{{\mathcal D}'(0,T), {\mathcal D}(0,T)}.
\end{eqnarray*}
It follows that 
\begin{eqnarray*}
\begin{array}{ll}
\displaystyle \left| \left\langle \int_{\Omega} \pi_{\varepsilon} \varpi^* \, dx, \zeta \right\rangle_{{\mathcal D}'(0,T), {\mathcal D} (0,T)} \right|
= \displaystyle \left| \left\langle \int_{\Omega} \pi_{\varepsilon} \varpi \, dx, \zeta \right\rangle_{{\mathcal D}'(0,T), {\mathcal D} (0,T)} \right|
\\
\displaystyle \le 2 C' \| \varpi^* \zeta \|_{H^1_0 (0,T; L^{p'} (\Omega))} \quad \forall \varpi^* \in L^{p'} (\Omega), \  \forall \zeta \in {\mathcal D}(0,T).
\end{array}
\end{eqnarray*}
Finally we may conclude by using   the density of  $\mathcal{D}(0,T)\otimes L^{p'}(\Omega)$ into  $H_{0}^{1} \bigl(0,T;L^{p'}(\Omega) \bigr)$.    
\end{proof}

\begin{remark}
Let us emphasize once again that these estimates are independent of ${\bf u}$.
\end{remark}


\subsection{Existence and uniqueness result for $(P_{{\bf u}})$ } \label{convergence}

\bigskip




With the previous estimates  we infer that, by possibly extracting 
  a subsequence  still denoted  $(\overline{\upsilon}_{\varepsilon}, \pi_{\varepsilon})_{\varepsilon >0}$, there exists $\overline{\upsilon} \in L^{p}(0,T;V^p_{0.div}) \cap L^{\infty} \bigl(0,T;{\bf L}^2(\Omega) \bigr)$ such that $\displaystyle \frac{\partial \overline{\upsilon}}{\partial t} \in L^p \bigl( 0,T; (V_{0.div}^{p'})'  \bigr)$, and $ \pi \in H^{-1} \bigl( 0,T; L^{p}_0(\Omega) \bigr)$ such that 
 \begin{eqnarray}\label{conv_epsilon1}
  \overline{\upsilon}_{\varepsilon}\rightharpoonup  \overline{\upsilon} \quad \mbox{\rm weakly in $ L^{p}(0,T;V^p_{0.div})$} 
  \end{eqnarray}
   \begin{eqnarray}\label{conv_epsilon1bis}
  \overline{\upsilon}_{\varepsilon}\rightharpoonup  \overline{\upsilon} \quad \mbox{\rm weakly in $ L^{r} \bigl(0,T; {\bf L}^2 (\Omega) \bigr)$ for any $r>1$ and weakly$*$ in $L^{\infty} \bigl(0,T;{\bf L}^2(\Omega) \bigr)$}
  \end{eqnarray}
%
  \begin{eqnarray}\label{conv_epsilon4}
  \displaystyle \frac{\partial \overline{\upsilon}_{\varepsilon}}{\partial t} \rightharpoonup  \frac{\partial \overline{\upsilon}}{\partial t}  \quad \mbox{\rm weakly in $L^{p} \bigl(0,T;(V^{p'}_{0.div})' \bigr)$}
  \end{eqnarray}
  and 
  \begin{eqnarray}\label{conv_pression}
\pi _{\varepsilon}\rightharpoonup  \pi\quad \mbox{\rm weakly$*$ in  $H^{-1} \bigl( 0,T;L_{0}^{p}(\Omega) \bigr)$.}
\end{eqnarray}

Owing that $V_{0.div}^{r} \subset {\bf L}^2(\Omega) \subset (V_{0.div}^{r})'$ for any $r \ge 6/5$ with a compact  injection of $V_{0.div}^{r}$ into ${\bf L}^2(\Omega)$ whenever $r>6/5$,  we infer that $V_{0.div}^{p}$ is compactly embedded into $\bigl( V_{0.div}^{p'} \bigr)' $. Thus 
we may apply Aubin's lemma and, by possibly extracting 
  another subsequence  still denoted  $(\overline{\upsilon}_{\varepsilon}, \pi_{\varepsilon})_{\varepsilon >0}$, we have 
   \begin{eqnarray}\label{conv_epsilon10}
  \overline{\upsilon}_{\varepsilon} \longrightarrow \overline{\upsilon}  \quad \mbox{\rm strongly in $ L^{p} \bigl(0,T; (V_{0.div}^{p'})' \bigr)$.}
  \end{eqnarray}
   With  Simon's lemma we have also
  \begin{eqnarray}\label{conv_epsilon8}
  \overline{\upsilon}_{\varepsilon} \longrightarrow  \overline{\upsilon}  \quad \mbox{strongly in  $C \bigl([0,T]; \widetilde{H} \bigr)$}
  \end{eqnarray}
  where  $\widetilde{H}$ is a Banach space such that    $ {\bf L}^2(\Omega)\subset \widetilde{H}\subset( V_{0.div}^{p'})'$ and the embedding of  $ {\bf L}^2(\Omega)$ in  $\widetilde{H}$   is compact. 
It follows that 
\begin{eqnarray*}
\overline{\upsilon}_{\varepsilon} (t=0, \cdot) \longrightarrow \overline{\upsilon} (t=0, \cdot) = 0.
\end{eqnarray*}
Moreover the sequence $\bigl( \overline{\upsilon}_{\varepsilon} (T) \bigr)_{\varepsilon >0}$ is bounded in ${\bf L}^2(\Omega)$. The compact embedding of ${\bf L}^2(\Omega)$ into $\widetilde{H}$ combined with (\ref{conv_epsilon8}) implies that,
  by possibly extracting another subsequence  still denoted  $(\overline{\upsilon}_{\varepsilon}, \pi_{\varepsilon})_{\varepsilon >0}$, we have 
  \begin{eqnarray} \label{conv_epsilon8bis}
  \overline{\upsilon}_{\varepsilon} (T) \rightharpoonup \overline{\upsilon} (T) \quad \hbox{\rm weakly in ${\bf L}^2(\Omega)$ and strongly in $\widetilde H$.}
  \end{eqnarray}



Furthermore

\begin{lemma} \label{trace}
We have 
\begin{eqnarray} \label{conv_trace}
\gamma_p ( \overline{\upsilon}_{\varepsilon}) \longrightarrow \gamma_p ( \overline{\upsilon} ) \quad \mbox{\rm strongly in $L^p \bigl( 0,T; {\bf L}^p (\partial \Omega) \bigr)$}
\end{eqnarray}
where $\gamma_p$ is the trace operator form ${\bf W}^{1,p}(\Omega)$ into ${\bf L}^p (\partial \Omega) $.
\end{lemma}

\begin{proof}
Let us prove that for any  $ \eta>0$, there exists $c_{ \eta}>0$ such that
  \begin{eqnarray}\label{AuTema}
  \Vert \gamma_p(v) \Vert_{{\bf L}^{p}(\partial\Omega)}\leq  \eta  \Vert   v  \Vert_{V_{0.div}^p} 
  +c_{ \eta}\Vert  v \Vert_{(V_{0.div}^{p'})' } \quad \forall v \in V_{0.div}^p.
  \end{eqnarray}
Indeed if (\ref{AuTema}) is not true, there exists $\eta >0$ and a sequence $(v_m)_{m \ge 1}$ of $V_{0.div}^p$ such that
\begin{eqnarray*}
 \Vert \gamma_p (v_m) \Vert_{{\bf L}^{p}(\partial\Omega)} >  \eta  \Vert   v_m  \Vert_{V_{0.div}^p} 
  + m \, \Vert  v_m \Vert_{(V_{0.div}^{p'})' } \quad \forall m \ge 1.
  \end{eqnarray*}
It follows that $\Vert \gamma_p (v_m) \Vert_{{\bf L}^{p}(\partial\Omega)} >0$ for all $m \ge 1$. Since $\gamma_p \in {\mathcal L}_c( {\bf W}^{1,p}(\Omega), {\bf L}^p (\partial \Omega))$ and $V_{0.div}^p$ is a subspace of ${\bf W}^{1,p}(\Omega)$ we infer that $\Vert   v_m  \Vert_{V_{0.div}^p}  \not=0$ for all $m \ge 1$ and we may define $\displaystyle w_m = \frac{1}{\Vert   v_m  \Vert_{V_{0.div}^p} } v_m$ for all $m \ge 1$.

Thus $\Vert   w_m  \Vert_{V_{0.div}^p} =1$ for all $m \ge 1$ and the sequence $\bigl( \Vert \gamma_p(w_m) \Vert_{{\bf L}^p(\partial \Omega)} \bigr)_{m \ge 1}$ is bounded. Owing that  
\begin{eqnarray}  \label{AuTemabis}
 \Vert \gamma_p (w_m) \Vert_{{\bf L}^{p}(\partial\Omega)} >  \eta   
  + m \Vert  w_m \Vert_{(V_{0.div}^{p'})' } \ge \eta >0 \quad \forall m \ge 1
  \end{eqnarray}
  we infer  that 
  \begin{eqnarray}  \label{AuTemater}
  w_m \longrightarrow 0 \quad \hbox{\rm strongly in $(V_{0.div}^{p'})'$.}
  \end{eqnarray}
Moreover by possibly extracting a subsequence still denoted $( w_m)_{m \ge 1}$, we obtain that there exists $w_* \in V_{0.div}^p$ such that
\begin{eqnarray*}
w_m \rightharpoonup w_* \quad \hbox{\rm weakly in $V_{0.div}^p$}
\end{eqnarray*}
and with (\ref{AuTemater}) we get $w_*=0$. 
 Recalling that the trace operator $\gamma_p$ is compact from ${\bf W}^{1,p} (\Omega)$ into ${\bf L}^p(\partial \Omega)$ 
 it follows  that
 \begin{eqnarray*}
\gamma_p (w_m) \longrightarrow \gamma_p (w_*) =0 \quad \hbox{\rm strongly in $L^p(\partial \Omega)$.}
\end{eqnarray*}
which gives a contradiction with (\ref{AuTemabis}).





  By using (\ref{AuTema}) and (\ref{203}) we obtain that, for any $\eta >0$ there exists $c_{\eta}>0$ such that
   \begin{eqnarray*}
   \begin{array}{ll}
\displaystyle   \Vert \gamma_p (\overline{\upsilon}_{\varepsilon}) - \gamma_p (\overline{\upsilon}) \Vert_{L^p(0,T; {\bf L}^{p}(\partial\Omega)) }
  =  \Vert \gamma_p (\overline{\upsilon}_{\varepsilon} - \overline{\upsilon}) \Vert_{L^p( 0,T; {\bf L}^{p}(\partial\Omega))}
  \\
  \displaystyle 
  \leq  \eta  \Vert   \overline{\upsilon}_{\varepsilon} - \overline{\upsilon}  \Vert_{L^p(0,T; V_{0.div}^p)} 
  +c_{ \eta}\Vert  \overline{\upsilon}_{\varepsilon} - \overline{\upsilon}  \Vert_{L^p(0,T; (V_{0.div}^{p'})') } 
  \\
  \displaystyle 
  \leq  2 \eta C +  c_{\eta} \Vert  \overline{\upsilon}_{\varepsilon} - \overline{\upsilon}  \Vert_{L^p(0,T; (V_{0.div}^{p'})') }  \quad \forall \varepsilon >0
  \end{array}
  \end{eqnarray*}
 and (\ref{conv_epsilon10}) allows us to conclude.  
\end{proof}


\bigskip

For the sake of notational simplicity  we will identify the functions and their traces on $\partial \Omega$ in the rest of the paper.


Next we can prove that the limit pressure  $\pi$ and  $\displaystyle \frac{ \partial \overline{\upsilon}}{\partial t}$ satisfy better regularity properties.

 \begin{proposition} \label{coro 1.1}
 Let ${\bf u} \in L^ {p} \bigl(0,T; {\bf L}^p (\Omega) \bigr)$. Let     $f\in  L^ {p'} \bigl(0,T; {\bf L}^2(\Omega) \bigr)$, $k \in L^{p'} \bigl(0,T;L_{+}^{p'}(\Gamma_0) \bigr)$, $\mu$ satisfying (\ref{rop})-(\ref{mlo}), 
 $\theta \in  L^{\tilde q}\bigl(0,T;L^{\tilde p}(\Omega)\bigr)$ with $\tilde q \ge 1$ and $\tilde p \ge 1$, $s\in L^{ p} \bigl(0,T; {\bf L}^{p}(\Gamma_0)\bigr)$, $\xi\in W^{1,p'}(0,T)$ satisfying (\ref{xi})
and  $\upsilon_{0}\in \textbf{W}^{1,p }(\Omega)$ satisfying  (\ref{coco}). 
Then      $ \pi\in H^{-1} \bigl( 0,T;L_{0}^{p'}(\Omega) \bigr)$ and  $\displaystyle \frac{ \partial \overline{\upsilon}}{\partial t}\in   L^{p'} \bigl( 0,T;(V_{0.div}^p)' \bigr)$. 
 \end{proposition}
 
 \begin{proof}
Let us prove that   $ \pi\in H^{-1} \bigl( 0,T;L_{0}^{p'}(\Omega) \bigr)$. Indeed, for all $\varepsilon \in (0,1]$, for all $\tilde \vartheta \in {\bf W}^{1, p'}_0(\Omega)$ and for all $\zeta \in {\mathcal D}(0,T)$ we have (\ref{toto_epsilon}) and thus
\begin{eqnarray*}\label{Smai_2epsilon}
\begin{array}{ll}
& \displaystyle 
\left| \left\langle \int_{\Omega} \pi_{\varepsilon} div(\tilde \vartheta) \, dx, \zeta \right\rangle_{{\mathcal D}'(0,T), {\mathcal D} (0,T)} \right|
\le \Bigl( 
\sqrt{T} C \widetilde C
+  2 \mu_1 C_{\infty} T^{1/p'} \bigl( C + \Vert \upsilon_0 \xi  \Vert_{L^p(0,T; V^p_{\Gamma_1})} \bigr)^{p-1} 
\\
& \displaystyle 
+ \widetilde C \Vert \overline{f} \Vert_{L^{p'} (0,T; {\bf L}^2(\Omega))} 
 C_{\infty} T^{1/p}  \Bigr)
  \Vert \tilde \vartheta \zeta \Vert_{H^1_0 (0,T; V_{0}^{p})}
+   \varepsilon^{1/p'} C_{\infty} T^{1/p'} 
C^{p'-1} 
  \Vert \tilde \vartheta \zeta \Vert_{H^1_0 (0,T; V_{0}^{p'})}
\end{array}
\end{eqnarray*}
where $C>0$ is  the constant introduced  in Proposition \ref{apriori-estimates1} and $\widetilde C$ and $C_{\infty}$ are the norms of the continuous injections of $V_0^p$ into ${\bf L}^2(\Omega)$ and $H^1(0,T; \mathbb{R})$ into $L^{\infty} (0,T; \mathbb{R})$ respectively.
We may pass to the limit as $\varepsilon$ tends to zero in the previous inequality and we get
\begin{eqnarray*}
\begin{array}{ll}
& \displaystyle 
\left| \left\langle \int_{\Omega}  \pi div(\tilde \vartheta) \, dx, \zeta \right\rangle_{{\mathcal D}'(0,T), {\mathcal D} (0,T)} \right|
\le \Bigl( 
\sqrt{T} \widetilde C C
+  2 \mu_1 C_{\infty} T^{1/p'}  \bigl( C + \Vert \upsilon_0 \xi  \Vert_{L^p(0,T; V^p_{\Gamma_1})} \bigr)^{p-1} 
\\
& \displaystyle 
 + \widetilde C \Vert \overline{f} \Vert_{L^{p'} (0,T; {\bf L}^2(\Omega))} 
 C_{\infty} T^{1/p}  \Bigr)
  \Vert \tilde \vartheta \zeta \Vert_{H^1_0 (0,T; V_{0}^{p})}.
\end{array}
\end{eqnarray*}
By using Green's formula we infer that 
\begin{eqnarray*}
\begin{array}{ll}
& \displaystyle 
\left| \left\langle \int_{\Omega} \bigl( \nabla   \pi , \tilde \vartheta \bigr) \, dx, \zeta \right\rangle_{{\mathcal D}'(0,T), {\mathcal D} (0,T)} \right|
\le \Bigl( 
\sqrt{T} \widetilde C C
+  2 \mu_1 C_{\infty} T^{1/p'}  \bigl( C + \Vert \upsilon_0 \xi  \Vert_{L^p(0,T; V^p_{\Gamma_1})} \bigr)^{p-1} 
\\
& \displaystyle 
 + \widetilde C \Vert \overline{f} \Vert_{L^{p'} (0,T; {\bf L}^2(\Omega))} 
 C_{\infty} T^{1/p}  \Bigr)
  \Vert \tilde \vartheta \zeta \Vert_{H^1_0 (0,T; {\bf W}_{0}^{1,p} (\Omega) )}
\end{array}
\end{eqnarray*}
for all $\tilde \vartheta \in \bigl({\mathcal D}(\Omega) \bigr)^3$ and for all $\zeta \in {\mathcal D}(0,T)$, where $(\cdot, \cdot)$ denotes the Euclidean inner product of $\mathbb{R}^3$.

  With the density of  $\mathcal{D}(0,T)\otimes {\bf W}_{0}^{1,p}(\Omega)$ into  $H_{0}^{1} \bigl(0,T;{\bf W}_{0}^{1,p}(\Omega) \bigr)$,     we conclude that  $\nabla \pi \in H^{-1}\bigl(0,T;{\bf W} ^{-1,p'}(\Omega)  \bigr)$. Then    the properties of the gradient operator   (\cite{1'})   imply that  $ \pi \in H^{-1} \bigl(0,T; L^{p'}(\Omega) \bigr)$.  Owing that  $ \pi \in H^{-1} \bigl(0,T; L^{p}_0(\Omega) \bigr)$ we obtain finally that $ \pi \in H^{-1} \bigl(0,T; L^{p'}_0(\Omega) \bigr)$.

\bigskip

Let us  prove now that $\displaystyle  \frac{\partial \overline{\upsilon}}{\partial t} \in L^{ p'} \bigl( 0,T;(V_{0.div}^p)' \bigr)$. We choose  $\overline{\varphi} = \pm \tilde\vartheta \zeta$   with  $\tilde\vartheta\in V^{p'}_{0}$ and $\zeta\in \mathcal{D}(0,T)$ in (\ref{3.14}) and  we obtain
\begin{eqnarray*}
\begin{array}{ll}
\displaystyle{ \left\langle \frac{\partial}{\partial t} (\overline{\upsilon}_{\varepsilon} , \pm \tilde \vartheta )_{{\bf L}^2(\Omega)} ,  \zeta \right\rangle_{ {\mathcal D}'(0,T), {\mathcal D}(0,T)} 
+ \int_{0}^{T} \int_{\Omega} \mathcal{F} \bigl( \theta, {\bf u} + \upsilon_0 \xi, D( \overline{\upsilon}_{\varepsilon} + \upsilon_0 \xi) \bigr) : D( \pm \tilde\vartheta) \zeta \, dx dt } 
 \\
\displaystyle{
+ 2 \varepsilon \int_0^T \int_{\Omega} \bigl| D(\overline{\upsilon}_{\varepsilon} )  \bigr|^{p'-2} D(\overline{\upsilon}_{\varepsilon} )  : D(\pm  \tilde\vartheta) \zeta \, dx dt
- \left\langle \int_{\Omega} \pi_{\varepsilon} div(\pm \tilde \vartheta) \, dx, \zeta \right\rangle_{{\mathcal D}'(0,T), {\mathcal D}(0,T)} }
\\
\displaystyle 
+J(\overline{\upsilon}_{\varepsilon} \pm \tilde \vartheta \zeta)
-J(\overline{\upsilon}_{\varepsilon} )
\geq  \int_0^T ( \overline{f},\pm\tilde \vartheta \zeta)_{{\bf L}^2(\Omega)} \, dt .
\end{array}
\end{eqnarray*}
By recalling that $\overline{f} \in L^{p'} \bigl(0,T; {\bf L}^2(\Omega) \bigr) $ and $k \in L^{p'} \bigl(0,T; { L}^{p'}(\Gamma_0) \bigr) $, we perform  the same kind of computations as in Proposition \ref{apriori-estimates2} and we get
\begin{eqnarray*}
\begin{array}{ll}
& \displaystyle{ \left| \int_0^T ( \overline{\upsilon}_{\varepsilon}  , \tilde \vartheta)_{{\bf L}^2(\Omega)} \zeta' \, dt \right| } 
 \le 
\| \gamma_{p} \|_{{\mathcal L} ({\bf W}^{1,p} (\Omega) , {\bf L}^{p} (\partial \Omega))} \| k\|_{L^{p'}(0,T; L^{p'}(\Gamma_0))} \| \tilde \vartheta \zeta \|_{L^{p} (0,T; V_{0}^{p})} 
\\
& \displaystyle 
+ \widetilde C \| \overline{f} \|_{L^{p'} (0,T; {\bf L}^2(\Omega))} \| \tilde \vartheta \zeta \|_{L^{p} (0,T; V_{0}^{p})}
+ \left|  \int_{0}^{T} \int_{\Omega} \mathcal{F} \bigl( \theta, {\bf u} + \upsilon_0 \xi, D( \overline{\upsilon}_{\varepsilon} + \upsilon_0 \xi) \bigr) : D(  \tilde\vartheta) \zeta \, dx dt \right| 
\\
&
\displaystyle 
+ 
2 \varepsilon \left|  \int_0^T \int_{\Omega} \bigl| D(\overline{\upsilon}_{\varepsilon} )  \bigr|^{p'-2} D(\overline{\upsilon}_{\varepsilon} )  : D(  \tilde\vartheta) \zeta \, dx dt \right|
+ \,  \left|  \left\langle \int_{\Omega} \pi_{\varepsilon} div( \tilde \vartheta) \, dx, \zeta \right\rangle_{{\mathcal D}'(0,T), {\mathcal D}(0,T)}  \right|
\end{array}
\end{eqnarray*}
where $\widetilde C$ denotes the norm of the continuous injection of $V_{0}^{p}$ into ${\bf L}^2(\Omega)$ and 
\begin{eqnarray*}
\begin{array}{ll}
& \displaystyle
\left|  \int_{0}^{T} \int_{\Omega} \mathcal{F} \bigl( \theta, {\bf u} + \upsilon_0 \xi, D( \overline{\upsilon}_{\varepsilon} + \upsilon_0 \xi) \bigr) : D(  \tilde\vartheta) \zeta \, dx dt \right|
+ 
2 \varepsilon \left|  \int_0^T \int_{\Omega} \bigl| D(\overline{\upsilon}_{\varepsilon} )  \bigr|^{p'-2} 
D(\overline{\upsilon}_{\varepsilon})  : D(  \tilde\vartheta) \zeta \, dx dt \right|
\\
& 
\displaystyle 
\le 
2 \mu_1 \Bigl( \Vert  \overline{\upsilon}_{\varepsilon} \Vert_{L^p(0,T; V_{0.div}^p)}  
+ \Vert \upsilon_0 \xi  \Vert_{L^p(0,T; V^p_{\Gamma_1})} \Bigr)^{p-1} 
\Vert \tilde \vartheta \zeta \Vert_{L^{p}(0,T; V_{0}^{p})} 
\\
& \displaystyle 
+ 
\varepsilon^{1/p'} 
\bigl( \varepsilon^{1/p'} \Vert  \overline{\upsilon}_{\varepsilon} \Vert_{L^{p'}(0,T; V_{0.div}^{p'})} \bigr)^{p'-1} 
\Vert  \tilde \vartheta \zeta \Vert_{L^{p'}(0,T; V_{0}^{p'})}
\\
& \displaystyle \le 
2 \mu_1 \Bigl( C 
+ \Vert \upsilon_0 \xi  \Vert_{L^p(0,T; V^p_{\Gamma_1})} \Bigr)^{p-1} 
\Vert \tilde \vartheta \zeta \Vert_{L^{p}(0,T; V_{0}^{p})} 
+ 
\varepsilon^{1/p' }
 C^{p'-1} 
\Vert  \tilde \vartheta \zeta \Vert_{L^{p'}(0,T; V_{0}^{p'})}.
\end{array}
\end{eqnarray*}
We pass to the limit as $\varepsilon$ tends to zero. We obtain
\begin{eqnarray*}
\begin{array}{ll}
& \displaystyle{ \left| \int_0^T ( \overline{\upsilon}  , \tilde \vartheta)_{{\bf L}^2(\Omega)} \zeta' \, dt \right| } 
 \le 
\| \gamma_{p} \|_{{\mathcal L} ({\bf W}^{1,p} (\Omega) , {\bf L}^{p} (\partial \Omega))} \| k\|_{L^{p'}(0,T; L^{p'}(\Gamma_0))} \| \tilde \vartheta \zeta \|_{L^{p} (0,T; V_{0}^{p})} 
\\
& 
\displaystyle
+ \widetilde C \| \overline{f} \|_{L^{p'} (0,T; {\bf L}^2(\Omega))} \| \tilde \vartheta \zeta \|_{L^{p} (0,T; V_{0}^{p})}
+ 
2 \mu_1 \Bigl( C
+ \Vert \upsilon_0 \xi  \Vert_{L^p(0,T; V^p_{\Gamma_1})} \Bigr)^{p-1} 
\Vert \tilde \vartheta \zeta \Vert_{L^{p}(0,T; V_{0}^{p})} 
\\
& 
\displaystyle 
+  \left|  \left\langle \int_{\Omega}  \pi div( \tilde \vartheta) \, dx, \zeta \right\rangle_{{\mathcal D}'(0,T), {\mathcal D}(0,T)}  \right|.
\end{array}
\end{eqnarray*}
Moreover we know that $ \pi \in H^{-1} \bigl( 0,T; L^{p'}_0(\Omega) \bigr)$. It follows that 
\begin{eqnarray*}
\left|  \left\langle \int_{\Omega}  \pi div( \tilde \vartheta) \, dx, \zeta \right\rangle_{{\mathcal D}'(0,T), {\mathcal D}(0,T)}  \right| 
\le  \Vert  \pi \Vert_{H^{-1}(0,T; L^{p'}(\Omega))} \Vert div (\tilde \vartheta) \zeta \Vert_{H^1_0(0,T; L^p(\Omega))}
\end{eqnarray*}
and there exists a real number $C''>0$ such that
\begin{eqnarray} \label{estim_vitesse}
\begin{array}{ll}
&  \displaystyle{ \left| \int_0^T ( \overline{\upsilon}  , \tilde \vartheta)_{{\bf L}^2(\Omega)} \zeta' \, dt \right| } 
 \le C''  \Vert \tilde \vartheta \zeta \Vert_{L^p(0,T; V_0^p)}
+  \Vert  \pi \Vert_{H^{-1}(0,T; L^{p'}(\Omega))} \Vert div (\tilde \vartheta) \zeta \Vert_{H^1_0(0,T; L^p(\Omega))}
\\
& \displaystyle 
\le   C''  \Vert \tilde \vartheta \zeta \Vert_{L^p(0,T; V_0^p)}
 + 3  \Vert  \pi \Vert_{H^{-1}(0,T; L^{p'}(\Omega))} \Vert \tilde \vartheta \zeta \Vert_{H^1_0(0,T; V_0^p)}
\end{array}
\end{eqnarray}
for all $\tilde \vartheta \in V_0^{p'}$ and $\zeta \in {\mathcal D}(0,T)$.

\bigskip
 
Furthermore  $V_{0}^{p'}$ is  dense into  $V_{0}^p$. Indeed, by definition of $V_{0}^p$ we have   $\varphi=(\varphi_{1},\varphi_{2},\varphi_{3})\in V_{0}^p$ if and only if $\varphi=(\varphi_{1},\varphi_{2},\varphi_{3})\in W^{1,p}_{\Gamma_{1}\cup\Gamma_{L}}(\Omega) \times W^{1,p}_{\Gamma_{1}\cup\Gamma_{L}}(\Omega) \times W^{1,p}_0(\Omega)$ where 
\begin{eqnarray*}
W^{1,p}_{\Gamma_{1}\cup\Gamma_{L}}(\Omega) =\bigl\{ w\in  {W}^{1,p}(\Omega)  ; \   w=0 \  \mbox{on}\ \Gamma_{1}\cup\Gamma_{L} \bigr\}.
\end{eqnarray*}
Let us define 
\begin{eqnarray*}
W=\bigl\{ w \in  \mathcal{D}(\overline{\Omega}) ; \   w =0 \  \mbox{on} \  \Gamma_{1}\cup\Gamma_{L} \bigr\}
\end{eqnarray*}
and let $Q_+$ and $Q$ be given as
 \begin{eqnarray*}
 \begin{array}{ll}
 & \displaystyle   Q_ {+}=
\Omega =\bigl\{ (x',x_{3}) \in \mathbb{R}^{3}; \   x'=(x_{1},x_{2}) \in \omega \  \mbox{and} \   0<x_{3}< h(x')\bigr\}, \\
 & \displaystyle   
Q =\bigl\{ (x',x_{3}) \in \mathbb{R}^{3}; \  x'=(x_{1},x_{2})\in \omega \ \mbox{and} \  |x_{3} | < h(x')  \bigr\}.
\end{array}
\end{eqnarray*}


For any $u \in W^{1,p} (\Omega)$ we build the extension $u^*$ of $u$ to $Q$ by reflexion, i.e.
\begin{eqnarray*}
u^{*}(x',x_{3})= 
\left\{
\begin{array}{ll}
u(x',x_{3}) & \textrm{if $x_{3}>0$} ,\\
u(x',-x_{3}) & \textrm{if $x_{3}<0$}.\\ 
\end{array} 
\right.
\end{eqnarray*}
With classical results we know that $u^* \in W^{1,p}(Q)$. 
Hence for all $u \in W^{1,p}_{\Gamma_{1}\cup\Gamma_{L}}(\Omega)$ we have $u^* \in W^{1,p}_0(Q)$. Since ${\mathcal D}(Q)$ is dense into $W^{1,p}_0(Q)$ there exists a sequence $(v_n)_{n \ge 1}$ such that $v_n \in {\mathcal D}(Q)$ for all $n \ge 1$ and 
\begin{eqnarray*}
v_n \longrightarrow u^* \quad \hbox{\rm strongly in $W^{1,p}_0(Q)$.}
\end{eqnarray*}
It follows that $u_n = {v_n}_{|Q_+} \in W$
for all $n \ge 1$ and 
\begin{eqnarray*}
u_n \longrightarrow u \quad \hbox{\rm strongly in $W^{1,p}_{\Gamma_{1}\cup\Gamma_{L}}(\Omega)$.}
\end{eqnarray*}
We infer that $W$ is dense into $W^{1,p}_{\Gamma_{1}\cup\Gamma_{L}}(\Omega)$ and $W \times W \times {\mathcal D} (\Omega)$ is dense into $V^p_0$. Since $W \times W \times {\mathcal D}(\Omega) \subset V^{p'}_0 \subset V^p_0$ we obtain the announced density result.

\bigskip

Going back to (\ref{estim_vitesse}) we obtain
\begin{eqnarray*} 
 \displaystyle{ \left| \int_0^T ( \overline{\upsilon}  , \tilde \vartheta)_{{\bf L}^2(\Omega)} \zeta' \, dt \right| }
  \le C''  \Vert \tilde \vartheta \zeta \Vert_{L^p(0,T; V_0^p)}
 +   \Vert  \pi \Vert_{H^{-1}(0,T; L^{p'}(\Omega))} \Vert div( \tilde \vartheta) \zeta \Vert_{H^1_0(0,T; L^p (\Omega) )}
\end{eqnarray*}
for all $\tilde \vartheta \in V_0^p$ and for all $\zeta \in {\mathcal D}(0,T)$. Moreover if $\tilde \vartheta \in V_{0.div}^p$ we get 
\begin{eqnarray*} 
 \displaystyle{ \left| \int_0^T ( \overline{\upsilon}  , \tilde \vartheta)_{{\bf L}^2(\Omega)} \zeta' \, dt \right| } \le C''  \Vert \tilde \vartheta \zeta \Vert_{L^p(0,T; V_0^p)} \quad \forall \tilde \vartheta \in V_{0.div}^p, \ \forall \zeta \in {\mathcal D}(0,T).
\end{eqnarray*}
With the density of  $\mathcal{D}(0,T)\otimes V_{0.div}^p$ in  $L^{p}(0,T;V_{0.div}^p)$ we may conclude  that  $\displaystyle \frac{\partial \overline{\upsilon}}{\partial t} \in    L^{ p'} \bigl( 0,T;(V_{0.div}^p)' \bigr)$. 
\end{proof}

\bigskip

Let us recall the definition of the space ${\mathcal Y}$ already introduced in Section \ref{framework}
\begin{eqnarray*}
{\mathcal Y} = \bigl\{ \psi \in {\bf L}^2(\Omega); \ div(\psi) \in L^2(\Omega) \bigr\}
\end{eqnarray*}
 endowed with its canonical norm
 \begin{eqnarray*}
 \Vert \psi \Vert_{\mathcal Y} = \Bigl( \Vert \psi \Vert^2_{{\bf L}^2(\Omega)} + \Vert div(\psi) \Vert^2_{{ L}^2(\Omega)} \Bigr)^{1/2} \qquad \forall \psi \in {\mathcal Y} 
 \end{eqnarray*}
 and  its subspace $H$ given by 
 $H= \bigl\{ \psi \in {\bf L}^2(\Omega); \ div(\psi) = 0 \ {\rm in } \ \Omega, \ \psi \cdot n = 0 \ {\rm on } \ \partial \Omega \bigr\}$.
For any $p \ge 6/5$ the  embedding of $V^p_{0.div}$ into $H$ is continuous and dense (\cite{3'}). Hence we may consider the Gelfand triplet $V^p_{0.div} \subset H \equiv H' \subset  (V^p_{0.div} )'$ and we obtain that $\overline{\upsilon} \in C \bigl( [0,T]; H \bigr)$.






 \begin{theorem} \label{existence_Pu}
 Let ${\bf u} \in L^ {p} \bigl(0,T; {\bf L}^p (\Omega) \bigr)$. Let     $f\in  L^ {p'} \bigl(0,T; {\bf L}^2(\Omega) \bigr)$, $k \in L^{p'} \bigl(0,T;L_{+}^{p'}(\Gamma_0) \bigr)$, $\mu$ satisfying (\ref{rop})-(\ref{mlo}), 
 $\theta \in  L^{\tilde q}\bigl(0,T;L^{\tilde p}(\Omega)\bigr)$ with $\tilde q \ge 1$ and $\tilde p \ge 1$, $s\in L^{ p} \bigl(0,T; {\bf L}^{p}(\Gamma_0)\bigr)$, $\xi\in W^{1,p'}(0,T)$ satisfying (\ref{xi})
and  $\upsilon_{0}\in \textbf{W}^{1,p }(\Omega)$ satisfying  (\ref{coco}). 
Then     problem (P${}_{\bf u}$) admits an unique solution. 
 \end{theorem}

  \begin{proof}
  Let us introduce the operator  $\mathcal{A}_{{\bf u}} : L^{p} \bigl(0,T;V^p_{0} \bigr)\rightarrow   L^{p'} \bigl(0,T; (V^p_{0} )' \bigr) $  defined by
\begin{eqnarray*}
\bigl[\mathcal{A}_{{\bf u}}  (u) ,\overline{\varphi} \bigr]
= \int_0^T \int_{\Omega} \mathcal{F} \bigl( \theta, {\bf u} + \upsilon_0 \xi, D( u + \upsilon_0 \xi) \bigr) : D( \overline{\varphi}) \, dx dt 
\qquad \forall \overline{\varphi}\in L^{p} \bigl(0,T;V^p_{0} \bigr) 
\end{eqnarray*}
where   $[.,.]$ denotes  the duality product
between  $L^{p } \bigl(0,T;V^p_{0} \bigr)$ and  its dual $ \bigl( L^{p} \bigl(0,T;V^p_{0})  \bigr)'$.  

 \bigskip 
 

Now let $\varepsilon >0$ and $\overline{\varphi} = \tilde \vartheta \zeta$ with $\tilde \vartheta \in V^{p'}_{0.div} \subset V^{p}_{0.div}$ and $\zeta \in {\mathcal D}(0,T)$.
With   (\ref{3.14bis}) we have
 \begin{eqnarray*}
\begin{array}{ll}
&\displaystyle{ \int_{0}^{T} \left\langle \frac{\partial \overline{\upsilon}_{\varepsilon}}{\partial t},\overline{\varphi}-\overline{\upsilon}_{\varepsilon} \right\rangle_{(V_{0.div}^{p'})',V_{0.div}^{p'}} \, dt
+ \bigl[\mathcal{A}_{{\bf u}} (\overline{\upsilon}_{\varepsilon}),\overline{\varphi}-\overline{\upsilon}_{\varepsilon} \bigr]}
\\
& 
\displaystyle{+ 2\varepsilon \int_{0}^{T}\int_{\Omega} \bigl|D(\overline{\upsilon}_{\varepsilon} ) \bigr|^{p'-2} D(\overline{\upsilon}_{\varepsilon} ):D(\overline{\varphi}-\overline{\upsilon}_{\varepsilon} ) \, dx dt
+ J(\overline{\varphi})-J(\overline{\upsilon}_{\varepsilon})\geq  \int_0^T ( \overline{f},\overline{\varphi}-\overline{\upsilon}_{\varepsilon})_{{\bf L}^2(\Omega)} }.
\end{array}
\end{eqnarray*}
Hence
\begin{eqnarray} \label{lundi1}
\begin{array}{ll}
& \displaystyle{ \bigl[ \mathcal{A}_{{\bf u}} (\overline{\upsilon}_{\varepsilon}),\overline{\upsilon}_{\varepsilon}-\overline{\upsilon} \bigr]
\leq
\int_{0}^{T} \left\langle \frac{\partial \overline{\upsilon}_{\varepsilon}}{\partial t} ,\overline{\varphi}-\overline{\upsilon}_{\varepsilon} \right\rangle_{(V_{0.div}^{p'})',V_{0.div}^{p'}} \, dt 
+ \bigl[ \mathcal{A}_{{\bf u}} (\overline{\upsilon}_{\varepsilon}) ,\overline{\varphi}-\overline{\upsilon} \bigr] }
\\
&
\displaystyle{
- 2\varepsilon \int_{0}^{T} \int_{\Omega} \bigl| D(\overline{\upsilon}_{\varepsilon} ) \bigr|^{p'} \, dx dt 
+ 2\varepsilon \int_{0}^{T} \int_{\Omega} \bigl| D(\overline{\upsilon}_{\varepsilon} ) \bigr|^{p'-2}
D(\overline{\upsilon}_{\varepsilon} ):D(\overline{\varphi} ) \,dx dt }
\\
& \displaystyle{
+J(\overline{\varphi})-J(\overline{\upsilon}_{\varepsilon})
- \int_0^T ( \overline{f},\overline{\varphi}-\overline{\upsilon}_{\varepsilon} )_{{\bf L}^2(\Omega))} \, dt }
\\
&
\displaystyle{ \leq
- \int_0^T ( \overline{\upsilon}_{\varepsilon}, \tilde \vartheta )_{{\bf L}^2(\Omega)} \zeta' \, dt 
- \frac{1}{2} \Vert \overline{\upsilon}_{\varepsilon} (T) \Vert_{{\bf L}^2(\Omega))}^2
+ \bigl[ \mathcal{A}_{{\bf u}} (\overline{\upsilon}_{\varepsilon}) ,\overline{\varphi}-\overline{\upsilon} \bigr] }
\\
& 
\displaystyle{ 
+ 2 \varepsilon^{1/p'} \Bigl( \varepsilon^{1/p'} \Vert \overline{\upsilon}_{\varepsilon} \Vert_{L^{p'} (0,T; V_{0.div}^{p'})} 
\Bigr)^{p'-1} \Vert D(\overline{\varphi} ) \Vert_{L^{p'} (0,T; (L^{p'}(\Omega)^{3 \times 3} ))} }
\\
& 
\displaystyle{ +J(\tilde \vartheta \zeta )-J(\overline{\upsilon}_{\varepsilon})- \int_0^T (\overline{f},\tilde \vartheta \zeta -\overline{\upsilon}_{\varepsilon} )_{{\bf L}^2(\Omega))} \, dt }.
\end{array}
\end{eqnarray}
With (\ref{conv_epsilon1bis}) we have immediately
\begin{eqnarray*}
\begin{array}{ll}
& \displaystyle 
\int_0^T (\overline{f},\tilde \vartheta \zeta -\overline{\upsilon}_{\varepsilon} )_{{\bf L}^2(\Omega)} \, dt
\longrightarrow  \int_0^T (\overline{f},\tilde \vartheta \zeta -\overline{\upsilon} )_{{\bf L}^2(\Omega)} \, dt ,
\\
& \displaystyle 
\int_0^T ( \overline{\upsilon}_{\varepsilon}, \tilde \vartheta )_{{\bf L}^2(\Omega)} \zeta' \, dt 
\longrightarrow \int_0^T ( \overline{\upsilon}, \tilde \vartheta )_{{\bf L}^2(\Omega)} \zeta' \, dt 
\end{array}
\end{eqnarray*}
and with (\ref{conv_trace}) we get
\begin{eqnarray*}
\begin{array}{ll}
\displaystyle \bigl| J( \overline{\upsilon}_{\varepsilon} ) - J (\overline{\upsilon} ) \bigr| 
\le \int_0^T \int_{\Gamma_0} k \bigl| \overline{\upsilon}_{\varepsilon} - \overline{\upsilon} \bigr| \, d x' dt 
\\
\displaystyle 
\le \Vert k \Vert_{L^{p'} (0,T; L^{p'} (\Gamma_0) )} \Vert \gamma_p (\overline{\upsilon}_{\varepsilon}) - \gamma_p ( \overline{\upsilon} ) \Vert_{L^p(0,T; {\bf L}^p(\partial \Omega))} \longrightarrow 0.
\end{array}
\end{eqnarray*}
Moreover ${\mathcal A}_{{\bf u}} $ is a bounded operator from $L^p \bigl( 0,T; V_{0}^p \bigr)$ to $L^{p'} \bigl( 0,T; (V_{0}^p)' \bigr)$ and $( \overline{\upsilon}_{\varepsilon} )_{\varepsilon >0}$ is bounded in $L^p \bigl( 0,T; V_{0.div}^p \bigr)$. So, by possibly extracting another subsequence still denoted $( \overline{\upsilon}_{\varepsilon} , \pi_{\varepsilon})_{\varepsilon >0}$, there exists $\chi \in L^{p'} \bigl( 0,T; (V_{0}^p)' \bigr)$ such that
\begin{eqnarray} \label{toto1_epsilon}
{\mathcal A}_{{\bf u}}  ( \overline{\upsilon}_{\varepsilon} ) \rightharpoonup \chi \quad \hbox{\rm weakly in $L^{p'} \bigl( 0,T; (V_{0}^p)' \bigr)$.}
\end{eqnarray}
By using  estimate (\ref{es_epsilon1}) we infer from  (\ref{lundi1})  
\begin{eqnarray*}
\begin{array}{ll}
\displaystyle \lim \sup \bigl[ {\mathcal A}_{{\bf u}} (  \overline{\upsilon}_{\varepsilon}) ,  \overline{\upsilon}_{\varepsilon} -  \overline{\upsilon} \bigr] 
\le 
- \int_0^T (  \overline{\upsilon}, \tilde \vartheta )_{{\bf L}^2(\Omega) } \zeta' \, dt 
- \frac{1}{2} \lim \inf \Vert  \overline{\upsilon}_{\varepsilon} (T) \Vert_{{\bf L}^2(\Omega)}^2 
\\
\displaystyle + [ \chi , \tilde \vartheta \zeta  -  \overline{\upsilon} ]
+ J( \tilde \vartheta \zeta) - J( \overline{\upsilon} ) 
-  \int_0^T ( \overline{f}, \tilde \vartheta \zeta - \overline{\upsilon} )_{{\bf L}^2(\Omega)} \, dt 
\end{array}
\end{eqnarray*}
and with (\ref{conv_epsilon8bis}) we get
\begin{eqnarray*}
\begin{array}{ll}
\displaystyle \lim \sup \bigl[ {\mathcal A}_{{\bf u}} (  \overline{\upsilon}_{\varepsilon}) ,  \overline{\upsilon}_{\varepsilon} -  \overline{\upsilon} \bigr] 
\le 
- \int_0^T (  \overline{\upsilon}, \tilde \vartheta )_{{\bf L}^2(\Omega) } \zeta' \, dt 
- \frac{1}{2}  \Vert  \overline{\upsilon} (T) \Vert_{{\bf L}^2(\Omega)}^2 
\\
\displaystyle + [ \chi , \tilde \vartheta \zeta  -  \overline{\upsilon} ]
+ J( \tilde \vartheta \zeta) - J( \overline{\upsilon} ) 
-  \int_0^T ( \overline{f}, \tilde \vartheta \zeta - \overline{\upsilon} )_{{\bf L}^2(\Omega)} \, dt .
\end{array}
\end{eqnarray*}

With Proposition \ref{coro 1.1} we know that $\displaystyle \frac{\partial \overline{\upsilon}}{\partial t} \in L^{p'} \bigl( 0,T; (V_{0.div}^p)' \bigr)$ and $\tilde \vartheta \zeta \in L^{p} \bigl( 0,T; V_{0.div}^{p} \bigr) $. So
\begin{eqnarray*}
- \int_0^T (  \overline{\upsilon}, \tilde \vartheta )_{{\bf L}^2(\Omega) } \zeta' \, dt
= \left\langle \frac{\partial}{\partial t} (\overline{\upsilon}, \tilde \vartheta )_{{\bf L}^2(\Omega) },  \zeta
\right\rangle_{{\mathcal D}'(0,T), {\mathcal D}(0,T)} 
= \int_0^T \left\langle \frac{\partial \overline{\upsilon}}{\partial t} ,  \tilde \vartheta \zeta \right\rangle_{(V_{0.div}^p)', V_{0.div}^p } \, dt 
\end{eqnarray*}
and by density of ${\mathcal D}(0,T) \otimes V_{0.div}^{p'}$ into $L^{p} (0,T; V_{0.div}^{p'})$ we have 
\begin{eqnarray*}
\begin{array}{ll}
\displaystyle \lim \sup \bigl[ {\mathcal A}_{{\bf u}}  (  \overline{\upsilon}_{\varepsilon}) ,  \overline{\upsilon}_{\varepsilon} -  \overline{\upsilon} \bigr] 
\le 
\int_0^T \left\langle \frac{\partial \overline{\upsilon}}{\partial t} ,  \overline{\varphi} \right\rangle_{(V_{0.div}^p)', V_{0.div}^p } \, dt 
 - \frac{1}{2}  \Vert  \overline{\upsilon}  (T) \Vert_{{\bf L}^2(\Omega)}^2 
\\
\displaystyle + [ \chi , \overline{\varphi} -  \overline{\upsilon} ]
+ J( \overline{\varphi}) - J( \overline{\upsilon} ) 
-  \int_0^T  (  \overline{f}, \overline{\varphi} - \overline{\upsilon} )_{{\bf L}^2(\Omega)} \, dt 
\qquad \forall \overline{\varphi} \in L^{p} (0,T; V_{0.div}^{p'}).
\end{array}
\end{eqnarray*}
By choosing $\overline{\varphi} =  \overline{\upsilon}_{\varepsilon'}$ and taking the limit as $\varepsilon '$ tends to zero in the right-hand side of the previous inequality we obtain
\begin{eqnarray} \label{toto2_epsilon} 
\displaystyle \qquad \qquad  \quad \lim \sup \bigl[ {\mathcal A}_{{\bf u}} (  \overline{\upsilon}_{\varepsilon}) ,  \overline{\upsilon}_{\varepsilon} -  \overline{\upsilon} \bigr] 
\le 
\underbrace{ \int_0^T \left\langle \frac{\partial \overline{\upsilon}}{\partial t} , \overline{\upsilon}  \right\rangle_{(V_{0.div}^p)', V_{0.div}^p } }_{= \frac{1}{2} \Vert  \overline{\upsilon} (T) \Vert_{{\bf L}^2(\Omega)}^2 }
 - \frac{1}{2}  \Vert  \overline{\upsilon} (T) \Vert_{{\bf L}^2(\Omega)}^2  = 0 .
 \end{eqnarray}


By recalling that the mapping $\lambda_2 \mapsto {\mathcal F} ( \cdot, \cdot, \lambda_2)$ is monotone in ${\mathbb R}^{3 \times 3}$ for any $p>1$ (see Lemma 1 in \cite{BDP1}) we infer that the operator ${\mathcal A}_{{\bf u}}$ is monotone. Moreover with (\ref{rop}) and (\ref{eq17ante}) we obtain that ${\mathcal A}_{{\bf u}}$ is bounded and hemicontinuous.
Hence  ${\mathcal A}_{{\bf u}}$ is pseudo-monotone. We obtain
\begin{eqnarray*}
\displaystyle \lim \inf \bigl[ {\mathcal A}_{{\bf u}} (  \overline{\upsilon}_{\varepsilon}) ,  \overline{\upsilon}_{\varepsilon} -  \overline{\varphi} \bigr] 
\ge  \bigl[ {\mathcal A}_{{\bf u}} (  \overline{\upsilon}) ,  \overline{\upsilon} -  \overline{\varphi} \bigr]
\qquad \forall \overline{\varphi} \in L^{p} (0,T; V_{0}^{p}) 
\end{eqnarray*}
and with (\ref{toto1_epsilon}) we get 
\begin{eqnarray} \label{toto3_epsilon}
\displaystyle \lim \inf \bigl[ {\mathcal A}_{{\bf u}} (  \overline{\upsilon}_{\varepsilon}) ,  \overline{\upsilon}_{\varepsilon} \bigr] 
- [ \chi,   \overline{\varphi} ] 
\ge  \bigl[ {\mathcal A}_{{\bf u}} (  \overline{\upsilon}) ,  \overline{\upsilon} -  \overline{\varphi} \bigr]
\qquad \forall \overline{\varphi} \in L^{p} (0,T; V_{0}^{p}) .
\end{eqnarray}
By choosing $ \overline{\varphi}  =  \overline{\upsilon} $ we infer that 
\begin{eqnarray*}
\displaystyle \lim \inf \bigl[ {\mathcal A}_{{\bf u}} (  \overline{\upsilon}_{\varepsilon}) ,  \overline{\upsilon}_{\varepsilon} \bigr] 
\ge  [ \chi,   \overline{\upsilon} ] .
\end{eqnarray*}
On the other hand, with (\ref{toto2_epsilon}) we have 
\begin{eqnarray*}
\lim \sup  \bigl[ {\mathcal A}_{{\bf u}} (  \overline{\upsilon}_{\varepsilon}) ,  \overline{\upsilon}_{\varepsilon}  -   \overline{\upsilon} \bigr] 
= \lim \sup   \bigl[ {\mathcal A}_{{\bf u}} (  \overline{\upsilon}_{\varepsilon}) ,  \overline{\upsilon}_{\varepsilon} \bigr]  - [ \chi,    \overline{\upsilon} ] 
\le 0.
\end{eqnarray*}
It follows that 
\begin{eqnarray*}
[ \chi,   \overline{\upsilon} ]  =  \lim \inf \bigl[ {\mathcal A}_{{\bf u}} (  \overline{\upsilon}_{\varepsilon}) ,  \overline{\upsilon}_{\varepsilon} \bigr] 
=  \lim \sup \bigl[ {\mathcal A}_{{\bf u}} (  \overline{\upsilon}_{\varepsilon}) ,  \overline{\upsilon}_{\varepsilon} \bigr] .
\end{eqnarray*}
Going back to (\ref{toto3_epsilon}) we obtain 
\begin{eqnarray*}
[ \chi,   \overline{\upsilon}  - \overline{\varphi} ]  \ge  \bigl[ {\mathcal A}_{{\bf u}} (  \overline{\upsilon}) ,  \overline{\upsilon} -  \overline{\varphi} \bigr]
\qquad \forall \overline{\varphi} \in L^{p} (0,T; V_{0}^{p}) 
\end{eqnarray*}
and with $\overline{\varphi} = \overline{\upsilon} \pm \overline{\psi}$ with $\overline{\psi} \in L^{p} (0,T; V_{0}^{p})$ we may conclude that 
\begin{eqnarray*}
[ \chi -  {\mathcal A}_{{\bf u}} (  \overline{\upsilon}) ,    \overline{\psi} \bigr] =0
\qquad \forall \overline{\psi} \in L^{p} (0,T; V_{0}^{p}) 
\end{eqnarray*}
which yields $\chi = {\mathcal A}_{{\bf u}} (  \overline{\upsilon}) $ in $ L^{p'} \bigl(0,T; (V_{0}^{p})' \bigr)$, i.e.
\begin{eqnarray*}
{\mathcal A}_{{\bf u}} (  \overline{\upsilon}_{\varepsilon} ) \rightharpoonup {\mathcal A}_{{\bf u}} (  \overline{\upsilon}) \quad \hbox{\rm weakly in $ L^{p'} \bigl(0,T; (V_{0}^{p})' \bigr)$.}
\end{eqnarray*}

\bigskip

Now let $\tilde \vartheta  \in V_{0}^{p'}$ and $\zeta \in {\mathcal D}(0,T)$. Owing that $V_{0}^{p'} \subset V_{0}^p$ we obtain
\begin{eqnarray*}
\bigl[ {\mathcal A}_{{\bf u}} (  \overline{\upsilon}_{\varepsilon} ), \tilde \vartheta \zeta \bigr] \longrightarrow \bigl[ {\mathcal A}_{{\bf u}} (  \overline{\upsilon} ), \tilde \vartheta \zeta \bigr].
\end{eqnarray*}
We already know that 
\begin{eqnarray*}
\bigl| J( \overline{\upsilon}_{\varepsilon} ) - J (\overline{\upsilon} ) \bigr| 
\le \Vert k \Vert_{L^{p'} (0,T; L^{p'} (\Gamma_0) )} \Vert \gamma_p (\overline{\upsilon}_{\varepsilon}) - \gamma_p ( \overline{\upsilon} ) \Vert_{L^p(0,T; {\bf L}^p(\partial \Omega))} 
\longrightarrow 0
\end{eqnarray*}
and similarly
\begin{eqnarray*}
\bigl| J( \overline{\upsilon}_{\varepsilon} + \tilde \vartheta \zeta) - J (\overline{\upsilon} + \tilde \vartheta \zeta) \bigr| 
\le \Vert k \Vert_{L^{p'} (0,T; L^{p'} (\Gamma_0) )} \Vert \gamma_p (\overline{\upsilon}_{\varepsilon}) - \gamma_p ( \overline{\upsilon} ) \Vert_{L^p(0,T; {\bf L}^p(\partial \Omega))}.
 \longrightarrow 0.
\end{eqnarray*}
Moreover 
\begin{eqnarray*}
\begin{array}{ll}
& \displaystyle
\left| 2\varepsilon \int_{0}^{T}\int_{\Omega} \bigl|D(\overline{\upsilon}_{\varepsilon} ) \bigr|^{p'-2} D(\overline{\upsilon}_{\varepsilon} ):D(\tilde \vartheta \zeta) \, dx dt \right| 
 \\
 & \displaystyle  
\le 
 2 \varepsilon^{1/p'} \Bigl( \varepsilon^{1/p'} \Vert \overline{\upsilon}_{\varepsilon} \Vert_{L^{p'} (0,T; V_{0.div}^{p'})} 
\Bigr)^{p'-1} \Vert D(\tilde \vartheta \zeta ) \Vert_{L^{p'} (0,T; (L^{p'}(\Omega)^{3 \times 3} ))} 
\longrightarrow 0
  \end{array}
  \end{eqnarray*}
and 
\begin{eqnarray*}
\begin{array}{ll}
& \displaystyle
 \left\langle \frac{\partial}{\partial t} ( \overline{\upsilon}_{\varepsilon}, \tilde \vartheta)_{{\bf L}^2(\Omega)} , \zeta \right\rangle_{{\mathcal D}'(0,T), {\mathcal D}(0,T)} 
 = - \int_0^T ( \overline{\upsilon}_{\varepsilon}, \tilde \vartheta)_{{\bf L}^2(\Omega)}  \zeta' \, dt
 \\
 & \displaystyle
  \longrightarrow - \int_0^T ( \overline{\upsilon}, \tilde \vartheta)_{{\bf L}^2(\Omega)}  \zeta' \, dt 
  =  \left\langle \frac{\partial}{\partial t} ( \overline{\upsilon} , \tilde \vartheta)_{{\bf L}^2(\Omega)} , \zeta \right\rangle_{{\mathcal D}'(0,T), {\mathcal D}(0,T)} .
  \end{array}
  \end{eqnarray*}
Finally with (\ref{conv_pression}) we get
\begin{eqnarray*}
\left\langle \int_{\Omega} \pi_{\varepsilon} div(\tilde \vartheta) , \zeta \right\rangle_{{\mathcal D}'(0,T), {\mathcal D}(0,T)} 
 \longrightarrow \left\langle \int_{\Omega} \pi div(\tilde \vartheta) , \zeta \right\rangle_{{\mathcal D}'(0,T), {\mathcal D}(0,T)} .
 \end{eqnarray*}
 
 Thus we can pass to the limit in (\ref{3.14}) as $\varepsilon$ tends to zero and we get 
 \begin{eqnarray*}\label{3.14.1}
\begin{array}{ll}
\displaystyle{ \Big<\frac{\partial}{\partial t}(\overline{\upsilon}, \tilde\vartheta )_{\bf{L}^2(\Omega)},\zeta\Big>_{\mathcal{D}'(0,T),\mathcal{D}(0,T)}
+ \int_{0}^{T} \int_{\Omega} \mathcal{F} \bigl( \theta, {\bf u} + \upsilon_0 \xi, D( \overline{\upsilon} + \upsilon_0 \xi) \bigr) : D( \tilde\vartheta) \zeta \, dx dt } 
\medskip 
\\
\displaystyle{
-  \Big< \int_{\Omega} \pi \ div(\tilde \vartheta ) \, dx , \zeta\Big>_{\mathcal{D}'(0,T),\mathcal{D}(0,T)} } 
\displaystyle{
+ J( \overline{\upsilon} +\tilde\vartheta \zeta)-J(\overline{\upsilon}) }
\\
\displaystyle{
\geq\,\, \int_{0}^{T} \left( f + \frac{\partial \xi}{\partial t} \upsilon_0, \tilde\vartheta \right)_{{\bf L}^2(\Omega)}\zeta\,dt} 
\qquad \quad  \forall \tilde\vartheta\in V^{p'}_{0},  \  \forall \zeta\in \mathcal{D}(0,T) . 
 \end{array}
 \end{eqnarray*}
Finally, by using the density of $V^{p'}_{0}$ into $V^{p}_{0}$ and reminding that $\pi \in H^{-1} \bigl(0,T; L^{p'}_0(\Omega) \bigr)$, the same inequality holds for any $\tilde \vartheta \in V^p_0$ and $\zeta \in  \mathcal{D}(0,T)$ and we may conclude that $( \overline{\upsilon}, \pi)$ is a solution of problem (P${}_{{\bf u}}$).

\smallskip

Let ${\overline{\varphi}}= \tilde\vartheta \zeta$ with $\tilde\vartheta \in V^{p}_{0.div}$ and $\zeta\in \mathcal{D}(0,T) $. Owing that $\displaystyle \frac{\partial  \overline{\upsilon}}{\partial t} \in L^{p'} \bigl( 0,T; (V_{0.div}^{p})' \bigr)$ 
 we have 
\begin{eqnarray*}\label{3.14bisbis}
\displaystyle 
\underbrace{ \Big<\frac{\partial}{\partial t}(\overline{\upsilon} , \tilde\vartheta )_{\bf{L}^2(\Omega)},\zeta\Big>_{\mathcal{D}'(0,T),\mathcal{D}(0,T)}}_{\displaystyle = \int_0^T \Big< \frac{\partial \overline{\upsilon} }{\partial t}, {\overline{\varphi}}  \Big>_{ (V^{p}_{0.div})', V^{p}_{0.div} }    \, dt }
+  \bigl[\mathcal{A}_{{\bf u}} (\overline{\upsilon} ) , {\overline{\varphi}}  \bigr]
\displaystyle{
+ J( \overline{\upsilon} + {\overline{\varphi}} )-J(\overline{\upsilon})
\geq\,\, \int_{0}^{T} \left( f + \frac{\partial \xi}{\partial t} \upsilon_0, {\overline{\varphi}}  \right)_{{\bf L}^2(\Omega)} \,dt} 
 \end{eqnarray*}
and by density of $\mathcal{D}(0,T) \otimes V^{p}_{0.div}$ into $L^{p} \bigl(0,T; V^{p}_{0.div} \bigr)$ the same inequality is true for any ${\overline{\varphi}} \in L^{p} \bigl(0,T; V^{p}_{0.div} \bigr)$.

\bigskip


Let us prove now that  (P${}_{{\bf u}}$) admits an unique solution. Let us argue by contradiction and let us assume that this problem admits two solutions  $(\overline{\upsilon}_{1}, \pi_1) $ and $(\overline{\upsilon}_{2}, \pi_2) $. So, we have  the following two inequalities
\begin{eqnarray*}
\begin{array}{ll}
\displaystyle
\int_{0}^{T} \left\langle \frac{ \partial \overline{\upsilon}_1}{\partial t}, \overline{\varphi} - \overline{\upsilon}_{1} \right\rangle_{(V_{0.div}^p)', V_{0.div}^p}  \,dt
+ \bigl[ \mathcal{A}_{{\bf u}} (\overline{\upsilon}_{1}) ,\overline{\varphi}-\overline{\upsilon}_{1} \bigr]
+ J(\overline{\varphi}) - J(\overline{\upsilon}_{1})
\\
\displaystyle \geq \int_0^T (\overline{f},\overline{\varphi}-\overline{\upsilon}_{1} )_{{\bf L}^2(\Omega) }  \, dt 
\end{array}
\end{eqnarray*}
\begin{eqnarray*}
\begin{array}{ll}
\displaystyle \int_{0}^{T} \left\langle \frac{ \partial \overline{\upsilon}_2}{\partial t}, \overline{\varphi} - \overline{\upsilon}_{2} \right\rangle_{(V_{0.div}^p)', V_{0.div}^p}  \,dt
+ \bigl[ \mathcal{A}_{{\bf u}} (\overline{\upsilon}_{2}) ,\overline{\varphi}-\overline{\upsilon}_{2} \bigr]
+ J(\overline{\varphi}) - J(\overline{\upsilon}_{2})
\\
\displaystyle 
\geq \int_0^T (\overline{f},\overline{\varphi}-\overline{\upsilon}_{2} )_{{\bf L}^2(\Omega) }  \, dt
\end{array}
\end{eqnarray*}
for all $\overline{\varphi}\in L^{p}(0,T;V^p_{0.div})$. 

Let  $t\in (0,T]$. We may choose   $\overline{\varphi}=\overline{\upsilon}_{2} {\bf 1}_{[0,t]}+\overline{\upsilon}_{1} (1-{\bf 1}_{[0,t]}) $ in the first inequality  and  $\overline{\varphi}=\overline{\upsilon}_{1} {\bf 1}_{[0,t]}+\overline{\upsilon}_{2} (1-{\bf 1}_{[0,t]}) $ in the second one, where  ${\bf 1}_{[0,t]}$  denotes  the indicatrix function of the interval $[0,t]$.
By adding the two inequalities, we get
\begin{eqnarray*}
\frac{1}{2} \Vert  \overline{\upsilon}_{1} (t) - \overline{\upsilon}_{2} (t) \Vert_{{\bf L}^2(\Omega)} 
+ \bigl[ \mathcal{A}_{{\bf u}} (\overline{\upsilon}_{1}) -\mathcal{A}_{{\bf u}} (\overline{\upsilon}_{2}) ,(\overline{\upsilon}_{1}-\overline{\upsilon}_{2}) {\bf 1}_{[0,t]} \bigr] \leq \frac{1}{2} \Vert  \overline{\upsilon}_{1} (0) - \overline{\upsilon}_{1} (0) \Vert_{{\bf L}^2(\Omega)}  = 0.
\end{eqnarray*}
On the other hand
 \begin{eqnarray*}
 \begin{array}{ll}
\displaystyle  \bigl[ \mathcal{A}_{{\bf u}} (\overline{\upsilon}_{1}) -\mathcal{A}_{{\bf u}} (\overline{\upsilon}_{2}) ,(\overline{\upsilon}_{1}-\overline{\upsilon}_{2}) {\bf 1}_{[0,t]} \bigr] 
 \\
 \displaystyle 
 = \int_0^t \int_{\Omega} \bigl( {\mathcal F} \bigl( \theta, {\bf u} + \upsilon_0 \xi, D( \upsilon_1 + \upsilon_0 \xi) \bigr)  -   {\mathcal F} \bigl( \theta, {\bf u} + \upsilon_0 \xi, D( \upsilon_2 + \upsilon_0 \xi) \bigr) \bigr) : D( \upsilon_1 - \upsilon_2) \, dx d \tilde t.
\end{array}
\end{eqnarray*}
With Lemma 1 in \cite{BDP1} we know that the mapping $\lambda_2 \mapsto {\mathcal F} ( \cdot, \cdot, \lambda_2)$ is monotone on $\mathbb{R}^{3 \times 3}$. It follows that  
  \begin{eqnarray*}
 \bigl[ \mathcal{A}_{{\bf u}} (\overline{\upsilon}_{1}) -\mathcal{A}_{{\bf u}} (\overline{\upsilon}_{2}) ,(\overline{\upsilon}_{1}-\overline{\upsilon}_{2}) {\bf 1}_{[0,t]} \bigr]  \ge 0
 \end{eqnarray*}
 and 
 \begin{eqnarray*}
  \Vert  \overline{\upsilon}_{1} (t) - \overline{\upsilon}_{1} (t) \Vert_{{\bf L}^2(\Omega)}  \le  \Vert  \overline{\upsilon}_{1} (0) - \overline{\upsilon}_{1} (0) \Vert_{{\bf L}^2(\Omega)}  = 0 \quad \forall t \in (0,T].
  \end{eqnarray*}
Thus  $ \overline{\upsilon}_{1} = \overline{\upsilon}_{2}$.
  It follows that 
\begin{eqnarray*}
\left\langle \int_{\Omega} (\pi_1 - \pi_2 ) \, div(\tilde \vartheta) \, dx , \zeta  \right\rangle_{{\mathcal D}'(0,T), {\mathcal D}(0,T)} =0 \quad \forall \tilde \vartheta \in {\bf W}^{1,p}_0 (\Omega), \ \forall \zeta \in {\mathcal D}(0,T).
\end{eqnarray*}
By using the same notations as in Proposition \ref{apriori-estimates2} we obtain that 
\begin{eqnarray*}
\begin{array}{ll}
\displaystyle \left\langle \int_{\Omega} (\pi_1 - \pi_2 ) \, \varpi^* \, dx , \zeta  \right\rangle_{{\mathcal D}'(0,T), {\mathcal D}(0,T)} 
= \left\langle \int_{\Omega} (\pi_1 - \pi_2 ) \, \varpi \, dx , \zeta  \right\rangle_{{\mathcal D}'(0,T), {\mathcal D}(0,T)} 
\\
\displaystyle 
= \left\langle \int_{\Omega} (\pi_1 - \pi_2 ) \, div\bigl( P_p(\varpi) \bigr)  \, dx , \zeta  \right\rangle_{{\mathcal D}'(0,T), {\mathcal D}(0,T)} 
=0 \quad \forall \varpi^* \in L^p (\Omega), \ \forall \zeta \in {\mathcal D}(0,T)
\end{array}
\end{eqnarray*}
  which yields $\pi_1 = \pi_2$ and concludes the proof.
\end{proof}

\bigskip


\section{Existence result for problem (P)} \label{fixed_point}

\bigskip

As already explained at the end of Section \ref{description} we will prove now the existence of a solution to problem (P) by using a fixed point argument. More precisely we consider the mapping $\Lambda: L^p \bigl(0,T; {\bf L}^p (\Omega) \bigr) \to L^p \bigl(0,T; {\bf L}^p (\Omega) \bigr)$   defined by $\Lambda ({\bf u}) = \overline{\upsilon}$ where  $(\overline{\upsilon}, \pi)$ is the unique solution of problem (P${}_{\bf u}$) for all ${\bf u} \in L^p \bigl(0,T; {\bf L}^p (\Omega) \bigr)$ i.e. $\Lambda ({\bf u}) = \overline{\upsilon}$ satisfies
\begin{eqnarray}\label{3.14.1bis}
\begin{array}{ll}
\displaystyle 
 \int_0^T \Big< \frac{\partial \overline{\upsilon} }{\partial t}, {\overline{\varphi}}  \Big>_{ (V^{p}_{0.div})', V^{p}_{0.div} }    \, dt 
+  \bigl[\mathcal{A}_{{\bf u}} (\overline{\upsilon} ) , {\overline{\varphi}}  \bigr]
\medskip\\
\displaystyle{
+ J( \overline{\upsilon} + {\overline{\varphi}} )-J(\overline{\upsilon})
\geq\,\, \int_{0}^{T} \left( f + \frac{\partial \xi}{\partial t} \upsilon_0, {\overline{\varphi}}  \right)_{{\bf L}^2(\Omega)} \,dt} 
 \quad \forall {\overline{\varphi}} \in L^{p} \bigl(0,T; V^{p}_{0.div} \bigr).
 \end{array}
  \end{eqnarray}
By using Schauder's fixed point theorem we will prove that $\Lambda$ admits a fixed point and we will use   De Rham's theorem to construct  the pressure term.

\begin{theorem} \label{main_theo}
  Let   $f\in  L^ {p'}\bigl(0,T;{\bf L}^2(\Omega) \bigr)$, $k \in L^{p'}\bigl(0,T;L_{+}^{p'}(\Gamma_0)\bigr)$, $\mu$ satisfying (\ref{rop})-(\ref{mlo}), 
  $\theta \in  L^{\tilde q} \bigl(0,T;L^{\tilde p}(\Omega) \bigr)$ with $\tilde q \ge 1$ and $\tilde p \ge 1$,
$s\in L^{ p} \bigl(0,T;{\bf L}^p(\Gamma_0) \bigr)$, $\xi\in W^{1,p'}(0,T)$ satisfying (\ref{xi})
and  $\upsilon_{0}\in \textbf{W}^{1,p}(\Omega)$ satisfying  (\ref{coco}). 
Then problem (P) admits a solution i.e. there exist $\overline{\upsilon}\in C \bigl([0,T];{\bf L}^2(\Omega) \bigr)\cap L^{ p} \bigl(0,T;V^p_{0.div} \bigr)$  with  $\displaystyle \frac{\partial \overline{\upsilon}}{\partial t} \in L^{p'} \bigl(0,T; (V^p_{0.div})' \bigr)$  and  $\pi \in H^{-1}\bigl(0,T;L_{0}^{p'}(\Omega) \bigr)$ such that 
\begin{eqnarray*}\label{3.14.2}
\begin{array}{ll}
\displaystyle{ \Big<\frac{\partial}{\partial t}(\overline{\upsilon} , \tilde\vartheta )_{\bf{L}^2(\Omega)},\zeta\Big>_{\mathcal{D}'(0,T),\mathcal{D}(0,T)}
+ \int_{0}^{T} \int_{\Omega} \mathcal{F} \bigl( \theta, \overline{\upsilon} + \upsilon_0 \xi, D( \overline{\upsilon} + \upsilon_0 \xi) \bigr) : D( \tilde\vartheta) \zeta \, dx dt}
\medskip\\
\displaystyle{
-  \Big< \int_{\Omega} \pi \ div(\tilde \vartheta ), \zeta\Big>_{\mathcal{D}'(0,T),\mathcal{D}(0,T)} 
+ J( \overline{\upsilon}+\tilde\vartheta \zeta)-J(\overline{\upsilon}) } 
\medskip\\
\displaystyle{

\geq\,\, \int_{0}^{T} \left( f + \frac{\partial \xi}{\partial t} \upsilon_0, \tilde\vartheta \right)_{\bf{L}^2(\Omega)}\zeta\,dt} 
\qquad  \quad  \forall \tilde\vartheta\in V^p_{0},  \  \forall \zeta\in \mathcal{D}(0,T)  
 \end{array}
 \end{eqnarray*}
and the initial condition
 \begin{eqnarray*}\label{3.15}
\overline{\upsilon}(0)=\upsilon_{0}-\upsilon_{0}\xi(0)=0\qquad  \mbox{\rm in}\  \Omega.
\end{eqnarray*}
 \end{theorem}


\begin{proof}
As a first step we prove that $\Lambda$ satisfies the assumptions of Schauder's fixed point theorem.

\bigskip

Let    ${\bf u}\in L^p\bigl(0,T; {\bf L}^p(\Omega) \bigr)$ and $\Lambda ({\bf u}) = \overline{\upsilon}$. With $\overline{\varphi}= - \overline{\upsilon}$ in  (\ref{3.14.1bis}) we obtain
\begin{eqnarray*}\label{bnb}
\underbrace{ \int_0^T \Big< \frac{\partial \overline{\upsilon} }{\partial t}, {\overline{\upsilon}}  \Big>_{ (V^{p}_{0.div})', V^{p}_{0.div} }    \, dt }_{=  \frac{1}{2} \Vert\overline{\upsilon}(T)\Vert_{{\bf L}^2(\Omega)}^2} 
 +\bigl[ \mathcal{A}_{\bf{u}}(\overline{\upsilon}),\overline{\upsilon} \bigr]
+ J(\overline{\upsilon}) \leq \int_{0}^{T} (\overline{f} , {\overline{\upsilon}}  )_{{\bf L}^2(\Omega)} \,dt
+J(0).
\end{eqnarray*}
Hence
\begin{eqnarray*}\label{bnb}
\bigl[ \mathcal{A}_{\bf{u}}(\overline{\upsilon}),\overline{\upsilon} \bigr]  \le  \frac{1}{2} \Vert\overline{\upsilon}(T)\Vert_{{\bf L}^2(\Omega)}^2 +  \bigl[ \mathcal{A}_{\bf{u}}(\overline{\upsilon}),\overline{\upsilon} \bigr] + J(\overline{\upsilon})
\leq  \int_{0}^{T} (\overline{f} , {\overline{\upsilon}}  )_{{\bf L}^2(\Omega)} \,dt  + J(0).
\end{eqnarray*}
With the same kind of computations as in Proposition \ref{apriori-estimates1} we get 
\begin{eqnarray*}
\begin{array}{ll}
\displaystyle  2(C_{Korn, p})^p\mu_{0} 
\left| 1 - \frac{ \Vert \upsilon_{0}\xi \Vert_{L^p(0,T;V_{\Gamma_{1}}^p)} }{\Vert \overline{\upsilon} \Vert_{L^p(0,T;V_{0.div}^p)} } \right|^p  
\le   \widetilde C \Vert \overline{f} \Vert_{L^{p'} (0,T; {\bf L}^2(\Omega))} \Vert \overline{\upsilon} \Vert_{L^p(0,T;V_{0.div}^p)}^{1-p} 
\\ 
\displaystyle 
+  \frac{J(0)}{\Vert \overline{\upsilon}\Vert_{L^p(0,T;V_{0.div}^p)}^p}
+ 2 \mu_{1}  \left( 1 + \frac{ \Vert \upsilon_{0}\xi \Vert_{L^p(0,T;V_{\Gamma_{1}}^p) } }{\Vert \overline{\upsilon} \Vert_{L^p(0,T;V_{0.div}^p)}}  \right)^{p-1}  
\frac{ \Vert \upsilon_{0}\xi  \Vert_{L^{p} (0,T;V_{\Gamma_{1}}^p)}  }{\Vert \overline{\upsilon} \Vert_{L^p(0,T;V_{0.div}^p)}} 
\end{array}
\end{eqnarray*}
if $ \Vert \overline{\upsilon} \Vert_{L^p(0,T;V_{0.div}^p)} \not=0$, where we recall that $\widetilde C$ denotes the norm of the continuous injection of $V_0^p$ into ${\bf L}^2 (\Omega)$.
Since the mapping 
\begin{eqnarray*}
\begin{array}{ll}
\displaystyle z \mapsto  2(C_{Korn, p})^p\mu_{0} \left| 1 - \frac{ \Vert \upsilon_{0}\xi \Vert_{L^p(0,T;V_{\Gamma_{1}}^p) } }{z} \right|^p  
\displaystyle - \widetilde C  \frac{ \Vert \overline{f} \Vert_{L^{p'}(0,T;{\bf L}^2(\Omega))} }{z^{p-1}} - \frac{J(0)}{z^p} 
\\
\displaystyle - 2 \mu_{1}  \left( 1 + \frac{ \Vert\upsilon_{0}\xi \Vert_{L^p(0,T;V_{\Gamma_{1}}^p) } }{z}  \right)^{p-1}  
\frac{\Vert \upsilon_{0}\xi  \Vert_{L^{p} (0,T;V_{\Gamma_{1}}^p)}  }{z} 
\end{array}
\end{eqnarray*}
admits $2(C_{Korn, p})^p\mu_{0} >0$ as limit when $z$ tends to $+ \infty$, there exists a real number $C>0$, independent of ${\bf u}$, such that 
\begin{eqnarray} \label{esom1}
 \Vert \Lambda ({\bf u})  \Vert_{L^p(0,T;V_{0.div}^p)}=  \Vert \overline{\upsilon} \Vert_{L^p(0,T;V_{0.div}^p)}\leq C  \qquad \forall {\bf u} \in L^p\bigl(0,T; {\bf L}^p(\Omega) \bigr).
\end{eqnarray}
Hence
\begin{eqnarray*} 
\Vert  \Lambda ({\bf u}) \Vert_{L^p(0,T; {\bf L}^p(\Omega))} \leq C_p  C \qquad \forall {\bf u}\in L^p\bigl(0,T; {\bf L}^p(\Omega) \bigr)
\end{eqnarray*}
where $C_p$ is the norm of the identity mapping from $ L^{p} \bigl(0,T;V_{0.div}^p \bigr)$ into $L^p\bigl(0,T; {\bf L}^p(\Omega) \bigr)
$. 


Let us consider now ${\overline \varphi} = \pm \tilde \vartheta \zeta$ with $ \tilde \vartheta \in V^p_{0.div}$ and  $\zeta \in {\mathcal D} (0,T)$.  With 
  the same computations as in Proposition \ref{apriori-estimates2} we obtain 
\begin{eqnarray*}
\begin{array}{ll}
\displaystyle{\left| \int_0^T \left\langle \frac{\partial {\overline \upsilon}}{\partial t} , \tilde \vartheta \zeta \right\rangle_{(V_{0.div}^p)', V_{0.div}^p}  \, dt \right|}
\\
\displaystyle \leq 
\Bigl(
\| \gamma_p\|_{ {\mathcal L} ( {\bf W}^{1,p}(\Omega) , {\bf L}^p( \partial \Omega) )} 
 \Vert k\Vert_{L^{p'}(0,T;L^{p'}(\Gamma_0))}
 + \widetilde C 
 \Vert\overline{f}\Vert_{L^{p'}(0,T;{\bf L}^2(\Omega)) }
  \Bigr) \Vert\tilde \vartheta \zeta\Vert_{ L^p(0,T;V_{0.div}^p)}
  +\bigl| \bigl[ \mathcal{A}_{\bf u}(\overline{\upsilon}), \tilde \vartheta \zeta \bigr] \bigr|  
\end{array}
\end{eqnarray*}
where $\gamma_p$ is  the trace operator from$ {\bf W}^{1,p} (\Omega) $ into $ {\bf L}^p (\partial \Omega)$.
But with (\ref{mlo}) we have 
\begin{eqnarray*}
\begin{array}{ll}
\displaystyle \bigl| \bigl[ \mathcal{A}_{\bf u}(\overline{\upsilon}), \tilde \vartheta \zeta \bigr] \bigr|  
\le 2 \mu_1 \Vert D( \overline{\upsilon} + \upsilon_0 \xi) \Vert_{L^{p} (0,T; (L^p(\Omega))^{3 \times 3}) }^{p-1} 
\Vert D( \tilde \vartheta \zeta ) \Vert_{L^{p} (0,T; (L^p(\Omega))^{3 \times 3} ) }
\\
\displaystyle \le  2 \mu_1 \bigl( \Vert  {\overline \upsilon}  \Vert_{L^{p} (0,T;V_{0.div}^p )}   +  \Vert  \upsilon_0 \xi \Vert_{L^{p} (0,T;V_{\Gamma_1}^p) } \bigr)^{p-1}
\Vert  \tilde \vartheta \zeta \Vert_{L^{p} (0,T;V_{0.div}^p ) } 
\\
\displaystyle
 \le 2 \mu_1 \bigl( C +  \Vert  \upsilon_0 \xi \Vert_{L^{p} (0,T;V_{\Gamma_1}^p) } \bigr)^{p-1} 
 \Vert  \tilde \vartheta \zeta \Vert_{L^{p} (0,T;V_{0.div}^p ) } .
\end{array}
\end{eqnarray*}
It follows that 
\begin{eqnarray} \label{esom2}
\begin{array}{ll}
\displaystyle 
\left\| \frac{\partial {\overline \upsilon}}{\partial t}  \right\|_{L^{p'}  (0,T; (V_{0.div}^p)')} 
 \le  \| \gamma_p\|_{ {\mathcal L} ( {\bf W}^{1,p}(\Omega) , {\bf L}^p( \partial \Omega) )} 
 \Vert k\Vert_{L^{p'}(0,T;L^{p'}(\Gamma_0))}
\\
\displaystyle
 + \widetilde C 
 \Vert\overline{f}\Vert_{L^{p'}(0,T;{\bf L}^2(\Omega)) }
 +  2 \mu_1 \bigl( C +  \Vert  \upsilon_0 \xi \Vert_{L^{p} (0,T;V_{\Gamma_1}^p) } \bigr)^{p-1} 
\end{array}
\end{eqnarray}
for all ${\bf u}\in L^p\bigl(0,T; {\bf L}^p(\Omega) \bigr)$.
   
 \bigskip
 
By using Aubin's lemma we infer that the closed ball ${\mathcal B} = B_{L^p (0,T; {\bf L}^p(\Omega) )} (0, C_p C) $ satisfies $\Lambda({\mathcal B} ) \subset {\mathcal B}$ and that $\Lambda({\mathcal B} )$ is relatively compact in $L^p\bigl(0,T; {\bf L}^p(\Omega) \bigr)
$.

\bigskip


It remains to prove that the mapping   ${\Lambda}$ is continuous. Indeed,
let  $({\bf u}_{n})_{n\geq0}$ be a sequence of $ L^p \bigl(0,T;{\bf L}^p(\Omega)\bigr)$ which converges strongly to  $\textbf{u}$. Let us prove that 
\begin{eqnarray*}
{\Lambda}(\textbf{u}_{n}) = \overline{\upsilon}_{n} \rightarrow \overline{\upsilon} = {\Lambda}(\textbf{u}) \quad \mbox{\rm strongly in $L^p\bigl(0,T;{\bf L}^p(\Omega) \bigr)$.}
\end{eqnarray*}
 The sequence $(\overline{\upsilon}_{n})_{n\geq0}$  satisfies  (\ref{esom1}) and (\ref{esom2}) so it  admits  strongly converging subsequences in $ L^p \bigl(0,T;{\bf L}^p(\Omega) \bigr)$. We consider such a subsequence, still denoted   $(\overline{\upsilon}_{n})_{n\geq0}$.   
 For all ${\overline{\varphi}} \in L^{p} \bigl(0,T; V^{p}_{0.div} \bigr)$ we have 
  \begin{eqnarray*}
\displaystyle 
 \int_0^T \Big< \frac{\partial \overline{\upsilon}_n }{\partial t}, {\overline{\varphi}} - \overline{\upsilon}_n  \Big>_{ (V^{p}_{0.div})', V^{p}_{0.div} }    \, dt 
+  \bigl[\mathcal{A}_{{\bf u_n}} (\overline{\upsilon}_n ) , {\overline{\varphi}} - \overline{\upsilon}_n  \bigr]
\displaystyle{
+ J(  {\overline{\varphi}} )-J(\overline{\upsilon}_n)
\geq\,\, \int_{0}^{T} ( \overline{f} , {\overline{\varphi}} - \overline{\upsilon}_n )_{{\bf L}^2(\Omega)} \,dt} 
 \end{eqnarray*}
 and 
  \begin{eqnarray*}
\displaystyle 
 \int_0^T \Big< \frac{\partial \overline{\upsilon} }{\partial t}, {\overline{\varphi}}  - \overline{\upsilon}  \Big>_{ (V^{p}_{0.div})', V^{p}_{0.div} }    \, dt 
+  \bigl[\mathcal{A}_{{\bf u}} (\overline{\upsilon} ) , {\overline{\varphi}} - \overline{\upsilon}  \bigr]
\displaystyle{
+ J(  {\overline{\varphi}} )-J(\overline{\upsilon})
\geq\,\, \int_{0}^{T} ( \overline{f} , {\overline{\varphi}} -\overline{\upsilon}  )_{{\bf L}^2(\Omega)} \,dt} .
 \end{eqnarray*}
 Let us choose    $\overline{\varphi}=\overline{\upsilon} $ as test-function in the first inequality  and   $\overline{\varphi}=\overline{\upsilon}_n $ as test-function in second one.   
By adding the two inequalities we obtain 
 \begin{eqnarray*}\label{3.5.86}
\underbrace{ \int_0^T \Big< \frac{\partial \overline{\upsilon}_n }{\partial t} -  \frac{\partial \overline{\upsilon} }{\partial t} ,  \overline{\upsilon}_n  - {\overline{\upsilon}}  \Big>_{ (V^{p}_{0.div})', V^{p}_{0.div} }    \, dt  }_{= \frac{1}{2}  \Vert\overline{\upsilon}_{n}(T)-\overline{\upsilon}(T)\Vert_{{\bf L}^2(\Omega)}^2 }
+ \bigl[\mathcal{A}_{{\bf u}_{n}} (\overline{\upsilon}_{n})-\mathcal{A}_{{\bf u}} (\overline{\upsilon}) , \overline{\upsilon}_{n}-\overline{\upsilon}  \bigr]\leq 0.
\end{eqnarray*}
Thus 
\begin{eqnarray*}\label{3.5.89}
\displaystyle{\frac{1}{2}  \Vert\overline{\upsilon}_{n}(T)-\overline{\upsilon}(T)\Vert_{\mathbf{L}^2(\Omega)}^2
+ \bigl[\mathcal{A}_{{\bf u}_{n}} (\overline{\upsilon}_{n})-\mathcal{A}_{{\bf u}_{n}} (\overline{\upsilon}) , \overline{\upsilon}_{n}-\overline{\upsilon}  \bigr]}
\displaystyle \leq \bigl[\mathcal{A}_{{\bf u} } (\overline{\upsilon}) -\mathcal{A}_{{\bf u}_{n}} (\overline{\upsilon} ) , \overline{\upsilon}_{n}-\overline{\upsilon} \bigr].
\end{eqnarray*}
The second term of the left hand side may be rewritten as 
\begin{eqnarray*}
\begin{array}{ll}
\displaystyle 
\bigl[ \mathcal{A}_{{\bf u}_{n}} (\overline{\upsilon}_{n}) -\mathcal{A}_{{\bf u}_{n}} (\overline{\upsilon}) , \overline{\upsilon}_{n}-\overline{\upsilon}  \bigr]
\\
= \displaystyle{ \int_{0}^{T} \int_{\Omega}  \Big( {\mathcal F} \bigl( \theta,{\bf u}_{n} + \upsilon_0 \xi , D( \overline{\upsilon}_{n} + \upsilon_0 \xi) \bigr) 
 - {\mathcal F} \bigl( \theta,{\bf u}_{n} + \upsilon_0 \xi , D( \overline{\upsilon} + \upsilon_0 \xi) \bigr)  \Big)
 : D( \overline{\upsilon}_{n}-\overline{\upsilon}  ) \, dx dt} 
 \\ 
 = \displaystyle{ \int_{0}^{T}  \int_{\Omega}  \Big( {\overline{\mathcal{F}}_1 } \bigl( D( \overline{\upsilon}_{n} + \upsilon_0 \xi) \bigr)
  - {\overline{\mathcal{F}}_1 } \bigl( D( \overline{\upsilon} + \upsilon_0 \xi) \bigr) \Big)
 : D( \overline{\upsilon}_{n}-\overline{\upsilon}  ) \, dx dt }
 \\
 \displaystyle{ 
 +  \int_{0}^{T}  \int_{\Omega}  \Big( {\overline{\mathcal{F}}_2 } \bigl( \theta,{\bf u}_{n} + \upsilon_0 \xi , D( \overline{\upsilon}_{n} + \upsilon_0 \xi) \bigr)
 - {\overline{\mathcal{F}}_2 } \bigl( \theta,{\bf u}_{n} + \upsilon_0 \xi , D( \overline{\upsilon} + \upsilon_0 \xi) \bigr) \Big)
 : D( \overline{\upsilon}_{n}-\overline{\upsilon}  ) \, dx dt} 
 \end{array}
\end{eqnarray*}
where 
\begin{eqnarray*}
\displaystyle 
{\overline{\mathcal{F}}_1 } (\lambda_{2})
=  \mu_0 |\lambda_{2}|^{p-2}\lambda_{2}
 \quad \hbox{\rm if }  \lambda_2 \not = 0_{\mathbb{R}^{3\times3}}, 
\quad {\overline{\mathcal{F}}_1 } (\lambda_{2})= 0_{\mathbb{R}^{3\times3}}
 \quad \hbox{\rm otherwise}
\end{eqnarray*}
\begin{eqnarray*}
\left\{\begin{array}{ll}
\displaystyle 
{\overline{\mathcal{F}}_2 } (\lambda_{0}, \lambda_{1},\lambda_{2})
= 2 {\overline \mu} \bigl(\lambda_{0} ,\lambda_{1},|\lambda_{2}| \bigr)|\lambda_{2}|^{p-2}\lambda_{2}
 \quad \hbox{\rm if }  \lambda_2 \not = 0_{\mathbb{R}^{3\times3}}, 
\quad
\\
\displaystyle  {\overline{\mathcal{F}}_2 } (\lambda_{0}, \lambda_{1},\lambda_{2})= 0_{\mathbb{R}^{3\times3}}
 \quad \hbox{\rm otherwise}
 \end{array}
 \right.
\end{eqnarray*}
and $\displaystyle {\overline \mu} = \mu - \frac{\mu_0}{2}$. Since ${\overline \mu}$ satisfies 
\begin{eqnarray*} 
\displaystyle   d\mapsto {\overline \mu} (\cdot, \cdot,d) \quad \mbox{is monotone  increasing   on} 
\  \mathbb{R_{+}},
\end{eqnarray*}
\begin{eqnarray*}
\displaystyle  0<\frac{\mu_{0}}{2} \leq {\overline \mu} (o,e,d) \leq \mu_{1} - \frac{\mu_0}{2} 
\ \hbox{\rm for all} \  (o,e,d)\in \mathbb{R}\times\mathbb{R}^{3}\times\mathbb{R}_{+},
\end{eqnarray*}
we infer with  Lemma 1 in \cite{BDP1}  that $\lambda_2 \mapsto {\overline{\mathcal F}}_2 (\cdot, \cdot, \lambda_2)$ is  monotone in $\mathbb{R}^{3\times3}$. Hence
\begin{eqnarray*}
\int_{0}^{T}  \int_{\Omega}  \Big( {\overline {\mathcal F}}_2 \bigl( \theta,{\bf u}_{n} + \upsilon_0 \xi , D( \overline{\upsilon}_{n} + \upsilon_0 \xi) \bigr) - {\overline {\mathcal F}}_2 \bigl(\theta,{\bf u}_{n} + \upsilon_0 \xi , D( \overline{\upsilon} + \upsilon_0 \xi) \bigr) \Big)
 : D( \overline{\upsilon}_{n}-\overline{\upsilon}  ) \, dx dt \ge 0
 \end{eqnarray*}
 and 
\begin{eqnarray*}
\displaystyle 
\bigl[ \mathcal{A}_{{\bf u}_{n}} (\overline{\upsilon}_{n}) -\mathcal{A}_{{\bf u}_{n}} (\overline{\upsilon}) , \overline{\upsilon}_{n}-\overline{\upsilon}  \bigr]
 \ge
  \displaystyle{ \int_{0}^{T}  \int_{\Omega}  \Big( {\overline{\mathcal{F}}_1 } \bigl( D( \overline{\upsilon}_{n} + \upsilon_0 \xi) \bigr)
  - {\overline{\mathcal{F}}_1 } \bigl( D( \overline{\upsilon} + \upsilon_0 \xi) \bigr) \Big)
 : D( \overline{\upsilon}_{n}-\overline{\upsilon}  ) \, dx dt } .
\end{eqnarray*}

 
But for all $(\lambda, \lambda') \in {\mathbb R}^{3 \times 3} \times {\mathbb R}^{3 \times 3}$ we have
\begin{eqnarray} \label{inegp}
\displaystyle \bigl( |\lambda| + |\lambda'| \bigr)^{2-p}  
\bigl( {\overline{\mathcal{F}}_1 } (\lambda) - {\overline{\mathcal{F}}_1 } (\lambda') \bigr)
 : ( \lambda - \lambda' ) 
\ge \mu_0 (p-1)  |\lambda - \lambda' |^2.
\end{eqnarray}
 Indeed if $\lambda = 0_{\mathbb{R}^{3\times3}}$ and/or $\lambda' =0_{\mathbb{R}^{3\times3}}$ the result is obvious. If $|\lambda| = |\lambda'| \not = 0$ we have
\begin{eqnarray*}
\displaystyle \bigl( |\lambda| + |\lambda'| \bigr)^{2-p} 
\bigl( {\overline{\mathcal{F}}_1 } (\lambda) - {\overline{\mathcal{F}}_1 } (\lambda') \bigr)
 : ( \lambda - \lambda' ) 
= \mu_0  2^{2-p}  |\lambda - \lambda' |^2
\end{eqnarray*}
and the conclusion follows from the inequality $2^{2-p} > p-1$ for all $p \in (1,2)$. Otherwise, if $\lambda \not = 0_{\mathbb{R}^{3\times3}}$, $\lambda' \not = 0_{\mathbb{R}^{3\times3}}$ and $|\lambda| \not= |\lambda'|$, without loss of generality we may assume that $|\lambda| > |\lambda'|$ and
we let
\begin{eqnarray*}
G(\lambda, \lambda') = \frac{\bigl( |\lambda| + |\lambda'| \bigr)^{2-p}  
\bigl( {\overline{\mathcal{F}}_1 } (\lambda) - {\overline{\mathcal{F}}_1 } (\lambda') \bigr)
: ( \lambda - \lambda' ) }{ |\lambda - \lambda' |^2}.
\end{eqnarray*}
With the elementary algebraic computations we get
\begin{eqnarray*}
\begin{array}{ll}
\displaystyle G(\lambda, \lambda') = \frac{\mu_0 }{2} \frac{ \bigl( |\lambda| + |\lambda'| \bigr)^{2-p}}{ |\lambda - \lambda' |^2}
\Bigl( \bigl(   |\lambda|^{p-2} + |\lambda '|^{p-2} \bigr) |\lambda - \lambda' |^2 
 +  \bigl(   |\lambda|^2 - | \lambda' |^2 \bigr)  \bigl(   |\lambda|^{p-2} - |\lambda '|^{p-2} \bigr) \Bigr)
\\
\displaystyle
\ge \mu_0 \frac{ \bigl( |\lambda| + |\lambda'| \bigr)^{2-p}}{2} 
\left( \bigl(   |\lambda|^{p-2} + |\lambda '|^{p-2} \bigr) 
 +  \frac{ \bigl(   |\lambda| + | \lambda' | \bigr)  \bigl(   |\lambda|^{p-2} - |\lambda '|^{p-2} \bigr)}{     |\lambda| - | \lambda' | } \right).
 \end{array}
 \end{eqnarray*}
Let $\displaystyle t = \frac{|\lambda| }{ |\lambda'|} >1$. Thus
 \begin{eqnarray*}
\begin{array}{ll}
\displaystyle G(\lambda, \lambda') \ge \mu_0  \frac{ \bigl( 1+t \bigr)^{2-p}}{2} 
\left( (  1+t^{p-2}) 
 +  \frac{ (1+t)  (t^{p-2} -1 )}{  t-1 } \right) \\
 \displaystyle = \mu_0  \frac{  (1+t)^{2-p} (t^{p-1} -1)}{t-1} 
 \ge \mu_0  \frac{t (1-t^{1-p})}{t-1} = \mu_0 \left( 1 - \frac{t^{2-p} -1}{t-1} \right) .
 \end{array}
 \end{eqnarray*}
But, for all $t >1$ we have $\displaystyle \frac{t^{2-p} -1}{t-1} < 2-p$ and (\ref{inegp}) is satisfied.

Hence 
\begin{eqnarray*}
\displaystyle \Bigl( \bigl( |\lambda| + |\lambda'| \bigr)^p \Bigr)^{\frac{2-p}{2}}
 \Bigl( 
 \bigl( {\overline{\mathcal{F}}_1 } (\lambda) - {\overline{\mathcal{F}}_1 } (\lambda') \bigr)
 : ( \lambda - \lambda' ) \Bigr)^{\frac{p}{2}}
\ge \bigl( \mu_0 (p-1) \bigr)^{\frac{p}{2}}  |\lambda - \lambda' |^p.
\end{eqnarray*}
Since $p>1$ we have also 
\begin{eqnarray*}
 \bigl( |\lambda| + |\lambda'| \bigr)^p \le 2^{p-1}  \bigl( |\lambda|^p + |\lambda'|^p \bigr)
 \end{eqnarray*}
 which yields
 \begin{eqnarray*}
\displaystyle  \bigl( |\lambda|^p + |\lambda'|^p \bigr)^{\frac{2-p}{2}}
 \Bigl( 
 \bigl( {\overline{\mathcal{F}}_1 } (\lambda) - {\overline{\mathcal{F}}_1 } (\lambda') \bigr)
 : ( \lambda - \lambda' ) \Bigr)^{\frac{p}{2}}
\ge \frac{ \bigl( \mu_0 (p-1) \bigr)^{\frac{p}{2}}}{ 2^{ \frac{(p-1)(2-p)}{2} } } |\lambda - \lambda' |^p \ge 0
\end{eqnarray*}
for all $(\lambda, \lambda') \in {\mathbb R}^{3 \times 3} \times {\mathbb R}^{3 \times 3}$. 

\bigskip

By replacing $\lambda = D( \overline{\upsilon}_n + \upsilon_0 \xi)$, $\lambda' = D( \overline{\upsilon} + \upsilon_0 \xi)$ and using H\"older's inequality we obtain
\begin{eqnarray*}
\begin{array}{ll}
\displaystyle \frac{ \bigl( \mu_0 (p-1) \bigr)^{\frac{p}{2}}}{ 2^{ \frac{(p-1)(2-p)}{2} } } \int_0^T \int_{\Omega} \bigl| D( \overline{\upsilon}_n  - \overline{\upsilon}) \bigr|^p \, dx dt 
\\
\displaystyle \le \left(\int_0^T  \int_{\Omega} 
\Bigl( 
{\overline{\mathcal{F}}_1 } \bigl( D( \overline{\upsilon}_n + \upsilon_0 \xi) \bigr)  - {\overline{\mathcal{F}}_1 } \bigl( D( \overline{\upsilon} + \upsilon_0 \xi) \bigr) \Bigr)
 : D(  \overline{\upsilon}_n -  \overline{\upsilon}) \, dx dt \right)^{\frac{p}{2}}
  \\
\displaystyle \times \left( \int_0^T \int_{\Omega} \Bigl(  \bigl| D(  \overline{\upsilon}_n + \upsilon_0 \xi ) \bigr|^p +  \bigl| D(  \overline{\upsilon} + \upsilon_0 \xi ) \bigr|^p \Bigr) \, dx  dt \right)^{\frac{2-p}{2}}
\\
\displaystyle 
\le 
 \left(\int_0^T  \int_{\Omega} 
\Bigl( 
{\overline{\mathcal{F}}_1 } \bigl( D( \overline{\upsilon}_n + \upsilon_0 \xi) \bigr) - {\overline{\mathcal{F}}_1 } \bigl( D ( \overline{\upsilon} + \upsilon_0 \xi) \bigr) \Bigr)
 : D(  \overline{\upsilon}_n -  \overline{\upsilon}) \, dx dt \right)^{\frac{p}{2}}
  \\
\displaystyle \times 
\left(   \|   \overline{\upsilon}_n + \upsilon_0 \xi  \|_{L^p(0,T; V^p_{\Gamma_1})}^p 
+  \|   \overline{\upsilon} + \upsilon_0 \xi  \|_{L^p(0,T; V^p_{\Gamma_1})}^p  \right)^{\frac{2-p}{2}}
\end{array}
\end{eqnarray*}
and with (\ref{esom1})
\begin{eqnarray*}
\begin{array}{ll}
\displaystyle \frac{ \bigl( \mu_0 (p-1) \bigr)^{\frac{p}{2}}}{ 2^{ \frac{(p-1)(2-p)}{2} } }
 \bigl\| D( \overline{\upsilon}_n  - \overline{\upsilon}) \bigr\|_{L^p(0,T; (L^p(\Omega))^{3\times 3})}^p  
\\
\displaystyle
\le \left(\int_0^T  \int_{\Omega} 
\Bigl( 
{\overline{\mathcal{F}}_1 } \bigl( D ( \overline{\upsilon}_n + \upsilon_0 \xi)  \bigr) - {\overline{\mathcal{F}}_1 } \bigl( D ( \overline{\upsilon} + \upsilon_0 \xi) \bigr) \Bigr)
 : D(  \overline{\upsilon}_n -  \overline{\upsilon}) \, dx dt \right)^{\frac{p}{2}}
  \\
\times \Bigl( 2 \bigl(C +  \|   \upsilon_0 \xi  \|_{L^p(0,T; V^p_{\Gamma_1})} \bigr)^p  \Bigr)^{\frac{2-p}{2}}.
\end{array}
\end{eqnarray*}
Hence
\begin{eqnarray*}\label{3.5.95}
 \bigl[ \mathcal{A}_{{\bf u}_{n}} (\overline{\upsilon}_{n}) -\mathcal{A}_{{\bf u}_{n}}(\overline{\upsilon}) , \overline{\upsilon}_{n}-\overline{\upsilon}  \bigr]
\geq 
 \frac{  \mu_0 (p-1)  }{  2^{ (2-p) }  
  \bigl( C + \|  \upsilon_0 \xi  \|_{ L^p (0,T; V^p_{\Gamma_1} ) } \bigr)^{ (2-p)    } }
 \Vert D(\overline{\upsilon}_{n}-\overline{\upsilon})\Vert_{ L^p (0,T; (L^p (\Omega) )^{3\times 3} ) }^2   .
\end{eqnarray*}

On the other hand,
\begin{eqnarray*}
\begin{array}{ll}
\displaystyle \bigl[ \mathcal{A}_{{\bf u} } (\overline{\upsilon}) -\mathcal{A}_{{\bf u}_{n}} (\overline{\upsilon} ) , \overline{\upsilon}_{n}-\overline{\upsilon}  \bigr] \\
\displaystyle 
= \int_{0}^{T} \int_{\Omega} \Big( {\overline {\mathcal F}}_2 \bigl(\theta,{\bf u} + \upsilon_0 \xi , D( \overline{\upsilon} + \upsilon_0 \xi) \bigr) - {\overline {\mathcal F}}_2 \bigl(\theta,{\bf u}_{n} + \upsilon_0 \xi , D( \overline{\upsilon} + \upsilon_0 \xi) \bigr) \Big)
 : D( \overline{\upsilon}_{n}-\overline{\upsilon}  ) \, dx dt \\
 \displaystyle 
 \le \bigl\| {\overline {\mathcal F}}_2 \bigl(\theta,{\bf u} + \upsilon_0 \xi , D( \overline{\upsilon} + \upsilon_0 \xi) \bigr) - {\overline {\mathcal F}}_2 \bigl(\theta,{\bf u}_{n} + \upsilon_0 \xi , D( \overline{\upsilon} + \upsilon_0 \xi) \bigr) \bigr\|_{L^{p'}(0,T; (L^{p'}(\Omega))^{3 \times 3})} 
\\
\displaystyle \times  \Vert D(\overline{\upsilon}_{n}-\overline{\upsilon})\Vert_{L^p(0,T; (L^p(\Omega))^{3 \times 3})}
 \end{array}
 \end{eqnarray*}
and we get
\begin{eqnarray*}
\begin{array}{ll}
\displaystyle \bigl\| {\overline {\mathcal F}}_2 \bigl(\theta,{\bf u} + \upsilon_0 \xi , D( \overline{\upsilon} + \upsilon_0 \xi) \bigr) - {\overline {\mathcal F}}_2\bigl(\theta,{\bf u}_{n} + \upsilon_0 \xi , D( \overline{\upsilon} + \upsilon_0 \xi) \bigr) \bigr\|_{L^{p'}(0,T; (L^{p'}(\Omega))^{3 \times 3})} 
\\
\displaystyle \times  \Vert D(\overline{\upsilon}_{n}-\overline{\upsilon})\Vert_{L^p(0,T; (L^p(\Omega))^{3 \times 3})} \\
 \displaystyle
 \le \frac{1}{4 \epsilon} \bigl\| {\overline {\mathcal F}} \bigl(\theta,{\bf u} + \upsilon_0 \xi , D( \overline{\upsilon} + \upsilon_0 \xi) \bigr) - {\overline {\mathcal F}} \bigl(\theta,{\bf u}_{n} + \upsilon_0 \xi , D( \overline{\upsilon} + \upsilon_0 \xi) \bigr) \bigr\|_{L^{p'}(0,T; (L^{p'}(\Omega))^{3 \times 3})}^{2} 
\\
\displaystyle + \epsilon \Vert D(\overline{\upsilon}_{n}-\overline{\upsilon})\Vert_{L^p(0,T; (L^p(\Omega))^{3 \times 3})}^2
\end{array}
\end{eqnarray*}
where 
we choose 
$\displaystyle \epsilon = \frac{1}{2}  \frac{  \mu_0 (p-1)  }{  2^{  (2-p)  }  
  \bigl( C + \|  \upsilon_0 \xi  \|_{ L^p (0,T; V^p_{\Gamma_1} ) } \bigr)^{ (2-p)   } }$.
We obtain 
 \begin{eqnarray*}\label{3.5.97}
\begin{array}{ll}
\displaystyle{\frac{1}{2}  \Vert\overline{\upsilon}_{n}(T)-\overline{\upsilon}(T)\Vert_{\mathbf{L}^2(\Omega)}^2 
+\epsilon (C_{Korn, p})^2 \Vert \overline{\upsilon}_{n}-\overline{\upsilon} \Vert_{ L^p(0,T; {\bf L}^p(\Omega))}^2 } \\
\displaystyle 
\le \frac{1}{4 \epsilon} \bigl\| {\overline {\mathcal F}}_2 \bigl(\theta,{\bf u} + \upsilon_0 \xi , D( \overline{\upsilon} + \upsilon_0 \xi) \bigr) - {\overline {\mathcal F}}_2 \bigl(\theta,{\bf u}_{n} + \upsilon_0 \xi , D( \overline{\upsilon} + \upsilon_0 \xi) \bigr) \bigr\|_{L^{p'}(0,T; (L^{p'}(\Omega))^{3 \times 3})}^{2} .
\end{array}
\end{eqnarray*}
The sequence  $({\bf u}_{n})_{n \ge 0}$ converges to $ {\bf u}$ strongly in $L^p(0,T;{\bf L}^p(\Omega))$. So, by possibly extracting a subsequence still denoted $({\bf u}_{n})_{n \ge 0}$, we have 
\begin{eqnarray*}
{\bf u}_{n} (t,x) \longrightarrow {\bf u} (t,x) \quad \forall \  {\rm a.a.} \  (x,t) \in (0,T) \times \Omega.
\end{eqnarray*}
By using the continuity and the uniform boundedness of the mapping $\overline{\mu}$, we infer from Lebesgue's theorem that 
\begin{eqnarray*}
{\overline {\mathcal F}}_2 \bigl(\theta,{\bf u}_n + \upsilon_0 \xi , D( \overline{\upsilon} + \upsilon_0 \xi) \bigr) \longrightarrow  {\overline {\mathcal F}}_2  \bigl(\theta,{\bf u} + \upsilon_0 \xi , D( \overline{\upsilon} + \upsilon_0 \xi) \bigr) \quad \hbox{\rm strongly in $L^{p'} \bigl(0,T; (L^{p'}(\Omega))^{3 \times 3} \bigr)$.}
\end{eqnarray*}
%
%
It follows that   
\begin{eqnarray*}
\lim_{n\rightarrow+\infty}\Vert \overline{\upsilon}_{n}-\overline{\upsilon}\Vert_{L^p(0,T;{\bf L}^p(\Omega))} = 0.
\end{eqnarray*}
 So,   any subsequence  $(\overline{\upsilon}_{n})_{n\geq 0}$ which is strongly convergent in $L^p \bigl( 0,T;{\bf L}^p(\Omega) \bigr)$ converges to $\overline{\upsilon}= \Lambda({\bf u}) $. Recalling that the whole sequence $(\overline{\upsilon}_{n})_{n\geq 0}$ is bounded in $L^p \bigl(0,T;{\bf L}^p(\Omega) \bigr)$, we infer that the whole sequence $(\overline{\upsilon}_{n})_{n\geq 0}$ converges to  $\overline{\upsilon}=\Lambda({\bf u})$  which proves the continuity of the mapping  $\Lambda$.
 
 \bigskip
 
 With Schauder's  fixed point theorem, we may conclude that  $\Lambda$  admits a fixed point ${\bf u} \in L^p \bigl(0,T;{\bf L}^p(\Omega) \bigr)$. Let us denote $\Lambda({\bf u})=\textbf{u}= \overline{\upsilon}$. We obtain that $\overline{\upsilon}\in C \bigl(0,T;\textbf{L}^{2}(\Omega) \bigr)\cap L^p \bigl(0,T;V^p_{0.div} \bigr)$  with   $\displaystyle \frac{\partial \overline{\upsilon}}{\partial t} \in L^{p'} \bigl(0,T;(V^p_{0.div})' \bigr)$ and satisfies
 \begin{eqnarray} \label{3.17}
\begin{array}{ll}
\displaystyle 
 \int_0^T \Big< \frac{\partial \overline{\upsilon} }{\partial t}, {\overline{\varphi}}  \Big>_{ (V^{p}_{0.div})', V^{p}_{0.div} }    \, dt 
+  \bigl[\mathcal{A}_{\overline{\upsilon}} (\overline{\upsilon} ) , {\overline{\varphi}}  \bigr]
\medskip\\
\displaystyle{
+ J( \overline{\upsilon} + {\overline{\varphi}} )-J(\overline{\upsilon})
\geq\,\, \int_{0}^{T} \left( f + \frac{\partial \xi}{\partial t} \upsilon_0, {\overline{\varphi}}  \right)_{{\bf L}^2(\Omega)} \,dt} 
 \quad \forall {\overline{\varphi}} \in L^{p} \bigl(0,T; V^{p}_{0.div} \bigr)
 \end{array}
 \end{eqnarray}
 with $\overline{\upsilon} (0) = 0$.

 
 \bigskip
 
In order to conclude the study of problem (P) it remains now to construct the pressure term. As usual the key tool is De Rham's theorem.

\bigskip


Reminding that $\overline{\upsilon} \in C\bigl( [0,T]; {\bf L}^2(\Omega) \bigr)$ we may define $F(t) \in \bigl(V_0^p \bigr)'$ for all $t \in [0,T]$ by
\begin{eqnarray*}
\begin{array}{ll}
\displaystyle
 F(t)(\tilde\vartheta)= \left( \int_{0}^{t} \overline{f}\,d \tilde t  , \tilde \vartheta \right)_{{\bf L}^2(\Omega)}
 - \bigl(\overline{\upsilon}(t),  \tilde \vartheta \bigr)_{{\bf L}^2(\Omega)} 
 \\
 \displaystyle
 - \int_{0}^{t} \int_{\Omega} {\mathcal F} \bigl(\theta, \overline{\upsilon} + \upsilon_0 \xi, D(\overline{\upsilon} + \upsilon_0 \xi ) \bigr): D( \tilde \vartheta) \, dx d\tilde t \quad \forall \tilde \vartheta \in V_0^p.
\end{array}
\end{eqnarray*}
The mapping $F$ belongs to $C\bigl( [0,T]; (V_0^p)' \bigr)$. Moreover for all $\tilde\vartheta\in W^{1,p}_{0.div}$  and for all $t \in [0,T]$ we may consider $\vartheta = \tilde \vartheta {\bf 1}_{| [0,t]} \in L^p \bigl(0,T; V_{0.div}^p \bigr)$ and with $\overline{\varphi}=  \pm\vartheta$ in the variational inequality  (\ref{3.17}) we obtain
\begin{eqnarray*}\label{3.59}
\begin{array}{ll}
\displaystyle{
\int_{0}^{T} \Big< \frac{\partial \overline{\upsilon} }{\partial t}, \vartheta  \Big>_{ (V^{p}_{0.div})', V^{p}_{0.div} }    \, dt 
+\int_{0}^{T} \int_{\Omega} {\mathcal F} \bigl(\theta, \overline{\upsilon} + \upsilon_0 \xi, D(\overline{\upsilon} + \upsilon_0 \xi ) \bigr): D(  \vartheta) 
\, dx dt }
\\
\displaystyle{=\int_{0}^{T} ( \overline{f},\vartheta)_{{\bf L}^2(\Omega) } \,dt}
\end{array}
\end{eqnarray*}
which implies that $F(t) (\tilde \vartheta) =0$.
We infer that, for all $t \in [0,T]$, there exists a unique distribution $\tilde\pi(t)\in L^{p'}_0(\Omega)$  such that
\begin{eqnarray*}\label{con d'}
 F(t)=\nabla  \tilde \pi(t)
\end{eqnarray*}
(see for instance Lemma 2.7 in  \cite{1'}). Since the gradient operator is an endomorphism from $L^{p'}_{0}(\Omega)$   into  $\textbf{W}^{-1,p'}(\Omega)$ (see Corollary 2.5 in \cite{1'}), we obtain that $\nabla \tilde \pi\in C \bigl([0,T];\textbf{W}^{-1,p'}(\Omega) \bigr)$ and  
$\tilde\pi\in C \bigl([0,T];L_{0}^{p'}(\Omega) \bigr)$. Then, for all $t \in [0,T]$ and for all $\tilde \vartheta\in {\bf D} (\Omega)=(\mathcal{D}(\Omega))^3$, we have 
\begin{eqnarray*}
\begin{array}{ll}
\displaystyle  F(t)(\tilde\vartheta)= \left(  \int_{0}^{t} \overline{f}\,d\tilde t, \tilde \vartheta \right)_{{\bf L}^2(\Omega)}
 - \bigl(\overline{\upsilon}(t),  \tilde \vartheta \bigr)_{ {\bf L}^2(\Omega)}
 \\
 \displaystyle -\int_{0}^{t} \int_{\Omega} {\mathcal F} \bigl(\theta, \overline{\upsilon} + \upsilon_0 \xi, D(\overline{\upsilon} + \upsilon_0 \xi ) \bigr): D( \tilde \vartheta) \, dx d\tilde t \\
 \displaystyle = \bigl\langle \nabla  \tilde \pi(t),\tilde\vartheta \bigr\rangle_{ {\bf D}' (\Omega),{\bf  D} (\Omega)}\quad \forall \tilde\vartheta\in {\bf D} (\Omega)  
 \end{array}
 \end{eqnarray*}
 and with  Green's formula 
 \begin{eqnarray*}
 \begin{array}{ll}
\displaystyle  \left( \int_{0}^{t} \overline{f}\, d \tilde t  , \tilde \vartheta \right)_{{\bf L}^2(\Omega)}
 - \bigl(\overline{\upsilon}(t),  \tilde \vartheta \bigr)_{{\bf L}^2(\Omega)}-\int_{0}^{t} \int_{\Omega} {\mathcal F} \bigl(\theta, \overline{\upsilon} + \upsilon_0 \xi, D(\overline{\upsilon} + \upsilon_0 \xi ) \bigr): D( \tilde \vartheta) \, dx d \tilde t  \\
 \displaystyle 
 =- \bigl\langle \tilde \pi(t),div(\tilde \vartheta ) \bigr\rangle_{{\mathcal D}'(\Omega), {\mathcal D} (\Omega)}
 = - \int_{\Omega}  \tilde \pi(t) div(\tilde \vartheta) \, dx.
 \end{array}
\end{eqnarray*} 
The density of  ${\bf D} (\Omega)$ in  $ \mathbf{W}^{1,p}_{0}(\Omega)$ yields the same equality  for all   $\tilde \vartheta \in   \mathbf{W}^{1,p}_{0}(\Omega)$. 
By deriving with respect to $t$ we get
 \begin{eqnarray*} \label{3.63}
 \begin{array}{ll}
\displaystyle  ( \overline{f}, \tilde \vartheta )_{{\bf L}^2(\Omega)}
 - \frac{\partial}{\partial t} \bigl(\overline{\upsilon},  \tilde \vartheta \bigr)_{{\bf L}^2(\Omega)}
 - \int_{\Omega} {\mathcal F} \bigl(\theta, \overline{\upsilon} + \upsilon_0 \xi, D(\overline{\upsilon} + \upsilon_0 \xi ) \bigr): D( \tilde \vartheta) \, dx
\\
\displaystyle  =-  \int_{\Omega}  \frac{\partial \tilde \pi}{\partial t} div(\tilde \vartheta)  \, dx
\quad \hbox{\rm in ${\mathcal D}'(0,T)$}
\end{array}
\end{eqnarray*} 
and we define $\displaystyle \pi = \frac{\partial \tilde \pi}{ \partial t} \in {\mathcal D}' \bigl(0,T; L^{p'}_0 (\Omega) \bigr)$.


\bigskip

Hence, for all $\zeta\in \mathcal{D}(0,T)$ and $\tilde \vartheta\in {\bf D} (\Omega)=(\mathcal{D}(\Omega))^3$, we have
\begin{eqnarray*}\label{3.1.67.000}
\begin{array}{ll}
\displaystyle{ \int_{0}^{T}\frac{ \partial }{\partial t}(\overline{\upsilon},\tilde \vartheta )_{ {\bf L}^2(\Omega)}\zeta \,dt
+\int_{0}^{T}  \int_{\Omega} {\mathcal F} \bigl(\theta, \overline{\upsilon} + \upsilon_0 \xi, D(\overline{\upsilon} + \upsilon_0 \xi ) \bigr): D( \tilde \vartheta \zeta) \, dx dt } \\
\displaystyle{ - \left\langle \int_{\Omega}  \pi div(\tilde \vartheta) \, dx , \zeta \right\rangle_{{\mathcal D}'(0,T), {\mathcal D}(0,T)} }
\displaystyle{ =\int_{0}^{T} (\overline{f},\tilde \vartheta)_{{\bf L}^2(\Omega) } \zeta\,dt.}
\end{array}
\end{eqnarray*}
But, for all $\omega\in L_{0}^{p}(\Omega) $ we know that there exists $\tilde \vartheta \in   \mathbf{W}^{1,p}_{0}(\Omega)$ such that 
\begin{eqnarray*}
div(\tilde \vartheta)=\omega
 \end{eqnarray*}
 and the mapping $P_p : L_{0}^{p}(\Omega) \to \mathbf{W}^{1,p}_{0}(\Omega)$ given by $\tilde \vartheta= P_p (\omega)$ is linear and continuous (see Corollary 3.1 in \cite{1'}). So for all $\omega\in L_{0}^{p}(\Omega) $ we have
\begin{eqnarray*}\label{3.1.67.0000}
\begin{array}{ll}
\displaystyle \left\langle \int_{\Omega}  \pi \omega \, dx, \zeta \right\rangle_{{\mathcal D}'(0,T), {\mathcal D}(0,T)}
= 
- \int_{0}^{T} \bigl(\overline{\upsilon},P_p(\omega) \bigr)_{ {\bf L}^2(\Omega)}\zeta' \,dt 
\\
\displaystyle
+\int_{0}^{T}  \int_{\Omega} {\mathcal F} \bigl(\theta, \overline{\upsilon} + \upsilon_0 \xi, D(\overline{\upsilon} + \upsilon_0 \xi ) \bigr): D \bigl( P_p(\omega) \zeta \bigr) \, dx dt 
-\int_{0}^{T}  ( \overline{f},P_p(\omega) )_{{\bf L}^2(\Omega) }\zeta\, dt
\end{array}
\end{eqnarray*}
and
\begin{eqnarray}\label{3.1.67}
\begin{array}{ll}
\displaystyle \left| \left\langle \int_{\Omega}  \pi \omega \, dx, \zeta \right\rangle_{{\mathcal D}'(0,T), {\mathcal D}(0,T)} \right| 
\le 
\Vert\overline{\upsilon}\Vert_{L^2(0,T;{\bf L}^2(\Omega))}\Vert P_p(\omega) \zeta'\Vert_{L^2(0,T;{\bf L}^2(\Omega))}
\\ 
\displaystyle + 2 \mu_{1} T^{\frac{1}{p}} 
\bigl( 
\Vert\overline{\upsilon}\Vert_{L^p(0,T;V^p_{0.div})}
+  \Vert D(\upsilon_{0} \xi) \Vert_{L^p(0,T; (L^p(\Omega)^{3\times 3}) } ) 
\bigr)^{p-1}
\Vert P_p(\omega) \zeta \Vert_{L^{\infty}(0,T;V^p_{0} )}
\\
\displaystyle + T^{\frac{1}{p}} \Vert\overline{f}\Vert_{L^{p'}(0,T; {\bf L}^2(\Omega) )}
\Vert P_p (\omega)\zeta \Vert_{L^{\infty}(0,T;{\bf L}^2(\Omega) )}.
\end{array}
\end{eqnarray}
By using the continuous embedding of $\mathbf{W}^{1,p}_{0}(\Omega)$ into ${\bf L}^2(\Omega)$ and $H^1_0(0,T; \mathbb{R})$ into $L^{\infty} (0,T; \mathbb{R} )$ we infer from (\ref{3.1.67}) that there exists a positive real number $C^*$, independent of $\omega$ and $\zeta$, such that 
\begin{eqnarray*}
 \displaystyle \left| \left\langle \int_{\Omega}  \pi \omega \, dx, \zeta \right\rangle_{{\mathcal D}'(0,T), {\mathcal D}(0,T)} \right| 
\le  C^* \Vert \omega \zeta \Vert_{H^{1}(0,T;L^p(\Omega) )} \qquad \forall \omega \in L^p_0(\Omega), \ \forall \zeta \in {\mathcal D} (0,T).
\end{eqnarray*}
Then, for all $\omega^*\in L^{p}(\Omega)$, we may apply the previous inequality with 
\begin{eqnarray*}
\omega=\omega^* -\frac{1}{{\rm meas} (\Omega)}\int_{\Omega}\omega^* \,dx  \quad \in L^p_0(\Omega).
\end{eqnarray*}
Since $\pi \in {\mathcal D}' \bigl(0,T; L^{p'}_0 (\Omega) \bigr)$ we have  
\begin{eqnarray*}
\begin{array}{ll}
\displaystyle \left\langle \int_{\Omega}  \pi \omega \, dx, \zeta \right\rangle_{{\mathcal D}'(0,T), {\mathcal D}(0,T)} 
\displaystyle =  \left\langle \int_{\Omega}  \pi \left( \omega^* -\frac{1}{{\rm meas} (\Omega) }\int_{\Omega}\omega^* \,dx  \right) \, dx, \zeta \right\rangle_{{\mathcal D}'(0,T), {\mathcal D}(0,T)}
 \\
\displaystyle
 = \left\langle \int_{\Omega}  \pi \omega^* \, dx, \zeta \right\rangle_{{\mathcal D}'(0,T), {\mathcal D}(0,T)}
 \end{array}
 \end{eqnarray*}
 and  
\begin{eqnarray*} \label{krahtbazf}
\Vert \omega\Vert_{L^p(\Omega)}
\leq 2 \Vert \omega^*\Vert_{L^p(\Omega)} .
\end{eqnarray*}
So, for all $\omega^*\in L ^{p}(\Omega)$ and   $ \zeta\in \mathcal{D}(0,T)$ we obtain
\begin{eqnarray*}
 \displaystyle \left| \left\langle \int_{\Omega}  \pi \omega^* \, dx, \zeta \right\rangle_{{\mathcal D}'(0,T), {\mathcal D}(0,T)} \right| 
\le  2 C^* \Vert \omega^* \zeta \Vert_{H^{1}(0,T;L^p(\Omega) )}.
\end{eqnarray*}
Finally the density of   $\mathcal{D}(0,T)\otimes L^p(\Omega)$ into  $H_{0}^{1}(0,T;L^{p}(\Omega))$     allows us to 
conclude that $\pi\in H^{-1} \bigl(0,T;L_{0}^{p'}(\Omega) \bigr)$.


\bigskip

Now let $\tilde \vartheta \in {\bf D}(\Omega) = \bigl( {\mathcal D}(\Omega) \bigr)^3$ and $\zeta \in {\mathcal D}(0,T)$. 
We have 
\begin{eqnarray*}\label{3.1.67.0}
\begin{array}{ll}
\displaystyle{ \int_0^T ({\overline f}, \tilde \vartheta)_{ {\bf L}^2(\Omega)} \zeta \,dt 
 - \int_{0}^{T} \frac{ \partial }{\partial t}(\overline{\upsilon},\tilde \vartheta )_{ {\bf L}^2(\Omega) }\zeta\,dt}
\\
 \displaystyle{
= \int_{0}^{T}  \int_{\Omega} {\mathcal F} \bigl(\theta, \overline{\upsilon} + \upsilon_0 \xi, D(\overline{\upsilon} + \upsilon_0 \xi ) \bigr): D( \tilde \vartheta \zeta) \, dx dt 
 + \int_0^T \bigl\langle \nabla \pi , \tilde \vartheta \bigr\rangle_{{\bf D}'(\Omega), {\bf D}(\Omega)}  \zeta \, dt }
\end{array}
\end{eqnarray*}
and with Green's formula
\begin{eqnarray*}\label{3.1.67.00}
\begin{array}{ll}
\displaystyle{\int_0^T ({\overline f}, \tilde \vartheta) _{ {\bf L}^2(\Omega)} \zeta\,dt 
 - \int_{0}^{T} \frac{ \partial }{\partial t}(\overline{\upsilon},\tilde \vartheta )_{ {\bf L}^2(\Omega) }\zeta\,dt}
 \displaystyle{
= -  \int_0^T \bigl\langle div (\sigma) , \tilde \vartheta \bigr\rangle_{ {\bf D}'(\Omega), {\bf D}(\Omega)}  \zeta \, dt 
}
\end{array}
\end{eqnarray*}
where $\sigma = {\mathcal F} \bigl(\theta, \overline{\upsilon} + \upsilon_0 \xi, D(\overline{\upsilon} + \upsilon_0 \xi ) \bigr) - \pi {\rm Id}$.
It follows that
\begin{eqnarray*}
\begin{array}{ll}
\displaystyle \left|\int_{0}^{T} \bigl\langle  div(\sigma) ,\tilde \vartheta \bigr\rangle_{ {\bf D}'(\Omega),{\bf D}(\Omega)}\zeta\,dt\right|
  \\
\displaystyle \leq 
(  C_{\infty} T^{\frac{1}{p}}   \Vert \overline{f}\Vert_{  L^{p'}(0,T;{\bf L}^2(\Omega))}+ \Vert\overline{\upsilon}\Vert_{  L^2(0,T;{\bf L}^2(\Omega))} ) 
  \Vert\tilde \vartheta \zeta\Vert_{  H^{1}(0,T;{\bf L}^2(\Omega))}
  \end{array}
\end{eqnarray*}
where  we recall that $C_{\infty}$  denotes  the norm of the continuous injection   of    $H^{1}(0,T; \mathbb{R})$  into $L^{\infty}(0,T ; \mathbb{R})$.
By density of  $\mathcal{D}(0,T)\otimes {\bf D}(\Omega)$ into  $H_{0}^{1} \bigl(0,T;{\bf L}^{2}(\Omega) \bigr)$   we may conclude that   $div(\sigma)\in H^{-1} \bigl(0,T; {\bf L}^{2}(\Omega) \bigr)$ and we have  
\begin{eqnarray*} \label{smai5'0}
\begin{array}{ll}
\displaystyle{\int_{0}^{T} (\overline{f},\tilde \vartheta )_{ {\bf L}^2(\Omega)}\zeta \,dt
- \int_{0}^{T}\frac{ \partial }{\partial t}(\overline{\upsilon},\tilde \vartheta )_{ {\bf L}^2(\Omega)}\zeta \,dt}
\\
\displaystyle{=-\int_{0}^{T}\int_{\Omega} div(\sigma)  \tilde \vartheta \zeta \,dx dt 
\qquad\forall \tilde \vartheta \in {\bf L}^2(\Omega),\  \forall \zeta\in \mathcal{D}(0,T) .}
\end{array}
 \end{eqnarray*}
Next we observe that   $ \sigma\in H^{-1} \bigl(0,T;\widetilde{\mathcal{Y}} \bigr)$ where
\begin{eqnarray*}
\widetilde{ \mathcal{Y} }= \Bigl\{ \tilde \sigma \in \bigl({L}^{2} (\Omega) \bigr)^{3 \times 3} ; \ div(\tilde \sigma) \in {\bf L}^{2}(\Omega) \Bigr\}.
\end{eqnarray*}
It follows that we may use Green's formula and we get 
\begin{eqnarray*}
\displaystyle - \int_{0}^{T}\int_{\Omega} div(\sigma)  \tilde \vartheta \zeta \,dx dt 
\displaystyle = 
\int_{0}^{T}\int_{\Omega} \sigma : \nabla  \tilde \vartheta \zeta \,dx dt - \int_0^T \int_{\partial \Omega} \sum_{i,j=1}^3 \sigma_{ij} n_j \tilde \vartheta_i \zeta \, dY dt
\end{eqnarray*}
for all $\tilde \vartheta \in {\bf W}^{1,2}(\Omega)$ and for all $\zeta \in {\mathcal D}(0,T)$.

\smallskip

Therefore, for all $\tilde \vartheta \in  V_0^2 \subset V^p_{0}$ and for all $\zeta \in {\mathcal D}(0,T)$ we get
 \begin{eqnarray}\label{sbon}
 \begin{array}{ll}
\displaystyle{\int_{0}^{T}\frac{ \partial }{\partial t}(\overline{\upsilon},\tilde \vartheta )_{ {\bf L}^2(\Omega)} \zeta \,dt
+ \int_{0}^{T}  \int_{\Omega} {\mathcal F} \bigl(\theta, \overline{\upsilon} + \upsilon_0 \xi, D(\overline{\upsilon} + \upsilon_0 \xi ) \bigr): D( \tilde \vartheta \zeta) \, dx dt }
\\
\displaystyle{-\int_{0}^{T} \int_{\Omega} \pi div(\tilde \vartheta) \zeta \, dx dt
+ J(\overline{\upsilon}+\tilde \vartheta \zeta) - J(\overline{\upsilon})  = \int_{0}^{T}( \overline{f},\tilde \vartheta )_{{\bf L}^2(\Omega)}\zeta \,dt +  A(\overline{\upsilon},\tilde \vartheta )}
\end{array}
 \end{eqnarray}
  where
 \begin{eqnarray*}
  A(\overline{\upsilon},\tilde \vartheta )= \int_{0}^{T} \int_{\Gamma_0} (\tilde \sigma_{\tau} \cdot \tilde\vartheta)  \zeta \,dx' dt \  +J( \overline{\upsilon}+\tilde \vartheta \zeta)- J(\overline{\upsilon}) .
 \end{eqnarray*}
 But
 \begin{eqnarray*}
 \int_{\partial\Omega} \tilde \vartheta\cdot n\,dY=0 \qquad \forall \tilde \vartheta \in V^2_{0}.
 \end{eqnarray*}
We infer that there exists  $\hat  \vartheta\in \textbf{W}^{1,2}(\Omega)$ satisfying
 \begin{eqnarray*}
 div(\hat \vartheta)=0\quad \mbox{in} \ \Omega, \quad \hat \vartheta=\tilde  \vartheta \quad\mbox{on} \ \partial\Omega
 \end{eqnarray*}
 (see for instance Lemma  3.3 in \cite{1'} or   Chapter 1, Lemma 2.2 in \cite{3'}).
 Hence  $\hat \vartheta\in V^2_{0.div}  \subset V^p_{0.div} $  and with  $\overline{\varphi}=\hat  \vartheta \zeta$ in   (\ref{3.17}) we obtain
\begin{eqnarray*}
\begin{array}{ll}
\displaystyle \int_{0}^{T}\frac{ \partial }{\partial t}(\overline{\upsilon},\hat  \vartheta )_{ {\bf L}^2(\Omega)}\zeta\,dt
+\int_{0}^{T}  \int_{\Omega} {\mathcal F} \bigl(\theta, \overline{\upsilon} + \upsilon_0 \xi, D(\overline{\upsilon} + \upsilon_0 \xi ) \bigr): D( \tilde \vartheta \zeta) \, dx dt \\
\displaystyle 
+ J( \overline{\upsilon}+\hat  \vartheta \zeta)- J(\overline{\upsilon})  
- \int_{0}^{T} (\overline{f},\hat  \vartheta \zeta)_{{\bf L}^2(\Omega) } \,dt \geq 0
\end{array}
 \end{eqnarray*}
and thus
$  A(\overline{\upsilon},\hat  \vartheta )\geq 0$.
 By observing that   $ A(\overline{\upsilon},\hat  \vartheta )= A(\overline{\upsilon},\tilde  \vartheta )$  since $\hat  \vartheta=\tilde  \vartheta$ on $\partial\Omega$  we get
  $ A(\overline{\upsilon},\tilde  \vartheta )\geq 0$
 and   (\ref{sbon}) becomes
  \begin{eqnarray*}
 \begin{array}{ll}
\displaystyle{ \int_{0}^{T} \frac{ \partial }{\partial t} (\overline{\upsilon},\tilde \vartheta )_{ {\bf L}^2(\Omega) } \zeta \,dt
+ \int_{0}^{T}  \int_{\Omega} {\mathcal F} \bigl(\theta, \overline{\upsilon} + \upsilon_0 \xi, D(\overline{\upsilon} + \upsilon_0 \xi ) \bigr): D( \tilde \vartheta \zeta) \, dx dt }
\\
\displaystyle{-\int_{0}^{T} \int_{\Omega} \pi div(\tilde \vartheta) \zeta \, dx dt
+ J(\overline{\upsilon}+\tilde \vartheta \zeta) - J(\overline{\upsilon}) \ge \int_{0}^{T}( \overline{f},\tilde \vartheta )_{{\bf L}^2(\Omega)}\zeta \,dt 
\quad \forall \tilde \vartheta \in V_0^2, \   \forall \zeta \in \mathcal D(0,T).}
\end{array}
 \end{eqnarray*}
 By density of $V_0^2$ into $V_0^p$ the same inequality holds for all $\tilde \vartheta \in V_0^p$ and $\zeta \in {\mathcal D} (0,T)$ and we may conclude that   $ (\overline{\upsilon},\pi )$  is a  solution   of problem (P).

\end{proof}

\bigskip

\begin{remark} If the vicosity $\mu$ does not depend on the velocity, then Theorem \ref{existence_Pu} yields directly the existence and uniqueness of a solution to problem (P).
\end{remark}

\bigskip





\begin{thebibliography}{00} 



\bibitem{1'}{C. Amrouche, V. Girault}.  \emph{ Decomposition of vector spaces and application to the Stokes problem in arbitrary dimension}, {Czechoslovak Mathematical Journal},   Vol.  44,  109-140, 1994. 






 \bibitem{H.B}{H.A. Barnes}. \emph{A review of the slip (wall depletion) of polymer solutions, emulsions and particle suspensions in viscometers: its cause, character, and cure}, Journal of Non-Newtonian Fluid Mechanics,  Vol. 56, 221-251, 1995.







\bibitem{BDP1} {M. Boukrouche, H. Debbiche, L. Paoli}. \emph{Non-isothermal non-Newtonian fluid flow problem with heat convection and Tresca's friction law},  preprint arXiv:2112.07266, 2021.

\bibitem{BDP2} {M. Boukrouche, H. Debbiche, L. Paoli}. \emph{Unsteady non-Newtonian fluid flows with boundary conditions of  friction type : the case of shear thickening fluids}, to appear in  Nonlinear Analysis: Theory, Methods and Applications, Vol. 216, 2022. https://doi.org/10.1016/j.na.2021.112701




 \bibitem{imane3}{M. Boukrouche, I. Boussetouan,  L.  Paoli}.  \emph{Non-isothermal Navier-Stokes system with mixed boundary conditions and friction law: uniqueness and regularity properties}, Nonlinear Analysis: Theory, Methods and Applications, Vol. 102, 168-185, 2014.


 \bibitem{imane4}{M. Boukrouche, I. Boussetouan,  L.  Paoli}.  \emph{Unsteady 3D-Navier-Stokes system with Tresca's friction boundary conditions law},  Quart. Applied Math., Vol. 78(3), 525-543, 2020.
















\bibitem{h'} {G. Duvaut, J.L. Lions}.  \emph{Les in\'equations en m\'ecanique et physique}, Dunod, Paris, 1972.

\bibitem{84}{N.  El  Kissi, J.M. Piau}. \emph{Slip and friction of polymer melt flows}, Rheology Series, Vol. 5,  357-388, 1996.







\bibitem{F0}
 H. Fujita,
\emph{Flow Problems with Unilateral Boundary Conditions.}
 Le\c cons, Coll\`ege de France (1993).

 
  \bibitem{F1} {H. Fujita}.  \emph{A mathematical analysis of motions of viscous incompressible fluid under leak or slip boundary conditions},  Mathematical Fluid Mechanics and Modeling, Vol. 888, 199-216, 1994.

 \bibitem{F2} {H. Fujita, H. Kawarada,  A. Sasamoto}.  \emph { Analytical and numerical approaches to stationary flow problems with leak and slip boundary conditions},
Advances in Numerical Mathematics, Vol. 14,  17-31, 1994.

 \bibitem{3}{H. Fujita, H.   Kawarada}. \emph{ Variational inequalities for the Stokes equation
with boundary conditions of friction type},   Recent Developments in
Domain Decomposition Methods and Flow Problems,   Vol.  11, 15-33, 1998.

 
   \bibitem{F5} {H. Fujita}. \emph{Remarks on the Stokes flows  under slip and leak boundary conditions of friction type}, Topics in  Mathematical Fluid Mechanics, Vol. 10,  73-94,  2002.

  \bibitem{F6} {H. Fujita}.   \emph{A coherent analysis of Stokes flows under boundary conditions of friction type}, Journal of Computational and Applied Mathematics, Vol. 149,  57-69, 2002.


\bibitem{3'} {V. Girault}, {P.A.  Raviart}. \emph{Finite element approximation of the Navier Stokes equations}, {Springer-Verlag, Berlin, Heidelberg,  New York,   1986}. 


\bibitem{H. HERV} {H. Hervet, L. L\'eger}. \emph{Flow with slip at the wall: from simple to complex fluids}, {Comptes Rendus Physique de l'Acad\'emie des Sciences, Paris, Vol.  4, 241-249,  2003}.
 


 



\bibitem{Roux2} {C. Le Roux, A. Tani}.  \emph{Steady flows of incompressible Newtonian fluids with threshold slip  boundary conditions}, Mathematical Analysis in
Fluid and Gas Dynamics,  Vol. 1353, 21-34, 2004.

 \bibitem{Roux1} {C. Le Roux}.  \emph{Steady Stokes flows with threshold slip boundary conditions}, {Mathematical Models and Methods in Applied Sciences, Vol. 15, 1141-1168, 2005}.

 \bibitem{Roux3}{C. Le  Roux,   A. Tani},  \emph{ Steady solutions of the Navier-Stokes equations with threshold slip boundary conditions}, Mathematical Methods in the Applied Sciences, Vol. 30,   595-624, 2007.














\bibitem{Amir}{A.A.  Pahlavan, J.B. Freund}. \emph{Effect of solid properties on slip at a fluid-solid interface}, Physical Review E, Vol. 83,  021602,   2011.

\bibitem{slip} {R. Pit, H. Hervet, L.  L\'eger}.  \emph{Friction and slip of a simple liquid at solid surface}, {Tribology Letters,} { Vol. 7,   147-152, 1999.}






\bibitem{saito1}
 N. Saito, H. Fujita,
\emph{Regularity of solutions to the Stokes equation under a certain nonlinear boundary condition},
 The Navier-Stokes equations, Lecture
Note Pure Appl. Math., Vol.  223, 73-86, 2001.


\bibitem{saito2}  
 N. Saito,
\emph{On the Stokes equations with the leak and slip boundary conditions of friction type: regularity of solutions},
 Publ. RIMS Kyoto Univ., Vol.  40, 345-383, 2004.



\bibitem{Soh}{T. Sochi}. \emph{Slip at fluid-solid interface}, Polymer Reviews,   Vol. 51, 309-340, 2011.

\bibitem{K.T} {K. Takahito}. \emph{On a strong solution of the non-stationary Navier-Stokes equations under slip or leak boundary conditions of friction type}, {Journal of Differential Equations},  Vol. 254,  756-778, 2013.








\bibitem{Korn's}  {L. Wang}.  \emph{On Korn's  inequality}, {Journal of Computational Mathematics, Vol. 21, 321-324, 2003.}



\end{thebibliography}
 \end{document}